\documentclass{article}

\usepackage{hyperref}
\usepackage[left=2.5cm,right=2.5cm,top=2.5cm,bottom=2.5cm]{geometry}
\usepackage{graphicx}
\usepackage{amssymb,amsmath,amscd,amsthm}
\usepackage{mathtools, color,xcolor}
\usepackage{cleveref}
\usepackage{multirow}
\usepackage{soul}

\usepackage{enumerate,array}

\newcommand{\N}{\mathbb{N}}

\newcommand{\R}{\mathbb{R}}
\newcommand{\C}{\mathbb{C}}

\newcommand{\mS}{\mathbb{S}}

\newcommand{\mBlue}[1]{{\color{blue} \ensuremath{\mathbf{#1}}}}
\newcommand{\mRed}[1]{{\color{red} \ensuremath{\mathbf{#1}}}}
\newcommand{\Fref}[1]{\Cref{#1}}
\newcommand{\fref}[1]{\cref{#1}}
\newcommand{\aNorm}[1]{\mid\!\mid\!#1\!\mid\!\mid}

\newcommand{\myHom}[2]{\mathbf{H_{#1}}\!\left[#2\right]}
\newcommand{\myEtoP}[3]{\mathbf{Q}_{#1}^{#2}\!\left[#3\right]}

\newtheorem{prop}{Proposition}[section]

\newtheorem{rem}[prop]{Remark}

\newcommand{\lp}{\left}
\newcommand{\rp}{\right}

\crefformat{equation}{(#2#1#3)}

\title{Cascades of  Global Bifurcations and Chaos\\ near a  Homoclinic Flip Bifurcation: A
  Case Study}
\author{Andrus Giraldo\footnotemark[2] , Bernd
  Krauskopf\footnotemark[2] , and Hinke M. Osinga\footnotemark[2]}

\begin{document}
\maketitle

\renewcommand{\thefootnote}{\fnsymbol{footnote}}
\footnotetext[2]{Department of Mathematics, The University of
  Auckland, Private Bag 92019, Auckland 1142, New Zealand
  (\href{mailto:a.giraldo@auckland.ac.nz)}{a.giraldo@auckland.ac.nz},
  \href{mailto:b.krauskopf@auckland.ac.nz)}{b.krauskopf@auckland.ac.nz},
  \href{mailto:h.m.osinga@auckland.ac.nz)}{h.m.osinga@auckland.ac.nz})}

\renewcommand{\thefootnote}{\arabic{footnote}}

\begin{abstract}
We study a specific homoclinic bifurcation called a homoclinic flip bifurcation of case~\textbf{C}, where a homoclinic orbit to a saddle equilibrium with real eigenvalues changes from being orientable to nonorientable. This bifurcation is of codimension two, meaning that it can be found as a bifurcation point on a curve of homoclinic bifurcations in a suitable two-parameter plane. In fact, this is the lowest codimension for a homoclinic bifurcation of a real saddle to generates chaotic behavior in the form of (suspended) Smale horseshoes and strange attractors. We present a detailed numerical case study of how global stable and unstable manifolds of the saddle equilibrium and of bifurcating periodic orbits interact close to a homoclinic flip bifurcation of case~\textbf{C}. This is a step forward  in the understanding of the generic cases of homoclinic flip bifurcations, which started with the study of the simpler cases \textbf{A} and \textbf{B}. In a three-dimensional vector field due to Sandstede, we compute relevant bifurcation curves in the two-parameter bifurcation diagram near the central codimension-two bifurcation in unprecedented detail. We present representative images of invariant manifolds, computed with a boundary value problem setup, both in phase space and as intersection sets with a suitable sphere. In this way, we are able to identify infinitely many cascades of homoclinic bifurcations that accumulate on specific codimension-one heteroclinic bifurcations between an equilibrium and various saddle periodic orbits. Our computations confirm what is known from theory but also show the existence of  bifurcation phenomena that were not considered before. Specifically, we identify the boundaries of the Smale--horseshoe region in the parameter plane, one of which creates  a strange attractor that resembles the R\"{o}ssler attractor. The computation of a winding number reveals a complicated overall bifurcation structure in the wider parameter plane that involves infinitely many further homoclinic flip bifurcations associated with so-called homoclinic bubbles.
\end{abstract}


\section{Introduction}
\label{sec:Intro}

Homoclinic bifurcations lie at the heart of complicated dynamics in smooth vector fields.  Apart from being interesting objects of study in their own right, homoclinic bifurcations appear in many applications, such as mathematical biology \cite{Kuz2,Lina1}, laser physics \cite{Wie1,Wie2} and electronic engineering \cite{Kopper1}; more on their relevance for applications can be found, for example, in \cite{Guck1}.  Indeed, the study of homoclinic bifurcations is important as a first step in the pursuit of understanding complex behavior that arises in mathematical models of different physical phenomena.  Since the work of Shilnikov \cite{Shil5} on homoclinic bifurcations to a saddle focus (with complex eigenvalues), which is now called a \emph{Shilnikov bifurcation}, a lot of attention has focused on the consequences of the existence of homoclinic bifurcations. Famously, the Shilnikov bifurcation, subject to certain eigenvalue conditions, is the only homoclinic bifurcation of codimension one that generates chaotic behavior in the form of (suspended) Smale--horseshoe dynamics; see \cite{san3} for a comprehensive review. This contrasts the situation of a codimension-one homoclinic bifurcation of a real saddle (with only real eigenvalues), which does not generate any nearby chaotic behavior. However, near the codimension-two homoclinic bifurcation of the real saddle, known as the \emph{homoclinic flip} bifurcation, (suspended) Smale--horseshoes \cite{Hom1} and strange attractors \cite{Nau1,Nau2} arise when certain eigenvalue and geometric conditions are satisfied. More generally, homoclinic flip bifurcations can be of different cases (discussed below), and these cases act as organizing centers for the creation of periodic orbits and multi-pulse solutions \cite{Hom1}. Specifically, the existence of homoclinic flip bifurcations explains the creation of spiking behavior in mathematical models of neurons such as the Hindmarsh--Rose system \cite{Lina1}.  

This paper is about the homoclinic flip bifurcation of case~\textbf{C}, which is the most complicated case that features chaotic dynamics.  We consider here the defining situation of a three-dimensional vector field with a real saddle equilibrium. Without loss of generality, we assume that this equilibrium has two different real stable eigenvalues, and one unstable eigenvalue. Generically, at the moment of a codimension-one homoclinic bifurcation of a real saddle equilibrium, the two-dimensional stable manifold forms either a topological cylinder or a M\"obius band as it is followed locally around the homoclinic orbit. The homoclinic flip bifurcation is characterized as the moment when this stable manifold changes from orientable to nonorientable \cite{san3,kis1,Shil2,Yan1}.  A homoclinic flip bifurcation has different unfoldings depending on the eigenvalues of the saddle equilibrium.  In general, three generic cases have been identified, denoted \textbf{A}, \textbf{B} and \textbf{C} \cite{san3}.  In case \textbf{A} a single attracting (or repelling) periodic orbit is created. The unfolding of case \textbf{B} involves saddle periodic orbits, and its unfolding features the main branch of homoclinic bifurcation, as well as only a single curve of saddle-node bifurcation,  a single curve of period-doubling bifurcation and a single curve of homoclinic bifurcation of a period-doubled orbit \cite{And1,san3,kis1}. Finally, the unfolding of case \textbf{C} gives rise to period-doubling cascades, $n$-homoclinic orbits, for any $n \in \N$, a region of Smale--horseshoe dynamics \cite{Deng1,Hom1,san4}, and strange attractors \cite{Nau1,Nau2}. Importantly, case \textbf{C} constitutes the homoclinic bifurcation of a real saddle with the lowest codimension that generates chaotic dynamics.  The unfoldings for the three cases have been studied theoretically with different techniques including return maps \cite{Deng1, Hom1}, Shilnikov variables \cite{kis1} and Lin's method \cite{san4}. The unfoldings for both cases $\mathbf{A}$ and $\mathbf{B}$ have been determined and proven for any vector field of dimension $n \geq 3$. The exact nature of the unfolding of case \textbf{C} is not as well understood as that for cases \textbf{A} and \textbf{B}; in particular, the complete unfolding is not known because infinitely many saddle periodic orbits are created and the interactions of their respective stable and unstable manifolds give rise to many other global bifurcations. This complexity is associated with the existence of Smale--horseshoes dynamics, which were conjectured to be part of the unfolding of a homoclinic flip bifurcation in \cite{Deng1} and subsequently proven in \cite{Hom1} to exist under specific eigenvalue and geometric conditions that define case~\textbf{C}. More specifically, cascades of period-doubling and homoclinic bifurcations are featured in the unfolding to explain the annihilation process of the infinitely many saddle periodic orbits that exist in the Smale--horseshoe region \cite{Hom1}. However, the literature does not yet provide an understanding of the boundaries of the Smale--horseshoe region, the nature of additional bifurcations due to interactions of manifolds, or the implications of such interactions for the reorganization of basin of attraction in phase space.  Moreover, all results concerning the dynamics and unfolding of case \textbf{C} have only been proven for three-dimensional vector fields. To conclude this discussion, in contrast to the other two cases, the understanding of case~\textbf{C} is not as complete.

The existing theoretical results concerning the unfoldings of the homoclinic flip bifurcation of case~\textbf{C} hold in a tubular neighborhood of the homoclinic orbit; they were obtained with well-known techniques by studying the family of local diffeomorphism defined on a Poincar\'e section.  In this paper we adopt a complementary point of view. Thanks to advances in two-point boundary value problem (2PBVP) continuation techniques \cite{Doe5}, we compute relevant invariant manifolds \cite{Kra2} and their interactions in a three-dimensional vector field introduced by Sandstede \cite{san1}, which exhibits a homoclinic flip bifurcation of case~\textbf{C}. In this way, we are able to present the overall dynamics near this bifurcation, beyond considering intersection sets on a Poincar\'e section.  More specifically, our focus is on the role of global invariant manifolds of saddle equilibria and saddle periodic orbits in the organization of the three-dimensional phase space. These objects and surfaces rearrange themselves in sequences of global bifurcations that arise from the codimension-two homoclinic flip bifurcation point of case~\textbf{C}. Our illustrations go well beyond the sketches of manifolds in earlier works (see, for example,~\cite{Deng1,Hom2,san3}) in that we show these objects beyond a tubular neighborhood of the homoclinic orbit. Furthermore, to enhance the illustrated geometric properties of these objects, all our illustrations of the three-dimensional phase space are accompanied by animations. By studying the interactions of these manifolds in phase space, we find additional bifurcation phenomena in the specific case study of Sandstede's model, which have been conjectured or not even considered before as part of the known literature of homoclinic flip bifurcation. In this way, our results contribute to the goal of finding the unfolding of the homoclinic flip bifurcation of case~\textbf{C}, by presenting a set of relevant bifurcation phenomena that are involved.

As in previous work for cases~\textbf{A} \cite{Agu1} and \textbf{B} \cite{And1}, we consider the three-dimensional vector field
\begin{equation}
X^s(x,y,z):
\begin{cases} 
\dot x = P^1(x,y,z) := ax + by -ax^2+(\tilde{\mu}-\alpha z)x(2-3x)+\delta z, \\
\dot y = P^2(x,y,z) := bx +ay -\frac32 bx^2-\frac32 axy-2y(\tilde{\mu}-\alpha
z)-\delta z, \\
\dot z = P^3(x,y,z) := cz +\mu x +\gamma xz +\alpha \beta  (x^2(1-x)-y^2),
\end{cases}
\label{eq:san}
\end{equation} 
which was introduced by Sandstede in \cite{san1} as a model containing homoclinic flip bifurcations of all three different cases. We choose parameters as discussed in \cref{sec:san}, such that the equilibrium located at the origin $\textbf{0} \in \R^3$ is a saddle with two different real negative (stable) and one positive (unstable) eigenvalues $\lambda^{ss}<\lambda^s<0<\lambda^u$. Then the origin has a two-dimensional stable manifold $W^s(\textbf{0})$ and a one-dimensional unstable manifold $W^u(\textbf{0})$, which consist of the points in phase space that converge to $\textbf{0}$ forward and backward in time, respectively.  In particular, system~\cref{eq:san} is at a homoclinic bifurcation when $W^s(\mathbf{0})\cap W^u(\mathbf{0}) \neq \emptyset$, that is, at the moment when there exists a trajectory that converges both to $\mathbf{0}$ forward and backward in time. System~\cref{eq:san} exhibits a homoclinic bifurcation for $\mu=0$ \cite{san1}, which is a homoclinic flip bifurcation of case~\textbf{C} when $\alpha \approx 0.3694818$. We denote this point by $\mathbf{C_I}$ throughout this paper, and we study the bifurcation diagram of system~\cref{eq:san} in the $(\alpha,\mu)$-plane near $\mathbf{C_I}$.

\subsection{The codimension-two homoclinic flip orbit}

\begin{figure}
\centering
\includegraphics{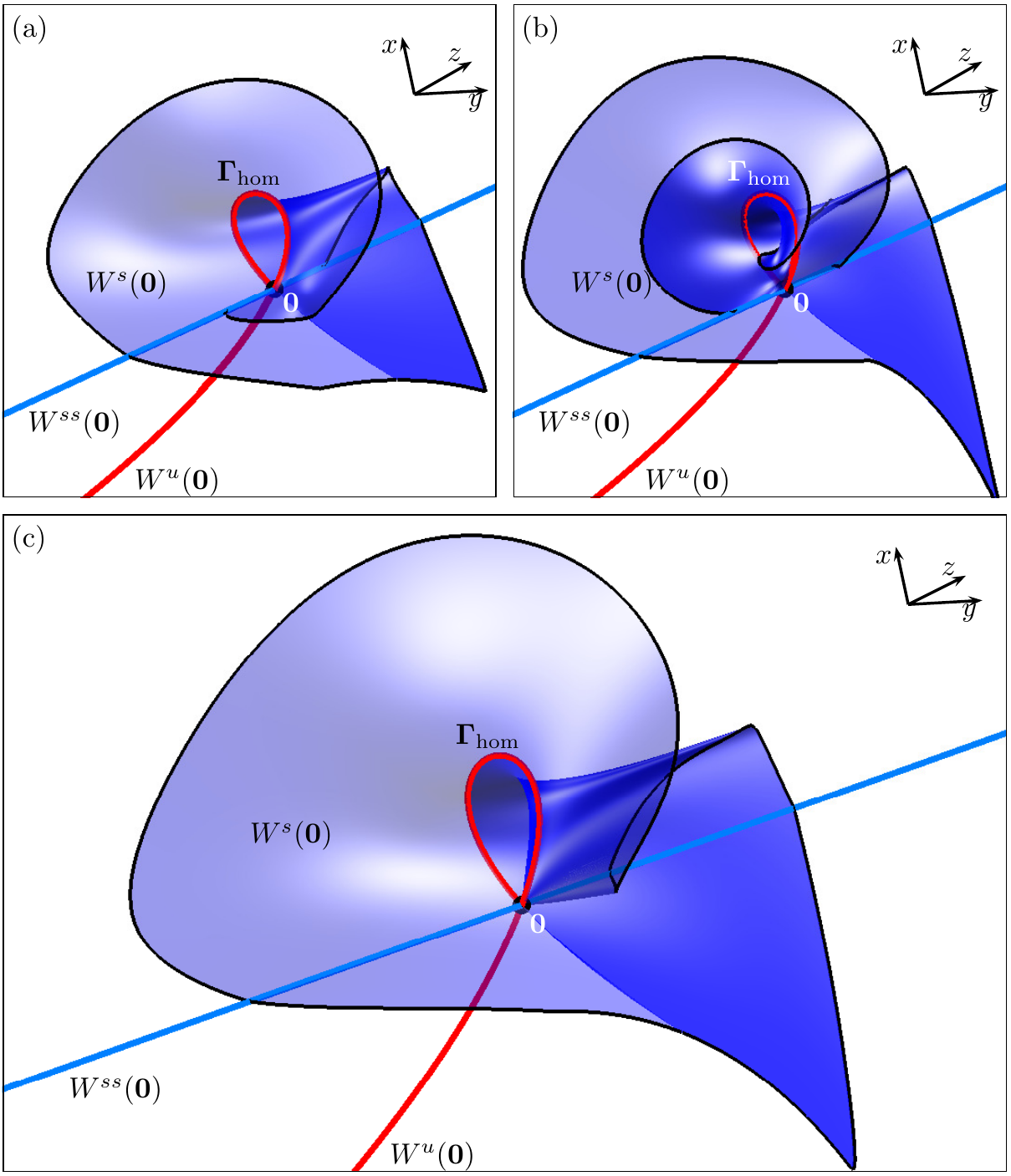}
\caption{The stable manifold $W^{s}(\mathbf{0})$ of system \cref{eq:san} in $\R^3$ at a codimension-one orientable and nonorientable homoclinic bifurcation in panels~(a) and (b), respectively; panel (c) shows $W^{s}(\mathbf{0})$ at a codimension-two inclination flip bifurcation.  Here, $W^{s}(\mathbf{0})$ is rendered transparent with one half colored dark blue and the other half light blue. The unstable manifold $W^u(\mathbf{0})$ is colored red and the homoclinic orbit is labeled $\Gamma_{\rm hom}$. The one-dimensional strong stable manifold $W^{ss}(\mathbf{0})$ is highlighted by a blue curve. The common parameter values for all three panels are $(a,b,c,\beta,\gamma,\mu,\tilde{\mu})= (0.7,1,-2,1,2,0,0)$; furthermore, $\alpha=0.2$ in panel (a), $\alpha=0.5$ in panel~(b) and $\alpha \approx 0.3694818$ in panel~(c).  See also the accompanying animation ({\color{red} GKO\_Cflip\_animatedFig1.gif}).} \label{fig:First}
\end{figure}

\Fref{fig:First} shows phase portraits of system~\cref{eq:san} at three different points along the homoclinic bifurcation curve given by $\mu=0$ in the $(\alpha,\mu)$-plane. The individual panels show the homoclinic orbit together with the associated stable and unstable manifolds of $\textbf{0}$, where the stable manifold $W^s(\mathbf{0})$ is rendered in two shades of blue to illustrate its orientability; here, every trajectory that forms $W^s(\mathbf{0})$ has arclength four. The homoclinic orbit $\mathbf{\Gamma_{\rm hom}}$ is the branch of $W^u(\mathbf{0})$ (red curves) that lies in $W^s(\mathbf{0})$. The homoclinic bifurcations in panels~(a) and (b) in \fref{fig:First} are of codimension one, because they fullfil the following genericity conditions \cite{Hom1,kis1}:
\begin{itemize}
\item[(\textbf{G1})] 
(Non-resonance) $|\lambda^s| \not = \lambda^u$;
\item[(\textbf{G2})] 
(Principal homoclinic orbit) In positive time the homoclinic trajectory approaches the equilibrium tangent to the eigenvector $e^s$ associated with $\lambda^s$;
\item[(\textbf{G3})] 
(Strong inclination) The tangent space ${\rm{T}}W^s(\mathbf{0})$ of the stable manifold $W^s(\mathbf{0})$, when followed along $\mathbf{\Gamma_{\rm hom}}$ backward in time, converges to the plane spanned by the eigenvectors associated with $\lambda^{ss}$ and $\lambda^{u}$.
\end{itemize}
These three properties ensure that a portion of $W^s(\mathbf{0})$ close to $\mathbf{\Gamma}_{\rm hom}$ folds over and closes along a trajectory of $W^s(\mathbf{0})$ tangent to the strongest stable eigenvector of $\mathbf{0}$, that is, along the strong stable manifold $W^{ss}(\mathbf{0})$ (blue curve); hence, locally near $\mathbf{\Gamma_{\rm hom}}$, the stable manifold $W^s(\mathbf{0})$ is topologically equivalent to a cylinder, in \cref{fig:First}~(a), or a M\"obius band, in \cref{fig:First}~(b); see \cite{Agu1} for more details. Recall that, for $\mu=0$, system~\cref{eq:san} always exhibits a homoclinic bifurcation; hence, it seems that one can transition between panels~(a) and (b) by varying $\alpha$ continuously. However, a cylinder and a M\"obius band are not homeomorphic surfaces, so there must exist a transition point, where $W^s(\mathbf{0})$ does not close in either of the two ways depicted in panels~(a) and (b). This transition case is shown in panel~(c) of \cref{fig:First} for $\alpha \approx 0.3694818$, where a homoclinic flip bifurcation takes place, which is called an \emph{inclination flip} (\textbf{IF}). At this bifurcation point genericity condition \textbf{(G3)} is not fulfilled and $W^s(\mathbf{0})$ does not close along $W^{ss}(\mathbf{0})$; see \cref{fig:First}~(c). It was proven in \cite{Deng1,Hom1,kis1} that the inclination flip bifurcation is one mechanism by which one can transition between the orientable and nonorientable case of a homoclinic bifurcation of codimension one. It is not the only known mechanism: a transition in orientability of the homoclinic bifurcation may also occur via a so called \emph{orbit flip} bifurcation, which corresponds to a homoclinic bifurcation that does not fulfill condition \textbf{(G2)}. Both inclination and orbit flip are codimension-two homoclinic flip bifurcations whose unfolding can be of cases~\textbf{A}, \textbf{B} and \textbf{C}. 

For the purpose of this paper, we only consider the inclination flip bifurcation of case \textbf{C}, which occurs if one of the following eigenvalue conditions are satisfied \cite{san3}: 
\begin{equation}\label{eq:cond}
|\lambda^{ss}| < \lambda^u \quad \text{or} \quad 2|\lambda^{s}| < \lambda^u,
\end{equation}
and the following geometric conditions are met:
\begin{enumerate}[(i)]
\item 
$|\lambda^s| \neq |\lambda^{ss}|/2$
\item 
If $|\lambda^s| < |\lambda^{ss}|/2$, the homoclinic orbit converges to $\mathbf{0}$ tangent to $e^s$ in a typical way. Geometrically, this means that the homoclinic orbit is not contained in the one-dimensional leading (weak) stable manifold tangent to $\mathbf{0}$, which exists uniquely as a smooth manifold under these eigenvalue conditions; for further details see \cite{san3}.
\item 
If $|\lambda^s| > |\lambda^{ss}|/2$, the invariant manifold of the homoclinic orbit, which is tangent to the span$\{e^{s},e^u\}$ backward in time, has a nondegenerate quadratic tangency with $W^s(\mathbf{0})$ along the homoclinic orbit.
\end{enumerate}

\subsection{Bifurcation diagram of case~\textbf{C}}

The literature distinguishes between two different unfoldings of a homoclinic flip bifurcation of case~\textbf{C} \cite{Hom1}. These are given by global conditions regarding the geometry of the stable manifold $W^s(\mathbf{0})$ during the creation of the Smale--horseshoe, and they are called an outward twist $\mathbf{C_{\rm out}}$ and an inward twist $\mathbf{C_{\rm in}}$ in the literature \cite{Deng1, Hom1}. Both twist cases contain the same codimension-one bifurcation curves but their relative positions differ. In particular, we find that the homoclinic flip bifurcation point $\mathbf{C_I}$ found for system~\cref{eq:san} corresponds to the outward twist case.

\begin{figure}
\centering
\includegraphics[height=6cm]{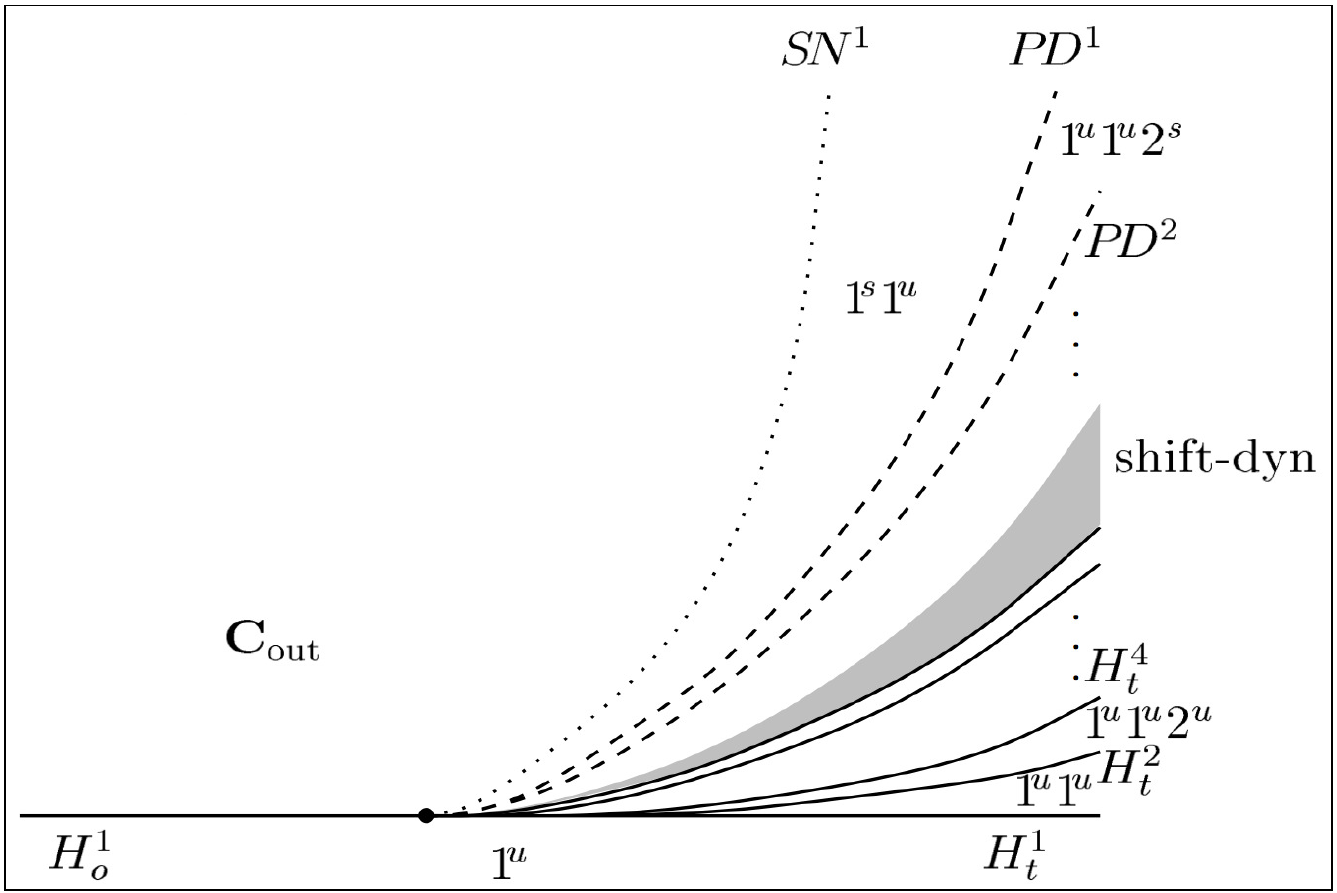}
\caption{Sketch of the theoretical unfolding of an outward twist point~$\mathbf{C_{out}}$ of case~$\mathbf{C}$ from \cite[Fig. 5]{Hom2}. [Reproduced from Journal of Dynamics and Differential Equations, Resonant Homoclinic Flip Bifurcations, 12(4), 2000, pages 807-850, A. J. Homburg and B. Krauskopf, \textcopyright Plenum Publishing Corporation 2000 with permission of Springer.]} 
\label{fig:TheoSketch} 
\end{figure}

\Cref{fig:TheoSketch} shows a sketch of what is known theoretically about the unfolding of a homoclinic flip bifurcation of case~\textbf{C} for an outward twist $\mathbf{C_{out}}$ \cite{Hom1, Hom2}. The unfolding consists of an orientable homoclinic bifurcation curve $H_o^1$ that becomes nonorientable (twisted) $H_t^1$ after transitioning through the codimension-two flip bifurcation point $\mathbf{C_{out}}$. Furthermore, the following bifurcation curves emanate from $\mathbf{C_{out}}$: a saddle-node bifurcation $SN^1$ of periodic orbits, an infinite sequence of period-doubling bifurcations labeled $PD^n$, $n=1,2,4,8,...$, and an infinite sequence of $n$-homoclinic bifurcations labeled $H_t^n$, $n=2,4,8, ...$, where $t$ indicates that these are codimension-one nonorientable (twisted) homoclinic bifurcations. The value of $n$ indicates the number of loops made by the homoclinic orbit before it connects back to the equilibrium point.  As each homoclinic bifurcation occurs, a saddle periodic orbit with $n$ loops is created. The Smale--horseshoe region lies in between the period-doubling and homoclinic cascades; however, the boundaries of this region are not identified in the sketch. Note that the orientability of the saddle periodic orbits, labeled $1^u$ and $2^u$, is not explicitly stated in \cref{fig:TheoSketch}, meaning that labels are not unique to each saddle periodic orbit. 

To explain the overall structure of the unfolding, we start with the orientable saddle periodic orbit $1^u$ that bifurcates from the orientable homoclinic bifurcation $H_o^1$. This same label $1^u$ is also used for the nonorientable saddle periodic orbit that bifurcates from $H_t^1$. Note in the sketch that the orientable and nonorientable saddle periodic orbits persist throughout the cascades of period-doubling and homoclinic bifurcations, and they disappear in the saddle-node bifurcation $SN^1$, after the nonorientable saddle periodic orbit becomes an attracting periodic orbit, labeled $1^s$, in the period-doubling bifurcation $PD^1$. Even though it is not illustrated in \cref{fig:TheoSketch}, the unfolding is known to have infinitely many saddle-node bifurcation curves that also emanate from $\mathbf{C_{out}}$ and occur along the period-doubling cascade \cite{san3}.

\begin{figure}
\centering
\includegraphics{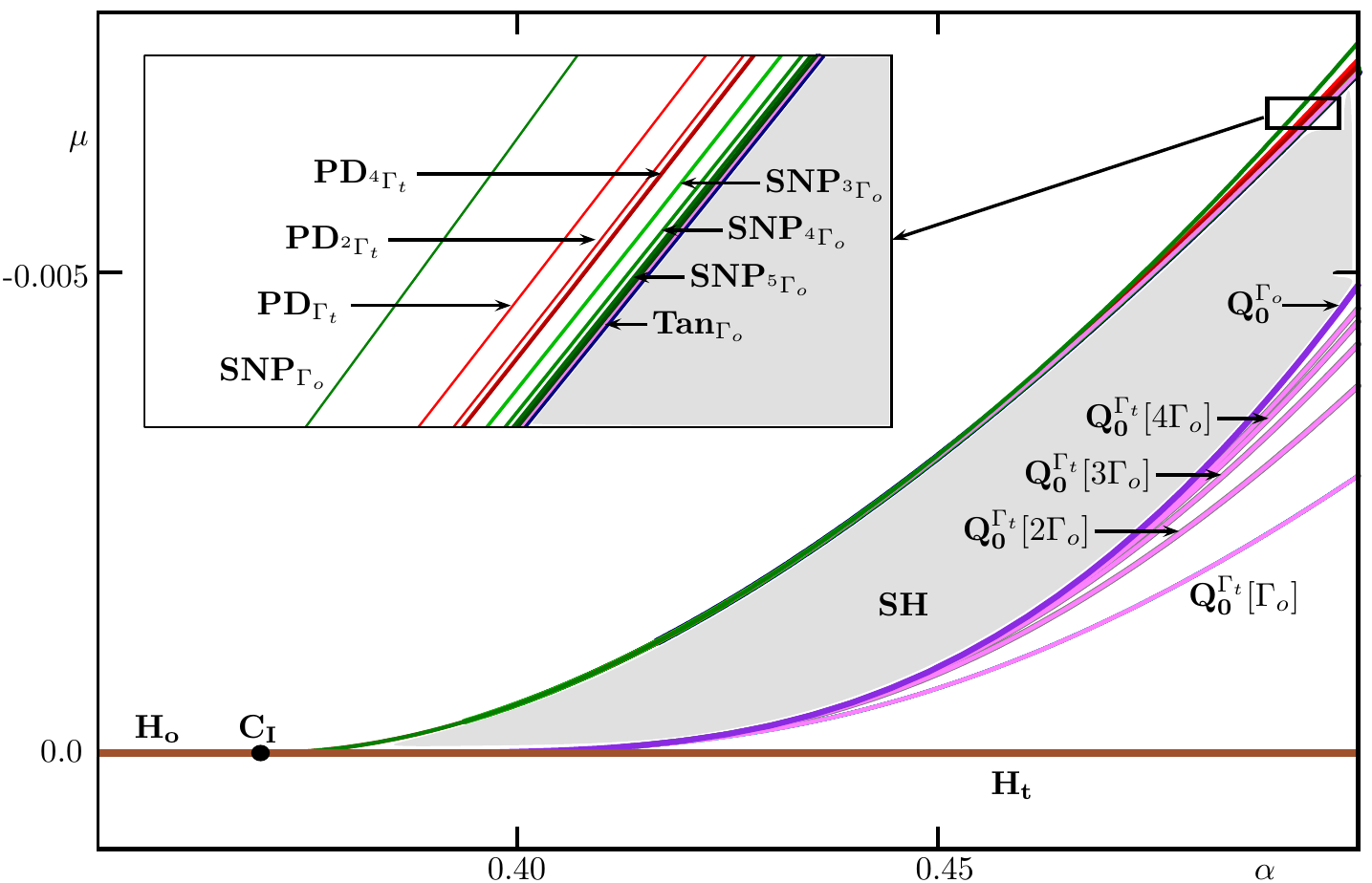}
\caption{Bifurcation diagram in the $(\alpha,\mu)$-plane near an homoclinic flip bifurcation $\mathbf{C_I}$ of system~\cref{eq:san}. The inset shows an enlargement of the indicated area.  Shown are the curves of homoclinic bifurcation~$\mathbf{H_o}$ and $\mathbf{H_t}$ (brown), saddle-node bifurcations $\mathbf{SNP}$ of periodic orbits (dark-green), period-doubling bifurcations $\mathbf{PD}$ (red), heteroclinic bifurcations $\mathbf{Q}_{\mathbf{0}}^{\Gamma_t}$ from $\mathbf{0}$ to $\Gamma_t$ (magenta), heteroclinic bifurcation $\mathbf{Q}_{\mathbf{0}}^{\Gamma_o}$ from $\mathbf{0}$ to $\Gamma_o$ (purple) and codimension-one homoclinic bifurcation bifurcation $\mathbf{Tan}_{\Gamma_o}$ of $\Gamma_o$ (violet). The Smale--horseshoe region $\mathbf{SH}$ is the gray region. The other parameter values are $(a,b,c,\beta,\gamma,\tilde{\mu},\delta)= (0.7,1,-2,1,2,0,0)$.} 
\label{fig:NotTopInv} 
\end{figure}  

Most of the literature on case~\textbf{C} does not delve into details about bifurcations related to interactions of manifolds of saddle periodic orbits. This is also the case in \cref{fig:TheoSketch} where the bifurcation diagram only shows local bifurcations of saddle periodic orbits and homoclinic bifurcations of the real saddle. For ease of exposition and in contrast to \cref{fig:TheoSketch}, we label the orientable saddle periodic orbit that bifurcates from $H_o^1$ as $\Gamma_o$, and the nonorientable saddle periodic orbit that bifurcates from $H_t^1$ as $\Gamma_t$ throughout this paper. 

\subsection{Computing the bifurcation diagram near $\mathbf{C_I}$}

By computing the interaction of the manifolds of these saddle periodic orbits with other objects, we are able to present what is known about the unfolding of case~\textbf{C} in a new light, namely, as part of the bifurcation diagram in the $(\alpha,\mu)$-plane of system~\cref{eq:san}.  Throughout this paper, we center our attention on the saddle periodic orbits $\Gamma_o$ and $\Gamma_t$, and their corresponding global manifolds.  As a starting point, we show in \cref{fig:NotTopInv} the bifurcation diagram for system~\cref{eq:san} in the $(\alpha, \mu)$-plane near the homoclinic flip point $\mathbf{C_I}$ of case~\textbf{C}. Note that we changed the orientation of the $\mu$-axis from top to bottom in the bifurcation diagram for ease of comparison. \Cref{fig:NotTopInv} shows the principal homoclinic branch (brown curve) together with representative period-doubling $\mathbf{PD}$ (red curves) and saddle-node $\mathbf{SNP}$ (green curves) bifurcations of periodic orbits. To exemplify the relevance of the global bifurcations concerning saddle periodic orbits, we also compute representative codimension-one heteroclinic and homoclinic bifurcations in system~\cref{eq:san} of the saddle periodic orbits $\Gamma_o$ and $\Gamma_t$, labeled $\mathbf{C}_{\mathbf{0}}^{\Gamma_o}$ and $\myEtoP{\mathbf{0}}{\Gamma_t}{m\Gamma_o}$ for $m=1,2,3,4$. At first glance, the period-doubling and saddle-node bifurcation curves are indiscernible near the inclination flip point $\mathbf{C_I}$ in \cref{fig:NotTopInv}. To visualize these curves, the inset shows an enlargement of a region of the bifurcation diagram that distinguishes the first three computed period-doubling bifurcation $\mathbf{PD}_{\Gamma_t}$, $\mathbf{PD}_{^2\Gamma_t}$ and $\mathbf{PD}_{^4\Gamma_t}$ of a period-doubling cascade. Furthermore, it shows the saddle-node bifurcation curves $\mathbf{SNP}_{\Gamma_o}$, $\mathbf{SNP}_{^3\Gamma_o}$, $\mathbf{SNP}_{^4\Gamma_o}$ and $\mathbf{SNP}_{^5\Gamma_o}$ which are responsible for the disappearance of the orientable saddle periodic orbits created during the homoclinic cascade, as will be discussed in detail in \cref{sec:Cascades}. In fact, by computing representative codimension-one homoclinic bifurcations close to $\mathbf{C_I}$, with an implementation of Lin's method in \textsc{Auto} \cite{san2}, we are able to obtain the bifurcating saddle periodic orbits; this allow us to find the subsequent saddle-node bifurcation curves presented in \cref{fig:NotTopInv}. For the remainder of this paper, the subindices in the label of each saddle-node bifurcation $\mathbf{SNP}$ of periodic orbits refers to the orientable saddle period orbit involved. The same applies to the period-doubling bifurcations $\mathbf{PD}$, where the subindices refer to the corresponding nonorientable saddle periodic orbit involved. Note how the saddle-node bifurcation curves accumulate onto the codimension-one homoclinic bifurcation $\mathbf{Tan}_{\Gamma_o}$ (violet curve) of $\Gamma_o$ where the stable and unstable manifolds of $\Gamma_o$ have a tangency. We identify the curve $\mathbf{Tan}_{\Gamma_o}$ as one of the boundaries of the Smale--horseshoe region $\mathbf{SH}$ (gray region), as it delimits the region in the bifurcation diagram where there exist structurally stable homoclinic orbits of $\Gamma_o$. 

Overall, the computed \cref{fig:NotTopInv} agrees well with the sketch in \cref{fig:TheoSketch}; however, one of the most apparent differences are the curves $\myEtoP{\mathbf{0}}{\Gamma_t}{m\Gamma_o}$, with $m=1,2,3,4$, of heteroclinic connecting orbits between $\mathbf{0}$ and $\Gamma_t$ (magenta curves). For each of these heteroclinic bifurcation curves, we find many homoclinic and heteroclinic bifurcation curves accumulating on it; these are not shown in \cref{fig:NotTopInv} because they lie very close to the curves $\myEtoP{\mathbf{0}}{\Gamma_t}{m\Gamma_o}$. We describe in more detail these accumulation cascades in \cref{sec:Cascades}. Notice in \cref{fig:NotTopInv} how the sequence of heteroclinic bifurcation curves $\myEtoP{\mathbf{0}}{\Gamma_t}{m\Gamma_o}$ accumulates onto the heteroclinic bifurcation $\mathbf{Q}_{\mathbf{0}}^{\Gamma_o}$ (purple curve) of a heteroclinic connecting orbit between $\mathbf{0}$ and $\Gamma_o$, which corresponds to the other boundary of the Smale--horseshoe region $\mathbf{SH}$ in the bifurcation diagram. Both boundary curves, $\mathbf{Tan}_{\Gamma_o}$ and $\mathbf{Q}_{\mathbf{0}}^{\Gamma_o}$, were not included or studied in the theoretical unfolding as shown in \cref{fig:TheoSketch}. Furthermore, the bifurcation $\mathbf{Tan}_{\Gamma_o}$ acts as mechanism to destroy a strange attractor, which was conjectured to exist close to the inclination flip point due to the cascades of period-doubling and saddle-node bifurcation \cite{Nau1,Nau2}; this will be discussed in more detail in \cref{sec:strAttrPD}.  

The main purpose of this paper is to perform a detailed study of the role of representative invariant manifolds in the reorganization of phase space at different codimension-one bifurcations in the bifurcation diagram of system~\cref{eq:san} near the point $\mathbf{C_I}$. In this way, we obtain new insights into the bifurcation structures that must be expected as part of the unfolding of the inclination flip bifurcation of case~\textbf{C}. Of particular interest is to understand the nature and roles of the infinitely many cascades of codimension-one homoclinic and heteroclinic bifurcations, the boundaries in parameter plane of the Smale--horseshoe regions, and the existence and annihilation of strange attractors, etc. To present our results, we choose parameter values in $(\alpha,\mu)$-plane close to $\mathbf{C_I}$ and provide representative figures of the relevant objects in phase space, as well as their intersection sets with a suitable sphere. This allows us to show the consequences of different codimension-one bifurcations for the organization of phase space, and to present an overall picture of the unfolding through the bifurcation diagram of $\mathbf{C_I}$ in the $(\alpha,\mu)$-plane.  Motivated by the existence of infinitely many homoclinic and heteroclinic bifurcations, we define a winding number $\zeta$ as a topological invariant for system~\cref{eq:san}. We compute the value of $\zeta$ for system~\cref{eq:san} as $\alpha$ and $\mu$ vary over a grid. Such a  two-parameter sweep allows us to identify open regions in the $(\alpha,\mu)$-plane with constant winding number, both near $\mathbf{C_I}$ and even further away from this central codimension-two point. The respective boundaries of these regions correspond to codimension-one homoclinic bifurcation curves. Hence, this parameter sweep contributes to obtaining a clearer picture of how the infinitely many homoclinic bifurcation are organized and accumulate onto heteroclinic bifurcations curves of the saddle periodic orbit $\Gamma_t$.  Furthermore, by considering a larger parameter range in the $(\alpha,\mu)$-plane, we are able to identify and characterize a phenomenon where two different homoclinic bifurcation curves emanating from $\mathbf{C_I}$ meet and create a structure referred to as a homoclinic bubble \cite{Hom2}. Such homoclinic bubbles do not arise as part of the unfolding in a small neighborhood of a codimension-two homoclinic flip bifurcation; however, they have been shown to be part of the unfolding of a codimension-three resonant homoclinic flip bifurcation, where they constitute a mechanism of  transition between cases~\textbf{B} and \textbf{C} \cite{Hom2, OldKra1}.

The organization of this paper is as follows. In \cref{sec:Not} we introduce some notation; here, we also present the parameter values used, and give the definition of the winding number $\zeta$. We present, in \cref{sec:IncCaseC}, the bifurcation diagram of system~\cref{eq:san} near $\mathbf{C_I}$. \Cref{sec:EasyRegions} then focusses on the transition through the main codimension-one homoclinic bifurcation curve, \cref{sec:Cascades} on the homoclinic and heteroclinic cascades, \cref{sec:HorseShoe} on the Smale--horseshoe region, and \cref{sec:strAttrPD} on the period-doubling cascade and the existence of a strange attractor that resembles the R\"{o}ssler attractor. Finally, we characterize, in \cref{sec:Bubbles}, the bubble phenomenon arising in the bifurcation diagram of $\mathbf{C_I}$.  We end in \cref{sec:Dis} with a discussion and an outlook on future research.  

The computations in this paper are performed with the pseudo-arclength continuation package \textsc{Auto} \cite{Doe1,Doe2} and its extension \textsc{HomCont} \cite{san2}.  More specifically, global manifolds are computed with a two-point boundary value problem set-up \cite{Call1, Kra2} and the heteroclinic orbits are obtained with Lin's method \cite{KraRie1,Kirk2}. The parameter sweeping of $\zeta$ in parameter plane, visualization and post-processing of the data are performed with \textsc{Matlab}\textsuperscript{\textregistered}.

\section{Notation and set-up} 
\label{sec:Not}

We choose parameters such that the Jacobian $Df(\mathbf{0})$ of $\textbf{0}$ has two stable and one unstable eigenvalues, $\lambda^{ss} < \lambda^{s} < 0 < \lambda^{u}$ with eigenvectors $e^{ss}, e^s$ and $e^u$, respectively. The global stable manifold $W^s(\textbf{0})$ is a surface foliated by trajectories that converge to $\textbf{0}$ as $t \rightarrow \infty$, and global unstable manifold $W^u(\textbf{0})$ consist of two trajectories that converge to $\textbf{0}$ as $t \rightarrow -\infty$. The manifold $W^s(\mathbf{0})$ and $W^u(\mathbf{0})$ are immersed manifolds in $\R^3$: they are as smooth as $f$ and tangent to the linear spaces spanned by the stable and unstable eigenvectors of $\mathbf{0}$, respectively \cite{Palis1}.  Furthermore, associated with $\lambda^{ss}$, there is a unique one-dimensional strong stable manifold $W^{ss}(\mathbf{0}) \subset W^{s}(\mathbf{0})$, defined as the subset of points on $W^{s}(\mathbf{0})$ that converges to $\mathbf{0}$ tangent to the eigenvector $e^{ss}$. 

System~\cref{eq:san} has a second equilibrium $\mathbf{q}$ for the parameters chosen, which is a stable focus that lies near $\mathbf{0}$. We denote its basin of attraction as $\mathcal{B}(\mathbf{q})$. The set $\mathcal{B}(\mathbf{q})$ is an open connected set of $\R^3$ and consists of all points in phase space that converge to $\mathbf{q}$. We also denote by $W^{ss}(\mathbf{q}) \subset \mathcal{B}(\mathbf{q})$ the subset of points that converge to $\mathbf{q}$ tangent to the eigenvector associated with the real eigenvalue of $\mathbf{q}$, which is the largest eigenvalue in modulus.

Let $\Gamma$ be a periodic orbit of system~\cref{eq:san}. We denote the two nontrivial Floquet multipliers of $\Gamma$ by $\Lambda_1,\Lambda_2 \in \C$; they are the eigenvalues of the variational equation of system~\cref{eq:san} along $\Gamma$ at time $T$, where $T$ is the period of $\Gamma$. Note that the Floquet multipliers of $\Gamma$ in a three-dimensional vector field are always such that their real parts have the same sign; each one has an associated eigenfunction that is referred as the Floquet bundle \cite{Leonid1}. If $\Lambda_1,\Lambda_2\in\R$ and $0<|\Lambda_1|<1<|\Lambda_2|$ then one speaks of a saddle periodic orbit. A saddle periodic orbit has stable $W^s(\Gamma)$ and unstable $W^u(\Gamma)$ manifolds which consist of points that converge to $\Gamma$ forward and backward in time, respectively. As for the saddle equilibrium case, the associated stable and unstable manifolds of a saddle periodic orbit are two dimensional immersed manifolds; they are tangent to the Floquet bundle of the periodic orbit associated with $\Lambda_1$ and $\Lambda_2$, respectively \cite{Palis1}.  If $0<\Lambda_1<1<\Lambda_2$, one speaks of an orientable saddle periodic orbit, which we denote by $\Gamma_o$, and its stable and unstable manifolds $W^s(\Gamma_o)$ and $W^u(\Gamma_o)$, respectively, are topological cylinders \cite{Hin1, Leonid1}. Similarly, if $\Lambda_2<-1<\Lambda_1<0$, then the saddle periodic orbit is nonorientable, denoted $\Gamma_t$, and $W^s(\Gamma_t)$ and $W^u(\Gamma_t)$ are locally topological M\"obius bands \cite{Hin1, Leonid1}. 

Let $S_1, S_2, S_3$ be hyperbolic saddle invariant objects of system~\cref{eq:san} with manifolds $W^u(S_1)$, $W^s(S_2)$, $W^u(S_2)$ and $W^s(S_3)$ that are all two-dimensional. If there exist structurally stable heteroclinic orbits from $S_1$ to $S_2$ and from $S_2$ to $S_3$, that is, if $W^u(S_1)\cap W^s(S_2)$ and $W^u(S_2)\cap W^s(S_3)$ are non-empty transversal intersections, then we use the notation $S_1 \rightarrow S_2 \rightarrow S_3$ to represent this situation.

\subsection{Sandstede's Model} 
\label{sec:san}

System~\cref{eq:san} was constructed specifically to study different homoclinic flip bifurcations in three-dimensional vector fields \cite{san1}.  It is a very convenient vector field, because its parameter can be chosen in such a way that either one of the cases \textbf{A}, \textbf{B} and \textbf{C} occur for both inclination and orbit flip bifurcations. System~\cref{eq:san} has been used extensively, particularly to study numerically transitions between the three cases as codimension-three phenomena due to resonance \cite{OldKra1}, and to investigate the unfoldings of homoclinic flip bifurcations of cases \textbf{A} and \textbf{B} \cite{Agu1, And1}. 

Note that $\mathbf{0}$ is an equilibrium of $X^s$ for all parameter values. We use the parameter ranges found in \cite{OldKra1} as a reference to study case \textbf{C} in the $(\alpha,\mu)$-plane and fix parameters as given in \cref{tab:sanValues}. Since $\delta=0$, the $z$-axis is invariant under the flow of system~\cref{eq:san} and the eigenvalues of $\mathbf{0}$ are given by
\begin{equation*}
\lambda_{1,2}= a \pm \sqrt{b^2+4\tilde{ \mu}^2} \text{ and }
\lambda_3=c.
\end{equation*}

\begin{table}
\begin{center}
\begin{tabular}{|c|c|c|c|c|c|c|}
\hline
$a$    & $b$ & $c$  & $\beta$ & $\gamma$ & $\tilde{\mu}$  & $\delta$\\ \hline
0.7  & 1.0 & -2.0 & 1.0    & 2.0     & 0.0    & 0.0   \\ \hline
\end{tabular}
\vspace{2mm}
\caption{Parameter values chosen for an inclination flip of case~$\mathbf{C}$ in system~\cref{eq:san}.} 
\label{tab:sanValues} 
\end{center}
\end{table}

The inclination flip bifurcation point $\mathbf{C_I}$ of case~\textbf{C} unfolds with respect to $\alpha$ and $\mu$. At $\mathbf{C_I}$, the equilibrium $\mathbf{0}$ has eigenvalues $\lambda_1=1.7$, $\lambda_2=-0.3$ and $\lambda_3=-2$, which confirms that $2|\lambda^s|<\lambda^u$ as required for case \textbf{C}. Note that the strong stable manifold $W^{ss}(\mathbf{0})$ is the $z$-axis. 

For our chosen parameters, there exists a stable equilibrium $\mathbf{q}$, which is the same equilibrium that appears in the discussion of cases \textbf{A} and \textbf{B} in \cite{Agu1, And1}. However, unlike for these two cases, the stability of $\mathbf{q}$ is not of relevance in our study of case~\textbf{C}. More specifically, we find that the orientable saddle periodic orbit $\Gamma_o$ takes on the role of \textbf{q}; in particular it is responsible for the creation of a fold bifurcation \textbf{F} in case~\textbf{C}. For the case \textbf{C} considered in this paper, the equilibrium $\mathbf{q}$ is attracting and identifying its basin of attraction is critical for understanding the reorganization of phase space close to $\mathbf{C_I}$; see \cref{sec:IncCaseC}. 

\subsection{Definition of the winding number}
\label{sec:topInv}

The unfolding of an inclination flip bifurcation of case \textbf{C} is characterized by the existence of $k$-homoclinic orbits for any $k \in \N$ \cite{Hom1}.  In \cite{OldKra1}, these codimension-one homoclinic bifurcations in the parameter plane are studied via the continuation of solutions to a suitable two-point boundary value problem (2PBVP) during the transition between cases~\textbf{B} and \textbf{C}. A drawback of this technique is the limitation of sampling only a small number of these bifurcation curves, which increases the risk of missing subtle interactions in the parameter plane.  We define a winding number $\zeta$ for system~\cref{eq:san} and run a two-parameter sweep in the $(\alpha,\mu)$-plane to complement the 2PBVP approach of finding bifurcation curves. Parameter sweeping techniques have been used in part to understand the nature of cascades of homoclinic bifurcation close to Bykov T-points in the Lorenz system \cite{Shil3} and Shimizu-Morioka system \cite{Shil3, Shil1}. It is also useful to illustrate spike adding in neuron models; for example, see \cite{Barrio1,Shil4}.  By combining continuation and parameter sweeping, we are able to characterize different phenomena not only in the vicinity of the homoclinic flip bifurcation point of case \textbf{C}, but also far away from it.  

\begin{figure}
\centering
\includegraphics{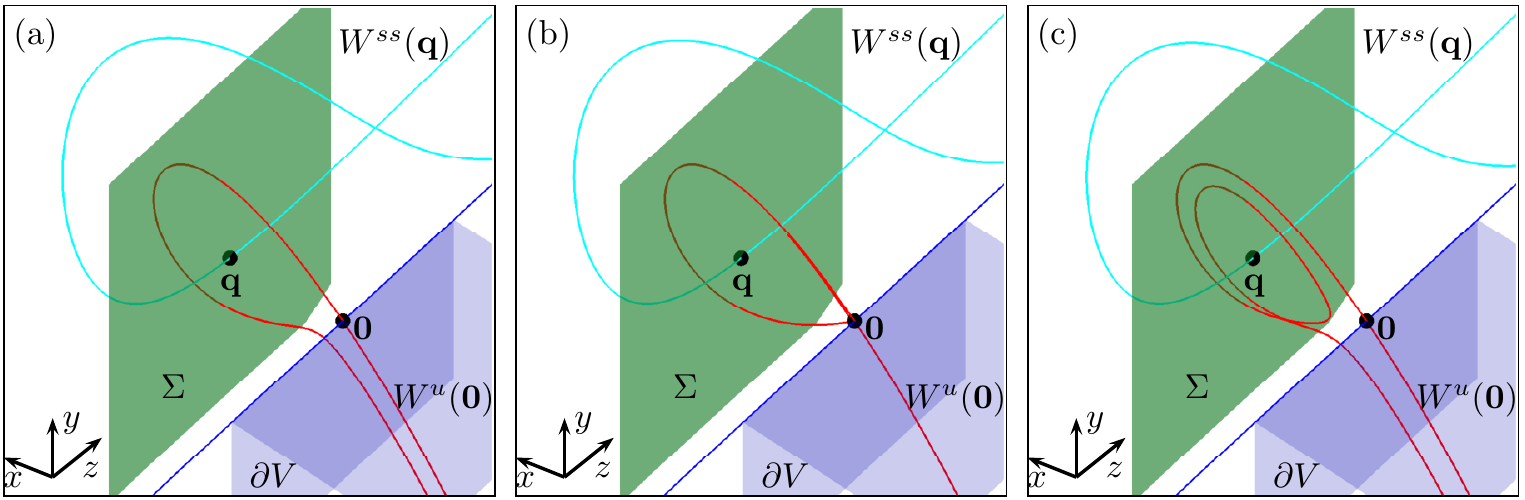}
\caption{Illustration of the change of the winding number $\zeta$ for system~\cref{eq:san}.  Shown are $W^u(\mathbf{0})$ (red curve), $W^{ss}(\mathbf{0})$ (dark-blue curve) and $W^{ss}(\mathbf{q})$ (cyan curve), $\Sigma$ (green plane), the boundaries $\partial V$ (blue planes). The parameter values are $\mu=0.001$, $\mu=0.0$ and $\mu=-0.001$ for panels~(a), (b) and (c), respectively. Furthermore, $\alpha=0.5$ and the other parameter values are as given in \cref{tab:sanValues}.} 
\label{fig:topInv} 
\end{figure}  

The value of $\zeta$ encodes the number of rotations that $W^u(\mathbf{0})$ makes around a tubular neighborhood of the homoclinic orbit that exists at the central singularity. In the present setting, $\zeta$ can be defined conveniently as the number of rotations around the one-dimensional manifold $W^{ss}(\mathbf{q})$, which always goes through `the hole' of any such tubular neighborhood. To define the quantity $\zeta$ formally, let $V:=\{p=(x,y,z)\in\R^3 \;|\; x\leq0\text{ and }y\leq 0\}$ and let $\partial V$ be its boundary. Since we set $b=1>0$ and $\delta=0$ in \cref{tab:sanValues}, we have $\dot{x}<0$ on $\{ x=0 \}\cap V\setminus \{y=0\}$ and $\dot{y}<0$ on $ \{ y=0 \} \cap V\setminus \{ x=0\}$. Furthermore, the $z$-axis, $\{ x=0\} \cap \{ y=0\}$, is a subset of $V$ and it is invariant.  Hence, $V$ is a positively invariant set for system~\cref{eq:san}, i.e., $\phi^t(V)\subset V$ for all $t \geq 0$ where $\phi^t$ is the flow defined by system~\cref{eq:san}. Therefore, any intersection of $W^u(\mathbf{0})$ with $\partial V$ must be transversal and $W^{ss}(\mathbf{q})$ cannot intersect $\partial V$. We view $\zeta$, that is, the number of rotations that $W^u(\mathbf{0})$ makes around $W^{ss}(\mathbf{q})$, as a kind of linking number, which can only vary through bifurcation. Indeed, $W^{ss}(\mathbf{q})$ never intersects $\partial V$, and $W^u(\mathbf{0})$ cannot follow $W^{ss}(\mathbf{q})$ to undo itself by restriction of the flow of system~\cref{eq:san}; therefore, we may view these orbits segments as closed curves by identifying their endpoints.

Homoclinic bifurcations are a mechanism for $\zeta$ to change, as is illustrated in \cref{fig:topInv} for system~\cref{eq:san} with $\alpha=0.5$ and three different values of $\mu$.  In panel (a) the manifold $W^u(\mathbf{0})$ loops once around $W^{ss}(\mathbf{q})$ before reaching $\partial V$, that is, $\zeta=1$.  As $\mu$ decreases, system~\cref{eq:san} goes through a homoclinic bifurcation at $\mu = 0$ in panel~(b), after which $W^u(\mathbf{0})$ makes an extra turn around $W^{ss}(\mathbf{0})$ before intersecting $\partial V$, as shown in panel~(c); hence, $\zeta$ increases to 2. 

In practice, we calculate $\zeta$ by counting the number of intersections of $W^u(\mathbf{0})$ with $\Sigma := \lp\{ (x,y,z) \in \R^3 : x = \mathbf{q}_x \rp\}$, where $\mathbf{q}_x$ is the $x$-component of $\mathbf{q}$, and dividing this number by two; see \cref{fig:topInv}.  The parameter sweeps of $\zeta$ shown in \cref{fig:Second,,fig:Third,,fig:Bubble} are performed on a $1000\times1000$ grid over the corresponding parameter ranges.

\section{Bifurcation diagram near $\mathbf{C_I}$} 
\label{sec:IncCaseC}

\begin{figure}
\centering
\includegraphics[height=250pt]{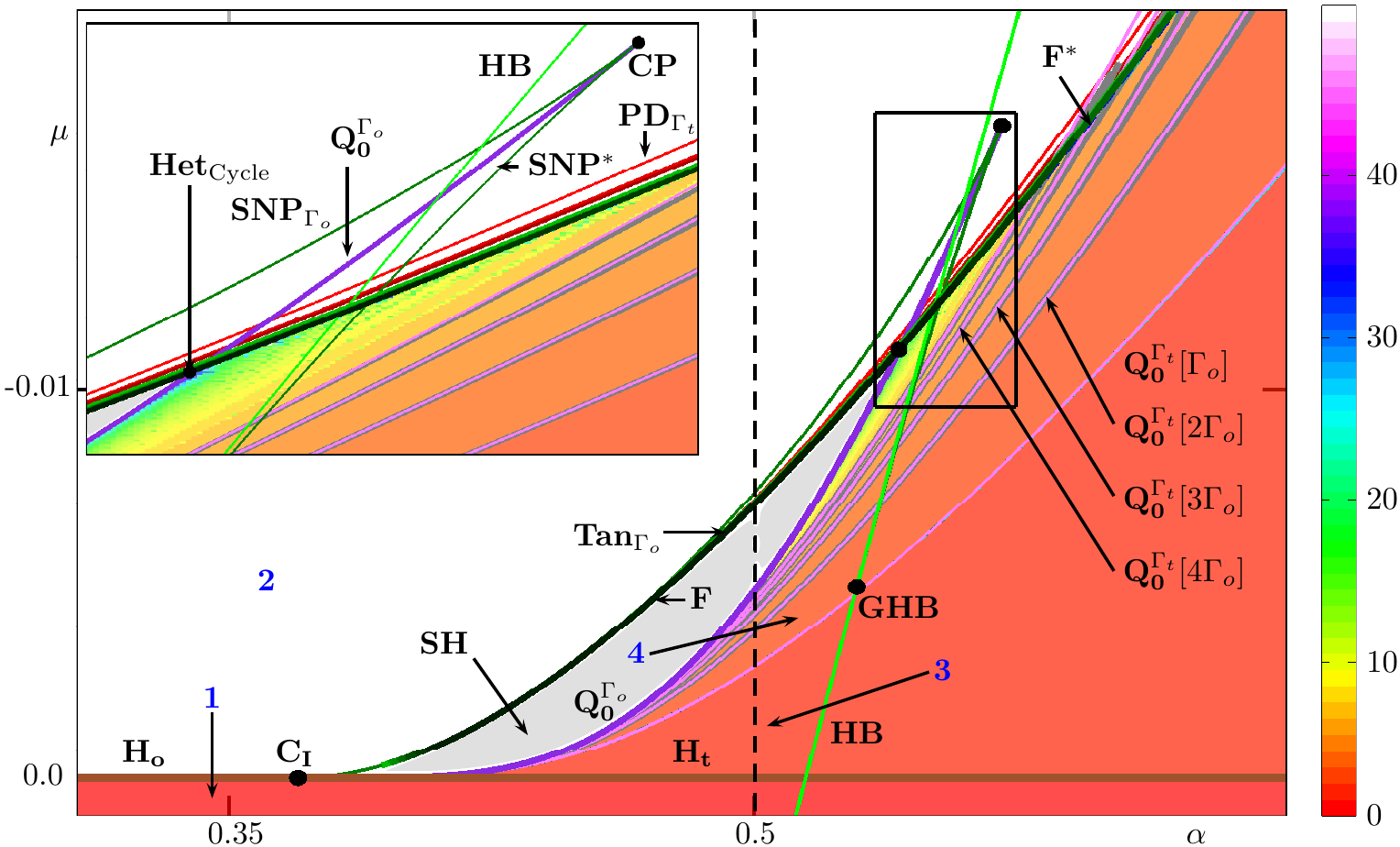} 
\caption{Bifurcation diagram in the $(\alpha,\mu)$-plane near an inclination flip bifurcation $\mathbf{C_I}$ of case~\textbf{C} of system~\cref{eq:san}, with coloring of regions according to the winding number $\zeta$ as given by the colorbar.  The inset shows an enlargement of the region indicated in the main panel. Shown are the curves of homoclinic bifurcation~$\mathbf{H_o}$ and $\mathbf{H_t}$ (brown), non-principal homoclinic bifurcation (cyan), saddle-node bifurcation $\mathbf{SNP}_{\Gamma_o}$ and $\mathbf{SNP^*}$ of periodic orbits (dark green), Hopf bifurcation $\mathbf{HB}$ (light green), period-doubling bifurcation $\mathbf{PD}$ (red), fold bifurcation curve $\mathbf{F}$ of heteroclinic orbits from $\Gamma_o$ to $\mathbf{0}$ (blue) and its extension $\mathbf{F^*}$ (blue dashed curve), heteroclinic bifurcation $\mathbf{Q}_{\mathbf{0}}^{\Gamma_t}$ from $\mathbf{0}$ to $\Gamma_t$ (magenta), heteroclinic bifurcation $\mathbf{Q}_{\mathbf{0}}^{\Gamma_o}$ from $\mathbf{0}$ to $\Gamma_o$ (purple) and codimension-one homoclinic bifurcation $\mathbf{Tan}_{\Gamma_o}$ of $\Gamma_o$ (violet). The curves $\mathbf{SNP}_{\Gamma_o}$ and $\mathbf{SNP^*}$ meet at the cusp point $\mathbf{CP}$, and $\mathbf{SNP^*}$ ends at the generalized Hopf bifurcation point $\mathbf{GHB}$. The curve $\mathbf{F}$ meets $\mathbf{Q}_{\mathbf{0}}^{\Gamma_o}$ at a codimension-two heteroclinic cycle point $\mathbf{Het_{\rm Cycle}}$; the Smale--horseshoe region $\mathbf{SH}$ (gray region) is indicated.  We denote by \mBlue{1}, \mBlue{2}, \mBlue{3} and \mBlue{4} four different regions close to the $\mathbf{C_I}$ point and the line $\alpha=0.5$ (dashed). The other parameters are as in \cref{tab:sanValues}.} 
\label{fig:Second} 
\end{figure}  

We now present more details of the bifurcation diagram shown in \cref{fig:NotTopInv}. Namely, \cref{fig:Second} shows it over a slightly larger range of the $(\alpha, \mu)$-plane together with the parameter sweep of $\zeta$ where the colors indicate the value of the winding number~$\zeta$; the inset shows an enlargement of the indicated region of the main panel. At first glance, we distinguish the curves of period-doubling, saddle-node and heteroclinic bifurcations presented in \cref{fig:NotTopInv}. Recall that the subindices in the label of each saddle-node bifurcation $\mathbf{SNP}$ and period-doubling bifurcation $\mathbf{PD}$ of periodic orbits refers to the orientable and nonorientable saddle period orbit involved, respectively. Specifically for this bifurcation diagram, we only compute and label one of the infinitely many saddle-node and period-doubling curves that emanate from $\mathbf{C_I}$, namely, the saddle-node bifurcation $\mathbf{SNP}_{\Gamma_o}$ of $\Gamma_o$ (green curve) and period-doubling bifurcation $\mathbf{PD}_{\Gamma_t}$ of $\Gamma_t$ (red curve). 

We now discuss the overall features of the bifurcation diagram in \cref{fig:Second}; further details will be presented in later sections. Note that, in between successive curves $\myEtoP{\mathbf{0}}{\Gamma_t}{m\Gamma_o}$, with $m=1,2,3,4$, of heteroclinic connecting orbits between $\mathbf{0}$ and $\Gamma_t$ (magenta curves), there exist big regions of the $(\alpha,\mu)$-parameter plane with constant values of $\zeta$. The value of $\zeta$ in these regions increases as we approach the bifurcation $\mathbf{Q}_{\mathbf{0}}^{\Gamma_o}$ (purple) of a heteroclinic connecting orbit between $\mathbf{0}$ and $\Gamma_o$. This indicates the existence of more bifurcation curves $\myEtoP{\mathbf{0}}{\Gamma_t}{m\Gamma_o}$, for $m > 4$. The labels of the codimension-one heteroclinic bifurcation are deliberately chosen to encode information of the corresponding heteroclinic orbit in phase space. \Cref{fig:ConfigHet} shows, in phase space, representative heteroclinic orbits $\mathbf{Q}_{\mathbf{0}}^{\Gamma_t}$ and $\mathbf{Q}_{\mathbf{0}}^{\Gamma_o}$ (red curves) together with the saddle periodic orbits $\Gamma_t$ (purple curve) and $\Gamma_o$ (green curve) at the moment of the corresponding heteroclinic bifurcation. For simplicity, we use the same label of the heteroclinic bifurcation to refer to the corresponding orbit in phase space.  The heteroclinic orbit is formed by one branch of $W^u(\mathbf{0})$, which is the branch shown in \fref{fig:ConfigHet}. In panel~(a1), the heteroclinic orbit $\myEtoP{\mathbf{0}}{\Gamma_t}{\Gamma_o}$ makes an excursion close to the $(x,y)$-plane, and then follows $\Gamma_o$ for one rotation before accumulating onto $\Gamma_t$. In panels~(a2)-(a5), the branch of $W^u(\mathbf{0})$ also accumulates onto $\Gamma_t$, but only after making two, three, four and five rotations near $\Gamma_o$, respectively. Hence, the number of rotations that the respective heteroclinic orbit of $\mathbf{Q}_{\mathbf{0}}^{\Gamma_t}$ makes in phase space around $\Gamma_o$ increases until it reaches its limiting case, where the number of rotations around $\Gamma_o$ has increased to infinity, at $\mathbf{Q}_{\mathbf{0}}^{\Gamma_o}$ as shown in panel~(b). This accumulation is seen in \cref{fig:Second}, where the cascade $\myEtoP{\mathbf{0}}{\Gamma_t}{m\Gamma_o}$ accumulates on the final codimension-one heteroclinic bifurcation $\mathbf{Q}_{\mathbf{0}}^{\Gamma_o}$ of a connecting orbit between $\mathbf{0}$ and $\Gamma_o $ (dark-purple curve).  Note also that the winding number $\zeta$ is at its computational maximum after this curve; this is due to the accumulation of the unstable manifold $W^u(\mathbf{0})$ onto $\mathbf{q}$ after $\mathbf{Q}_{\mathbf{0}}^{\Gamma_o}$ is crossed, as will be detailed in \cref{sec:HorseShoe}. This curve $\mathbf{Q}_{\mathbf{0}}^{\Gamma_o}$ represents the last moment where system~\cref{eq:san} exhibits a homoclinic or heteroclinic bifurcation close to $\mathbf{C_I}$, after which the dynamics are complicated due to the existence of structurally stable homoclinic orbits of $\Gamma_o$.

\begin{figure}
\centering
\includegraphics{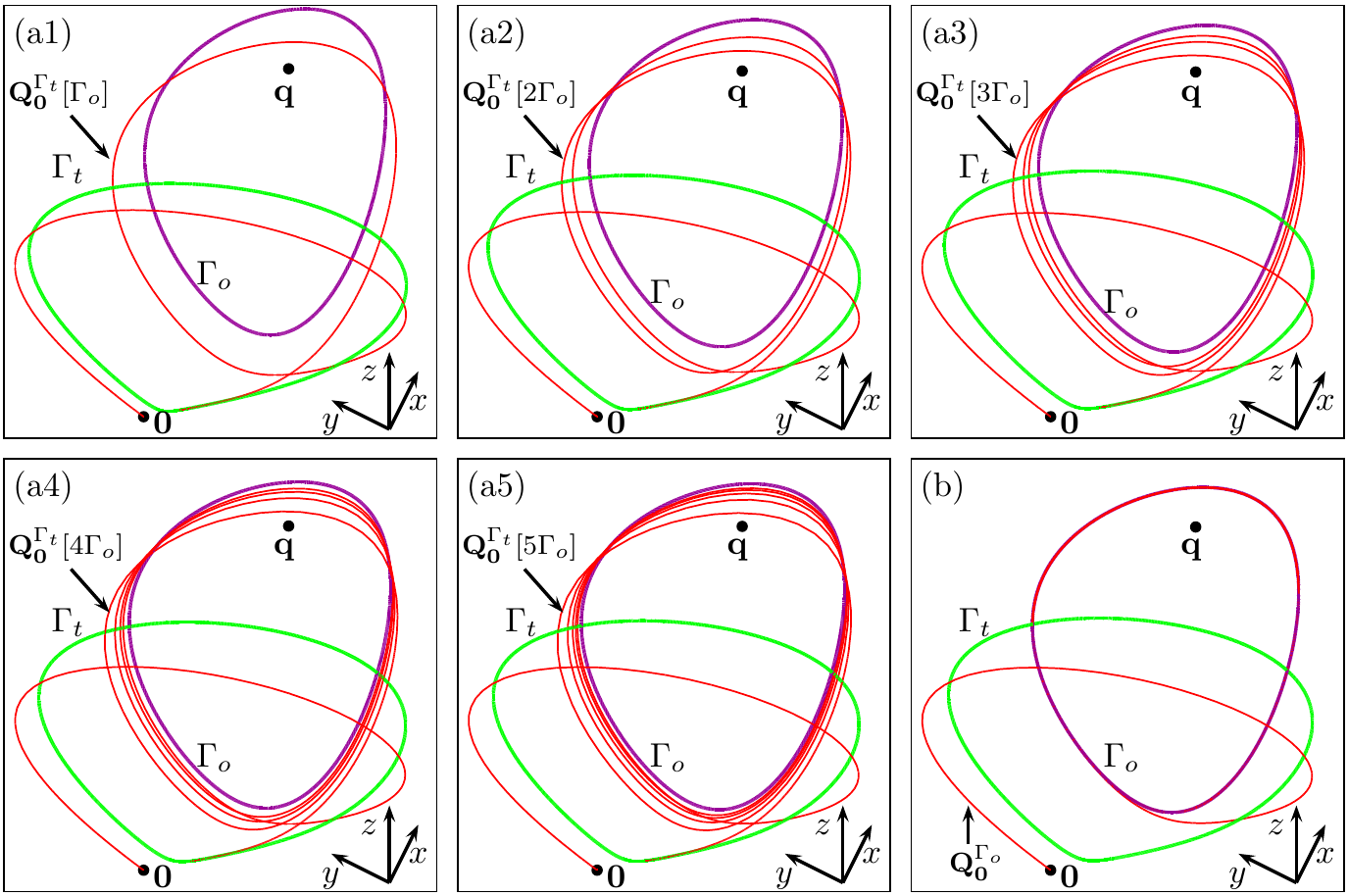}
\caption{Configuration of the unstable manifold $W^u(\mathbf{0})$ of system~\cref{eq:san} in $\R^3$ at the codimension-one heteroclinic bifurcations. Shown are the heteroclinic branch of $W^u(\mathbf{0})$ (red), and the periodic orbits $\Gamma_o$ (purple) and $\Gamma_t$ (green).  Panels~(a1)-(a5) illustrate connections from $\mathbf{0}$ to $\Gamma_t$ and panel (b) shows the connection from $\mathbf{0}$ to $\Gamma_o$. Here, we used $\alpha=0.5$ and the respective $\mu$-values for each panel are given in \cref{tab:Sequence}.} 
\label{fig:ConfigHet}
\end{figure}  

It was proved under the conditions of case~\textbf{C} that a H\'enon-like attractor unfolds from $\mathbf{C_I}$ \cite{Nau1}. The Smale--horseshoe region was conjectured to be bounded by a first quadratic tangency between the stable and unstable manifolds of a saddle periodic orbit \cite{Nau1}. However, it was not even clear at the time that such a first tangency existed in the unfolding of an inclination flip point.  We have found this first tangency and it corresponds to the bifurcation curve $\mathbf{Tan}_{\Gamma_o}$ (violet curve) in \cref{fig:Second} where the stable and unstable manifolds of $\Gamma_o$ have a tangency. After leaving the Smale--horseshoe region through $\mathbf{Tan}_{\Gamma_o}$, we find a strange attractor in phase space whose corresponding branched manifold resembles that of the R\"{o}ssler attractor. This attractor exists during a reverse period-doubling cascade which terminates at the bifurcation $\mathbf{PD}_{\Gamma_t}$, where the saddle periodic orbit $\Gamma_t$ changes to an attracting periodic orbit $\Gamma^a_t$ and the corresponding attracting periodic orbit $^2\Gamma^a_t$ of twice the period disappears. Then $\Gamma^a_t$ disappears in the saddle-node bifurcation $\mathbf{SNP}_{\Gamma_o}$ with $\Gamma_o$. 

Note in \cref{fig:Second} that we also find a bifurcation curve $\mathbf{F}$, similar to the ones found in \cite{Agu1, And1}, that represents the moment when $W^s(\mathbf{0})$ becomes tangent to $W^u(\Gamma_o)$. If one follows the curve~$\mathbf{F}$ in \cref{fig:Second}, one sees that it intersects the curve $\mathbf{Q}_{\mathbf{0}}^{\Gamma_o}$ at a point that we labeled $\mathbf{Het_{\rm Cycle}}$ which is seen more clearly in the corresponding enlargement. At this point, there exists a codimension-two heteroclinic cycle between $\mathbf{0}$ and $\Gamma_o$ in system~\cref{eq:san}.  The unfolding of this cycle has been studied theoretically \cite{Kirk1,Lohr1} as an organizing center for the creation of chaotic regions in the parameter plane and the creation of multi-pulse homoclinic solutions that exhibit flip bifurcations close to it. This agrees with our computations, because the Smale--horseshoe region $\mathbf{SH}$ in \cref{fig:Second} also unfolds from the point $\mathbf{Het_{\rm Cycle}}$; moreover, different homoclinic bifurcation branches in the bifurcation diagram exhibit inclination flip bifurcations, as will be discussed in detail in \cref{sec:Bubbles}. 

Finally, note in \fref{fig:Second} that the additional stable equilibrium $\mathbf{q}$ of system~\cref{eq:san} goes through a Hopf bifurcation labeled $\mathbf{HB}$ (light-green curve) and becomes an unstable saddle focus. On the curve $\mathbf{HB}$, there exists a generalized Hopf bifurcation point $\mathbf{GHB}$, which gives rise to a curve (dark-green) of saddle-node bifurcation $\mathbf{SNP^*}$ at which $\Gamma_o$ and an attracting periodic orbit $\Gamma^a$ are created. The curve $\mathbf{SNP^*}$ ends in a cusp bifurcation point $\mathbf{CP}$ with the curve $\mathbf{SNP}_{\Gamma_o}$; see the inset in \cref{fig:Second}. The periodic orbits $\Gamma_o$ and $\Gamma_t$ disappear on the other side of $\mathbf{SNP}_{\Gamma_o}$ and only $\Gamma^a$ persists; hence, the bifurcation curves $\mathbf{Q}_{\mathbf{0}}^{\Gamma_o}$ and $\mathbf{F}$ disappear at the point $\mathbf{CP}$ and the curve $\mathbf{SNP^*}$, respectively. In particular, we find another curve $\mathbf{F^*}$ (dashed blue line), which represents the moment when $W^s(\mathbf{0})$ becomes tangent to $W^u(\mathbf{q})$; this curve constitutes the extension of $\mathbf{F}$ past its intersection with $\mathbf{SNP^*}$. 

\begin{table}
\begin{center}
\begin{tabular}{c|r@{.}l|r@{.}l|r@{.}l|r@{.}l|r@{.}l|r@{.}l|} 
\cline{2-13}
\multirow{2}{*}{}  & \multicolumn{4}{c|}{\mBlue{1}}  &
\multicolumn{2}{c|}{\multirow{2}{*} {\mBlue{2}}} &
\multicolumn{2}{c|}{\multirow{2}{*}{\mBlue{3}}}   &
\multicolumn{2}{c|}{\multirow{2}{*}{$\mathbf{H_o}$}}  &
\multicolumn{2}{c|}{\multirow{2}{*}{$\mathbf{H_t}$}} \\ \cline{2-5}  
  & \multicolumn{2}{c|}{\mBlue{1_o}} &  \multicolumn{2}{c|}{\mBlue{1_t}}
  &  \multicolumn{2}{c|}{}  &  \multicolumn{2}{c|}{}   &
  \multicolumn{2}{c|}{}   &  \multicolumn{2}{c|}{}   \\ \hline 
\multicolumn{1}{|c|}{$\alpha$} & 0 & 200 &  0 & 500 & 0 & 200 & 0 & 500 & 0 & 200 & 0 & 500 \\ \hline
\multicolumn{1}{|c|}{$\mu$} & $0$ & 050 &  0 & 001 & $-0$ & 001 & $-0$ & 002 & $0$ & 000 &
$0$ & 000 \\  \hline
\end{tabular}
\vspace{2mm}
\caption{Representative parameter values for the open regions \mBlue{1}--\mBlue{3} and the homoclinic bifurcations $\mathbf{H_o}$ and $\mathbf{H_t}$, as used in \fref{fig:Easy}; all other parameter values are as in \cref{tab:sanValues}.} 
\label{tab:Inc1}
\end{center}
\end{table}

\begin{figure}
\centering
\includegraphics{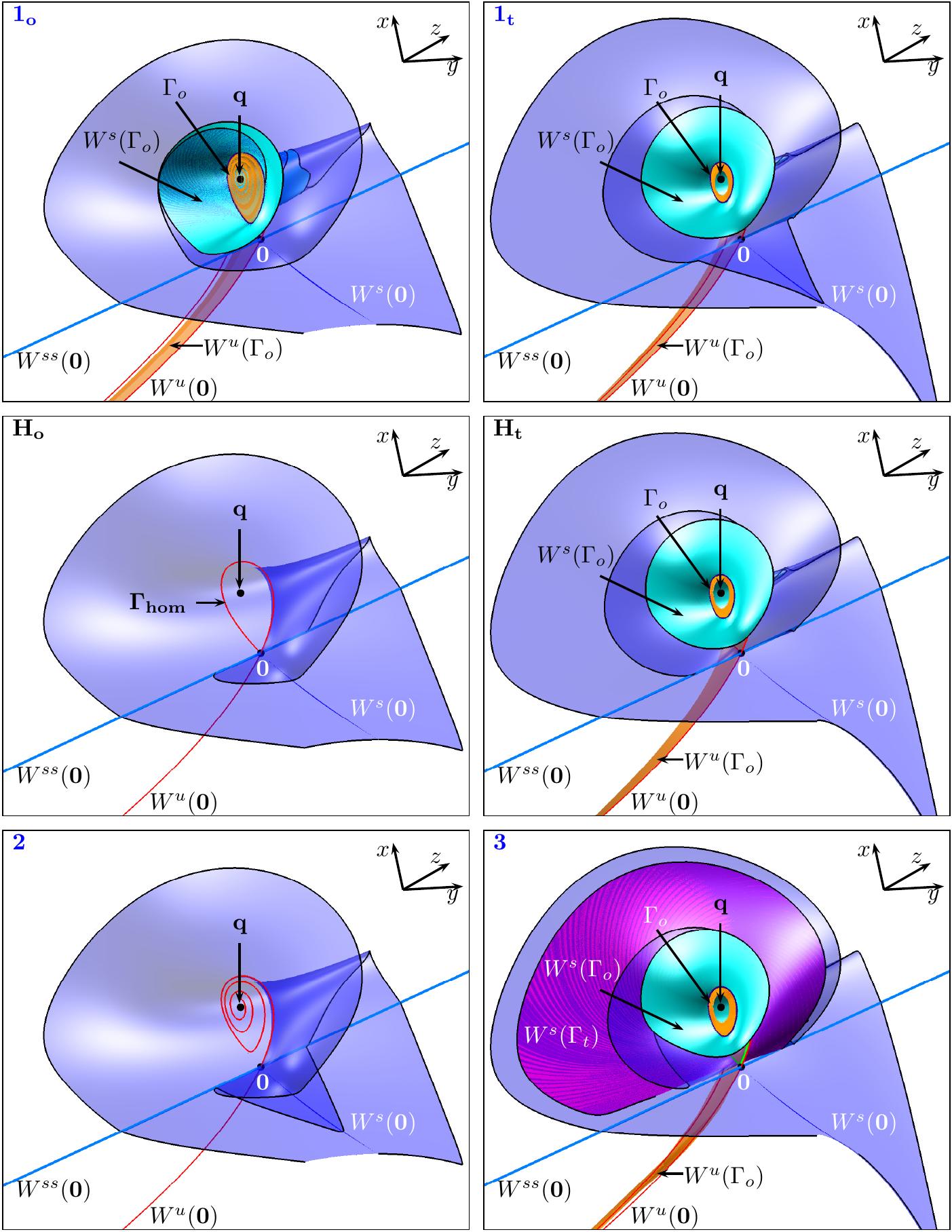}
\caption{Phase portraits of system~\cref{eq:san} in regions \mBlue{1}--\mBlue{3} of the $(\alpha,\mu)$-plane and along the bifurcations $\mathbf{H_o}$ and $\mathbf{H_t}$ near the inclination flip point $\mathbf{C_I}$.  Shown are $W^s(\mathbf{0})$ (blue surface), $W^{ss}(\mathbf{0})$ (blue curve), $W^{u}(\mathbf{0})$ (pink curve), $W^s(\Gamma_o)$ (cyan surface) and $W^s(\Gamma_t)$ (purple surface). The outer rims of the two-dimensional stable manifolds are highlighted in black. The $(\alpha,\mu)$-values for each panel are given in \cref{tab:Inc1}. See also the accompanying animation ({\color{red} GKO\_Cflip\_animatedFig7.gif}).} 
\label{fig:Easy}
\end{figure} 

\section{Transition through the homoclinic bifurcations} 
\label{sec:EasyRegions}

In order to illustrate the difference between crossing the orientable homoclinic bifurcation $\mathbf{H_o}$ and the nonorientable homoclinic bifurcation $\mathbf{H_t}$ in \cref{fig:Second}, we choose representative points in the regions \mBlue{1} to \mBlue{3} and at the bifurcations $\mathbf{H_t}$ and $\mathbf{H_o}$; see \cref{tab:Inc1} for the precise parameters chosen. In particular, we choose two points in region~\mBlue{1}, labeled  $\mBlue{1_o}$ and $\mBlue{1_t}$, which are close to $\mathbf{H_o}$ and $\mathbf{H_t}$, respectively.  The phase portraits with the respective global invariant objects in $\R^3$ are shown in \cref{fig:Easy}.

Region \mBlue{1} is characterized by the existence of the orientable saddle periodic orbit $\Gamma_o$, the saddle equilibrium $\mathbf{0}$ and the stable-focus $\mathbf{q}$.  Panel~$\mBlue{1_o}$ of \fref{fig:Easy} shows the corresponding phase portrait, which consists of the stable manifolds $W^s(\mathbf{0})$ (blue surface) and $W^s(\Gamma_o)$ (cyan surface), and the unstable manifolds $W^u(\mathbf{0})$ (red curve) and $W^u(\Gamma_t)$ (orange surface). Note that $W^s(\Gamma_o)$ is a topological cylinder that bounds the basin of attraction $\mathcal{B}(\mathbf{q})$ of $\mathbf{q}$. The manifold $W^s(\mathbf{0})$ accumulates from the outside onto $W^s(\Gamma_o)$ backward in time. One of the sheets that form $W^u(\Gamma_o)$ lies inside the topological cylinder $W^s(\Gamma_o)$ and, as such, it accumulates onto $\mathbf{q}$; the other sheet is bounded by $W^u(\mathbf{0})$. The accumulation of $W^s(\mathbf{0})$ onto $W^s(\Gamma_o)$ and the fact that $W^u(\mathbf{0})$ bounds one sheet of $W^u(\Gamma_o)$ are due to the $\lambda$-lemma \cite{Palis1,Wigg1}: the existence of a nontransversal intersection between $W^s(\mathbf{0})$ and $W^u(\mathbf{0})$ implies the existence of a structurally stable heteroclinic orbit from $\Gamma_o$ to $\mathbf{0}$. Notice that $W^u(\mathbf{0})$ spirals once around $\Gamma_o$ before escaping to infinity, so that $\zeta=1$ for this region. The boundary between regions~$\mBlue{1_o}$ to \mBlue{2} is the codimension-one orientable homoclinic bifurcation $\mathbf{H_o}$. At this bifurcation, the saddle periodic orbit $\Gamma_o$ becomes the homoclinic orbit $\mathbf{\Gamma_{\rm hom}}$ and the stable manifold $W^s(\mathbf{0})$ closes along its strong stable manifold $W^{ss}(\mathbf{0})$ in a topological cylinder; see panel~$\mathbf{H_o}$ of \fref{fig:Easy}. However, the moment one transitions into region~\mBlue{2}, the homoclinic orbit $\mathbf{\Gamma_{\rm hom}}$ disappears, allowing $W^u(\mathbf{0})$ to accumulate onto $\mathbf{q}$; see panel~\mBlue{2} of \fref{fig:Easy}. Hence, the $\zeta$-value in this region is infinite. 

We now focus on the transition through the codimension-one nonorientable homoclinic bifurcation $\mathbf{H_t}$. At the point~\mBlue{1_t}, system~\cref{eq:san} is close to a nonorientable homoclinic bifurcation.  Note that the computed manifolds in panels~\mBlue{1_t} and \mBlue{1_o} are topologically equivalent, but the way $W^s(\mathbf{0})$ approaches backward in time onto $W^{ss}(\mathbf{0})$ vastly differs. In particular, note in panel~\mBlue{1_t} how the bottom part of the outer layer of $W^s(\mathbf{0})$ that accumulates onto $W^s(\Gamma_o)$ twists as it gets closer to $W^{ss}(\mathbf{0})$.  In the transition from region~$\mBlue{1_t}$ to region~\mBlue{3}, system~\cref{eq:san} exhibits the bifurcation $\mathbf{H_t}$ illustrated in panel~$\mathbf{H_t}$ of \fref{fig:Easy}. At this bifurcation, the stable manifold $W^s(\mathbf{0})$ closes along its strong stable manifold $W^{ss}(\mathbf{0})$ in a topological M\"obius band. In contrast to $\mathbf{H_o}$, the saddle periodic orbit $\Gamma_o$ does not become the homoclinic orbit $\mathbf{\Gamma_{\rm hom}}$ at the bifurcation $\mathbf{H_t}$.  Then in region~\mBlue{3}, the saddle periodic orbit $\Gamma_o$ persists and $\mathbf{\Gamma_{\rm hom}}$ bifurcated into the nonorientable saddle periodic orbit $\Gamma_t$; see panel~\mBlue{3} of \cref{fig:Easy}. Its stable manifold $W^s(\Gamma_t)$ (purple surface) accumulates onto $\Gamma_o$ backward in time, and $W^s(\mathbf{0})$ lies in between $W^s(\Gamma_o)$ and $W^s(\Gamma_t)$. Contrary to region~\mBlue{2}, the unstable manifold $W^u(\mathbf{0})$ spirals twice around $W^s(\Gamma_o)$ and then escapes to infinity, so that $\zeta=2$ for this region. Furthermore, the two-dimensional stable and unstable manifolds in region~\mBlue{3} have transverse intersections; in particular, we have a configuration of the form $\Gamma_o\rightarrow \Gamma_t \rightarrow \mathbf{0}$, that is, there exist heteroclinic orbits from $\Gamma_o$ to $\Gamma_t$ and from $\Gamma_t$ to $\mathbf{0}$.  

By applying the $\lambda$-lemma \cite{Palis1,Wigg1} in a Poincar\'e section of $\Gamma_t$, we can show that $\Gamma_o\rightarrow \Gamma_t \rightarrow \mathbf{0}$ implies, in fact, the existence of infinitely many heteroclinic orbits from $\Gamma_o$ to $\mathbf{0}$.  Hence, we have the following:
\begin{rem} \label{rem:CaseB}
The phase portrait in region~\mBlue{3} is topological equivalent to that in region~\mRed{4} for case~\textbf{B} if we contract $\Gamma_o$ to the point $\mathbf{q}$; see \cite{And1}.
\end{rem}

\begin{figure}
\centering
\includegraphics{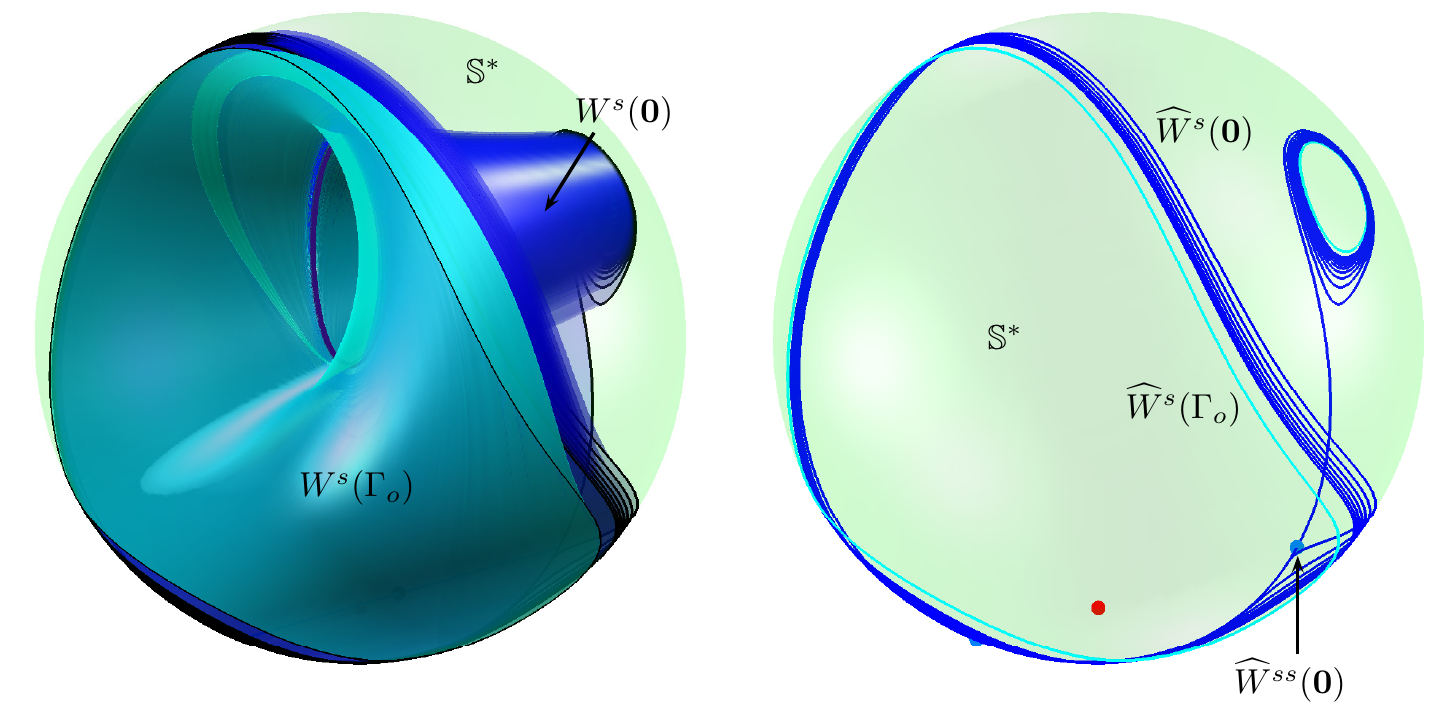}
\caption{ Illustration of the intersection of the stable manifolds of $\mathbf{0}$ and $\Gamma_o$ in system~\cref{eq:san} with the sphere $\mS^*$ of radius $0.6$ and center $(0.5,0,0)$; see also the accompanying animation ({\color{red} GKO\_Cflip\_animatedFig8.gif}). The left panel shows $W^s(\mathbf{0})$ (blue surface), $W^{ss}(\mathbf{0})$ (blue curve) and $W^s(\Gamma_o)$ (cyan surface) inside $\mS^*$ (green suface). The right panel shows only the corresponding intersection sets $\widehat{W}^s(\mathbf{0})$ (blue curve), $\widehat{W}^{ss}(\mathbf{0})$ (blue dots), and $\widehat{W}^s(\Gamma_o)$ (cyan curve) on $\mS^*$; the red dot is the stereographic projection point on $\mS^*$. The values of $\alpha$ and $\mu$ correspond to region $\mBlue{1_t}$ and are given in \cref{tab:Inc1}; compare with the corresponding panel in \cref{fig:FirstStereo}.} 
\label{fig:ExampleStereo}
\end{figure} 

\subsection{Intersection sets with a sphere}

We study now the intersection sets of the stable manifolds with the sphere $\mS^*:=\{x \in \R^{3} : \aNorm{x-c}=R\}$, where $c:=(c_x,c_y,c_z)=(0.5,0,0)$ and $R=0.6$. Since $\mS^*$ is a compact set, all intersections of the stable manifolds with $\mS^*$ are bounded. In particular, two-dimensional manifolds (in general position) lead to curves on $\mS^*$, while one-dimensional stable manifolds intersect $\mS^*$ in points. This is illustrated in \fref{fig:ExampleStereo} for the invariant manifolds in region $\mBlue{1_t}$. 

Next, we stereographically project these intersection sets onto one of the tangent planes of $\mS^*$ to present them as two-dimensional pictures. More precisely, we first apply the transformation
\begin{equation*}
\pi=\lp(\frac{u}{\aNorm{u}},\frac{v}{\aNorm{v}},\frac{w}{\aNorm{w}}\rp),
\end{equation*}
formed by the vectors 
\begin{equation*}
u=(0.5706,0.1854,0)^T, v=(1,-(u_x+u_z)/u_y,1)^T \text{ and } w=u \times v.
\end{equation*}
We rotate and translate points $(x,y,z)\in \mS^*$ to points $(x',y',z')$ on the sphere of radius $R=0.6$ centered at the origin, that is, $(x',y',z')=\pi(x-c)$. This transformation rotates the stereographic point, the red dot in \fref{fig:ExampleStereo}, so that it corresponds to the south pole of the transformed sphere. We then use stereographic projection from the south pole onto the tangent plane at the north pole, via the transformation
\begin{equation}\label{eq:steCaseC}
(x',y',z')  \mapsto \lp( \frac{Rx'}{R+z'} , \frac{Ry'}{R+z'} \rp).
\end{equation}
This particular stereographic projection was chosen to improve the visibility of certain features of the intersection sets specific to case~\textbf{C}.

\begin{figure}
\centering
\includegraphics{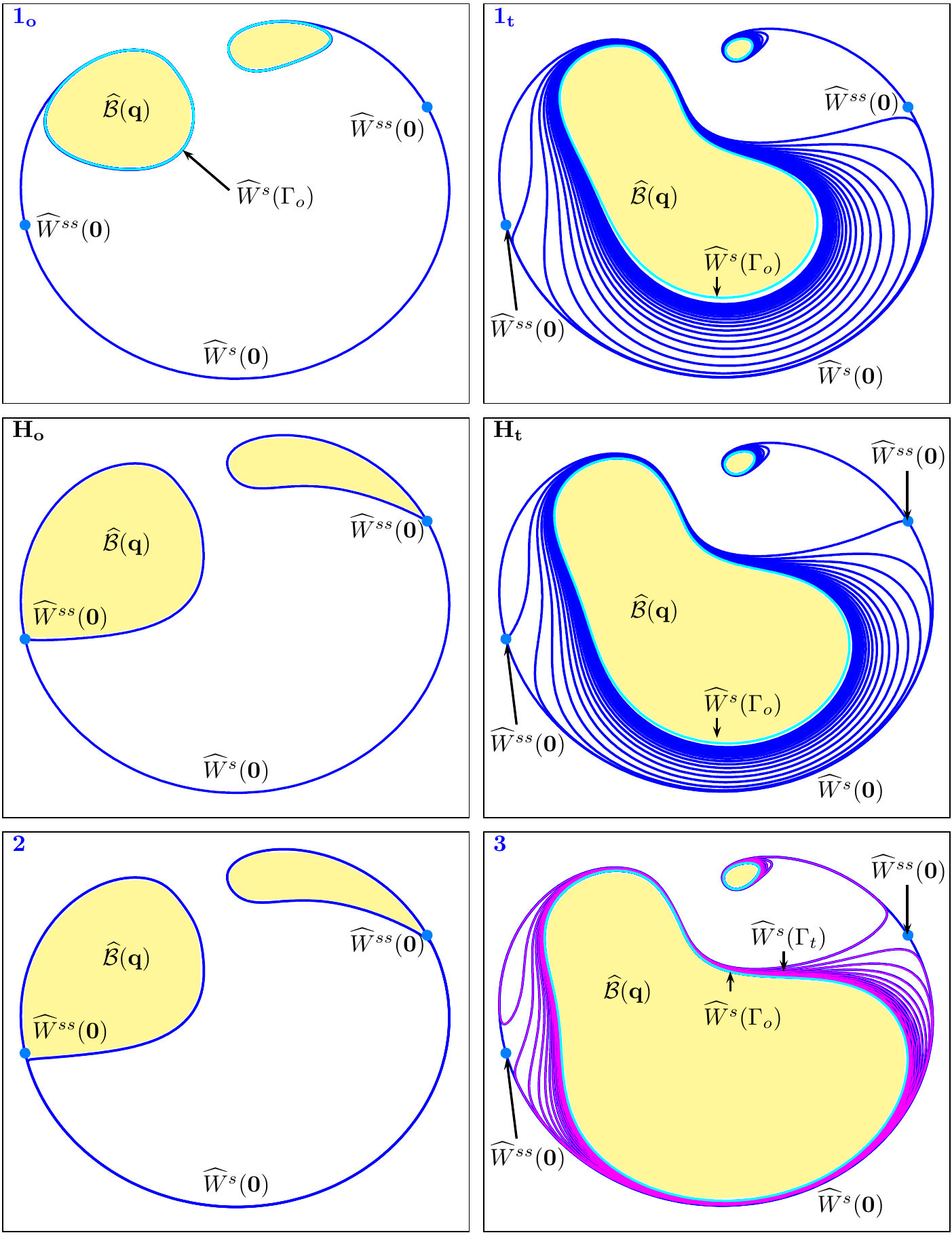}
\caption{Stereographic projections of the intersection sets of the invariant manifolds with $\mS^*$ in regions \mBlue{1}--\mBlue{3} and at the homoclinic bifurcations $\mathbf{H_o}$ and $\mathbf{H_t}$ near the inclination flip point $\mathbf{C_I}$. The left and right columns show the transition through the homoclinic bifurcations $\mathbf{H_o}$ and $\mathbf{H_t}$, respectively.  Shown are $\widehat{W}^s(\mathbf{0})$ (dark-blue curves), $\widehat{W}^{ss}(\mathbf{0})$ (light-blue dots), $\widehat{W}^s(\Gamma_o)$ (cyan curve), $\widehat{W}^s(\Gamma_t)$ (purple curve) and $\widehat{\mathcal{B}}(\mathbf{q})$ (shaded yellow region). See \cref{tab:Inc1} for the respective parameter values.} 
\label{fig:FirstStereo}
\end{figure}  

\Fref{fig:FirstStereo} shows the same series of phase portraits from \cref{fig:Easy} in terms of their intersection sets of stable manifolds with $\mS^*$. The left column of \cref{fig:FirstStereo} illustrates the transition through $\mathbf{H_o}$ and the right column through $\mathbf{H_t}$.  Starting from panel~\mBlue{1_o} in \fref{fig:FirstStereo}, note how the intersection set $\widehat{W}^s(\Gamma_o)$ (cyan curve) is made up of two topological circles that bound the disconnected basin $\widehat{\mathcal{B}}(\mathbf{q})$, that is, $\partial \widehat{\mathcal{B}}(\mathbf{q}) = \widehat{W}^s(\Gamma_o)$. The intersection $\widehat{W}^s(\mathbf{0})$ (blue curve) is a single curve that accumulates in a spiralling manner backward in time onto $\widehat{W}^s(\Gamma_o)$. The accumulation is as expected: the $\lambda$-lemma implies that the transversal heteroclinic orbit from $\Gamma_o$ to $\mathbf{0}$ causes $\widehat{W}^s(\textbf{0})$ to spiral around $\widehat{W}^s(\Gamma_o)$. Note that these phenomena are only observed if the sphere $\mS^*$ is chosen small enough. At the moment of the bifurcation, shown in panel $\mathbf{H_o}$, the curve $\widehat{W}^s(\textbf{0})$ closes back on itself along $\widehat{W}^{ss}(\textbf{0})$ and $\widehat{W}^s(\Gamma_o)$ disappears as $\Gamma_o$ becomes the homoclinic orbit $\Gamma_{\rm hom}$; the basin $\widehat{\mathcal{B}}(\mathbf{q})$ is now bounded by a subset of $\widehat{W}^s(\mathbf{0})$, that is, $\partial \widehat{\mathcal{B}}(\mathbf{q}) \subset \widehat{W}^s(\mathbf{0})$. Hence, at $\mathbf{H_o}$, the boundary of the basin of attraction of $\mathbf{q}$ is contained in $W^s(\mathbf{0})$. As soon as we enter region~\mBlue{2}, the entire intersection set $\widehat{W}^s(\mathbf{0})$ becomes the boundary of $\widehat{\mathcal{B}}(\mathbf{q})$, which is now a connected region that is topologically equivalent to an open disk; see panel~\mBlue{2}. 

In panel~\mBlue{1_t} of \fref{fig:FirstStereo}, one sees the same topological configuration as in panel~\mBlue{1_o}, because both phase portraits are from the same region; however, note how only one end of $\widehat{W}^s(\mathbf{0})$ gets close to both points in $\widehat{W}^{ss}(\mathbf{0})$ in panel~\mBlue{1_t}. At the nonorientable homoclinic bifurcation $\mathbf{H_t}$, the intersection set $\widehat{W}^s(\mathbf{0})$ connects back on itself at $\widehat{W}^{ss}(\mathbf{0})$ in a different way, as is illustrated in panel~$\mathbf{H_t}$ of \fref{fig:FirstStereo}. Here, $\widehat{W}^s(\mathbf{0})$ does not bound two open regions as in panel~$\mathbf{H_o}$. Instead, at $\mathbf{H_t}$ the two segments that form $\widehat{W}^s(\mathbf{0})$ accumulate onto the two topological circles $\widehat{W}^s(\Gamma_o)$ due to the persistence of the heteroclinic orbit from $\Gamma_o$ to $\mathbf{0}$.  The nature of $\widehat{\mathcal{B}}(\mathbf{q})$ and its boundary is unchanged from panel~\mBlue{1_o}. As one transitions to region~\mBlue{3}, the homoclinic orbit $\mathbf{\Gamma_{\rm hom}}$ becomes the periodic orbit $\Gamma_t$ and the intersection set $\widehat{W}^s(\Gamma_t)$ (purple curve) is a single curve that accumulates at both ends onto $\widehat{W}^s(\Gamma_o)$ backward in time, due to the existence of a heteroclinic orbit from $\Gamma_o$ to $\Gamma_t$.  As mentioned before, there are infinitely many heteroclinic orbits from $\Gamma_o$ to $\mathbf{0}$ in region \mBlue{3}. Since $\mS^*$ is a sufficiently small sphere that is transverse to $\widehat{W}^s(\mathbf{0})$, there exist infinitely many curves in $\widehat{W}^s(\mathbf{0})$. One curve in $\widehat{W}^s(\mathbf{0})$ accumulates onto a single topological circle of $\widehat{W}^s(\Gamma_o)$, and all other cuves accumulate onto both topological circles of $\widehat{W}^s(\Gamma_o)$. These sets of curves get arbitrary close to $\widehat{W}^s(\Gamma_t)$ from both sides; we say that $\widehat{W}^s(\Gamma_t)$ is a \textbf{geometric accumulation} curve of $\widehat{W}^s(\mathbf{0})$. In panel~\mBlue{3} of \fref{fig:FirstStereo} we only show three of these infinitely many intersection curves of $\widehat{W}^s(\mathbf{0})$, namely, the outer curve that accumulates onto the single topological circle of $\widehat{W}^s(\Gamma_o)$, and two other curves that track $\widehat{W}^s(\Gamma_t)$ on both sides. We remark that the intersection sets shown in \fref{fig:FirstStereo} are homotopic to the intersection sets shown in \cite{And1} for the transition through $\mathbf{H_o}$ and $\mathbf{H_t}$ in case~\textbf{B}, provided one contracts $\widehat{W}^{ss}(\Gamma^a_o)$ to $\widehat{W}^s(\mathbf{q})$ through $\mathbf{H_o}$ and blows up $\widehat{W}^s(\mathbf{q})$ to a closed curve through $\mathbf{H_t}$ in case~\textbf{B}. 

The fact that $\Gamma_o\rightarrow \Gamma_t \rightarrow \mathbf{0}$ implies the existence of infinitely many structurally stable heteroclinic orbits from $\Gamma_o$ to $\mathbf{0}$ is going to be a recurrent phenomenon for different equilibria and saddle periodic orbits. For ease of exposition, we summarize the consequences in general terms:
\begin{prop} \label{rem:lam}
Let $A, B$ and $C$ be hyperbolic saddle equilibria or periodic orbits of system~\cref{eq:san}. If $A \rightarrow B \rightarrow C$ then there exist infinitely many structurally stable orbits from $A$ to $C$. Furthermore, if $W^s(A)$ intersects $\mS^*$ transversally then there exist infinitely many curves in $\widehat{W}^s(C)$. Each curve in $\widehat{W}^s(C)$ accumulates onto $\widehat{W}^s(A)$ backward in time, and $\widehat{W}^s(B)$ is a geometric accumulation curve of $\widehat{W}^s(C)$. 
\end{prop}
\begin{proof}
The proof follows from the $\lambda$-lemma and is a variation of the proof given in \cite{And1} for the case $\mathbf{q} \rightarrow \Gamma_t \rightarrow \mathbf{0}$.
\end{proof}

\section{Cascades of homoclinic and  heteroclinic bifurcations}
\label{sec:Cascades}

We now focus our attention on the invariant manifolds of system~\cref{eq:san} during the cascades of homoclinic and heteroclinic bifurcations, which we study and illustrate along the line $\alpha=0.5$ in the $(\alpha,\mu)$-plane. 

\subsection{Transition through $\mathbf{^2H_t}$}

\begin{figure}
\centering
\includegraphics{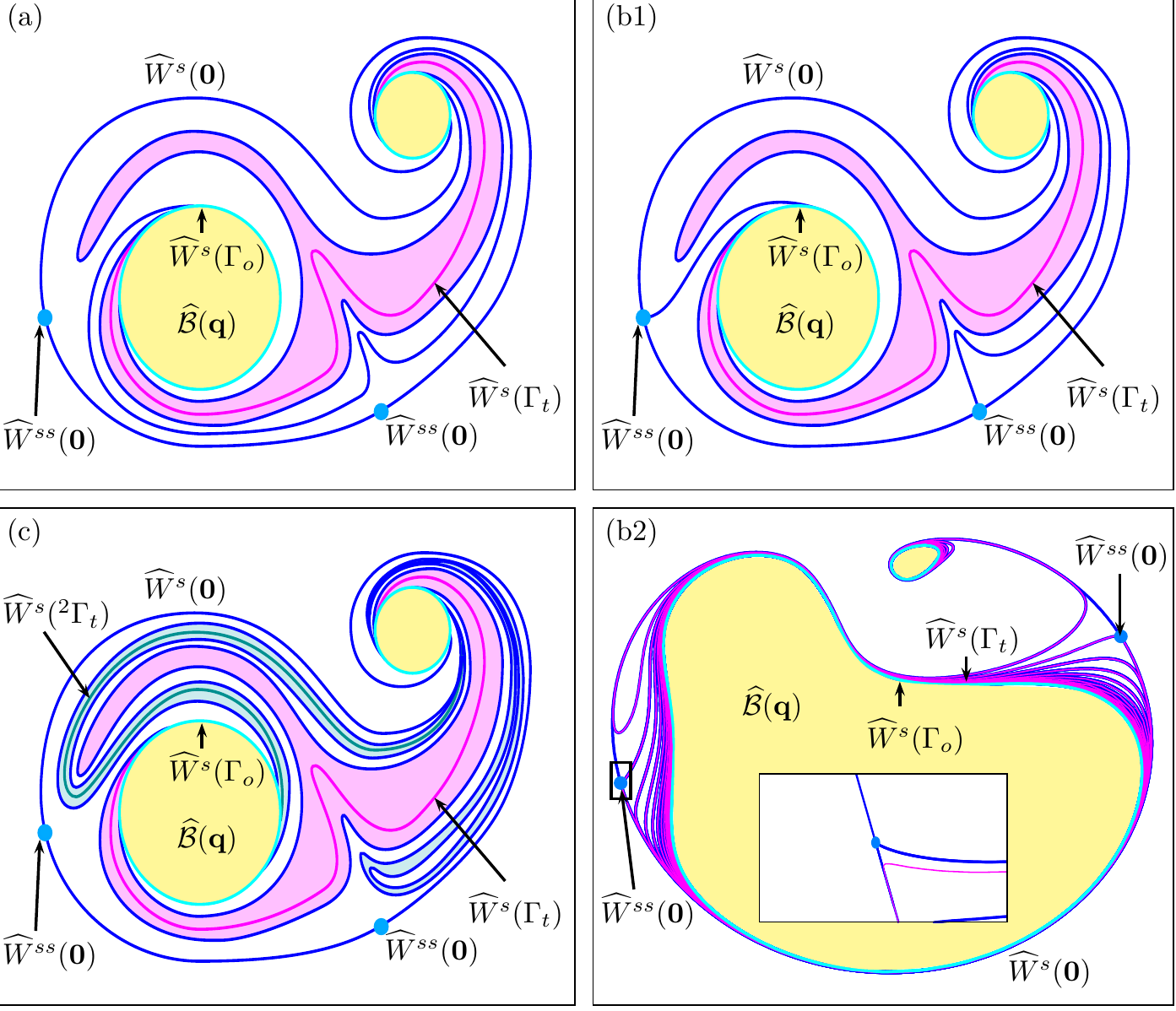}
\caption{Stereographic projection of the intersection sets of the invariant manifolds with $\mS^*$ before, at, and after the moment of the codimension-one homoclinic bifurcation $\mathbf{^2H_t}$ in panels~(a), (b1) and (b2), and (c), respectively.  Panel (a) shows a sketch of the stereographic projection in region~\mBlue{3}; panels~(b1) and (b2) show a sketch and the computed projections of system~\cref{eq:san} at $\mathbf{^2H_t}$, respectively; and panel (c) shows a sketch after $\mathbf{^2H_t}$.  The color code and labeling of the regions is as in \cref{fig:FirstStereo}. Panel~(b2) is for parameter values $(\alpha, \mu)=(0.5, -0.002880267)$.} 
\label{fig:SecondHt}
\end{figure}  

Starting from region~\mBlue{3}, we first encounter the codimension-one nonorientable homoclinic bifurcation $^2\mathbf{H_t}$ which is shown in \fref{fig:SecondHt} as a stereographic projection of the stable invariant objects on $\mS^*$ that exist in phase space. Panel~(a) shows a sketch of the intersection sets in region~\mBlue{3} shown in \fref{fig:FirstStereo}. The pink region indicates the relative location of the infinitely many curves in $\widehat{W}^s(\mathbf{0})$ that accumulate geometrically onto $\widehat{W}^s(\Gamma_t)$. At the moment of the codimension-one homoclinic bifurcation $^2\mathbf{H_t}$, the outer-most curve in $\widehat{W}^s(\mathbf{0})$ that connects both topological circles of $\widehat{W}^s(\Gamma_o)$ splits into two curves that close on $\widehat{W}^s(\mathbf{0})$ at the two points $\widehat{W}^{ss}(\mathbf{0})$; see the sketch in panel~(b1) of \cref{fig:SecondHt} and the associated computed phase portrait on $\mS^*$ in panel~(b2). The bifurcation $^2\mathbf{H_t}$ gives rise to the saddle periodic orbit $^2\Gamma_t$, and we find the following transversal connections: $\Gamma_o\rightarrow {^2\Gamma_t} \rightarrow \mathbf{0}$, $\Gamma_o\rightarrow {\Gamma_t} \rightarrow {^2\Gamma_t}$ and $\Gamma_o\rightarrow {^2\Gamma_t} \rightarrow \Gamma_t$. Panel~(c) illustrates the situation by way of a representative sketch. Note how the two curves of $\widehat{W}^s(\mathbf{0})$ that end at $\widehat{W}^{ss}(\mathbf{0})$ in panel~(b1) give rise to the two light-green regions, which represent the accumulation onto $\widehat{W}^s(^2\Gamma_t)$ (dark-cyan curves) of infinitely many curves in $\widehat{W}^s(\mathbf{0})$ and $\widehat{W}^s(\Gamma_t)$ due to the existence of $\Gamma_o\rightarrow {^2\Gamma_t} \rightarrow \mathbf{0}$ and $\Gamma_o\rightarrow {^2\Gamma_t} \rightarrow {\Gamma_t}$; we refer to \cref{rem:lam}. Both circles in $\widehat{W}^{s}(\Gamma_o)$ bound one of these two regions, but only one circle in $\widehat{W}^{s}(\Gamma_o)$ bounds the other region. Furthermore, since $\Gamma_o\rightarrow {\Gamma_t} \rightarrow {\mathbf{0}}$ and $\Gamma_o\rightarrow {\Gamma_t} \rightarrow {^2\Gamma_t}$, the pink region must contain infinitely many curves of $\widehat{W}^{s}(^2\Gamma_t)$, and infinitely many curves in $\widehat{W}^{s}(\mathbf{0})$ also accumulate geometrically onto these curves.  The existence of both pink and light-green regions is important. As $\mu$ decreases, the outer-most curve from the region that is closest to $\widehat{W}^{ss}(\mathbf{0})$ will meet $\widehat{W}^{ss}(\mathbf{0})$ at a certain value $\mu^*$; depending on whether this is a curve from $\widehat{W}^{s}(\mathbf{0})$, $\widehat{W}^{s}(\Gamma_t)$ or $\widehat{W}^{s}(^2\Gamma_t)$, system~\cref{eq:san} exhibits a particular corresponding codimension-one homoclinic or heteroclinic bifurcation. Furthermore, as the infinitely many curves in $\widehat{W}^{s}(\mathbf{0})$ accumulate onto either $\widehat{W}^{s}(\Gamma_t)$ or $\widehat{W}^{s}(^2\Gamma_t)$, infinitely many homoclinic bifurcations must occur before a codimension-one heteroclinic bifurcation can take place. Whenever a homoclinic bifurcation occurs, the above sequence of different kinds of bifurcations occurs again; that is, an additional saddle periodic orbit is created, accompanied by two regions of accumulation with respect to the stable manifold of the new saddle periodic orbit. Moreover, an accumulation of intersection curves from this stable manifold is created inside the already existing regions. In addition, the closer $\widehat{W}^{s}(\mathbf{0})$ lies to $\widehat{W}^{s}(\Gamma_t)$ or $\widehat{W}^{s}(^2\Gamma_t)$, the higher the number of rotations are for both the homoclinic orbit and the bifurcating saddle periodic orbit.

\subsection{Transition through $\myEtoP{\mathbf{0}}{\Gamma_t}{\Gamma_o}$}

\begin{figure}
\centering
\includegraphics{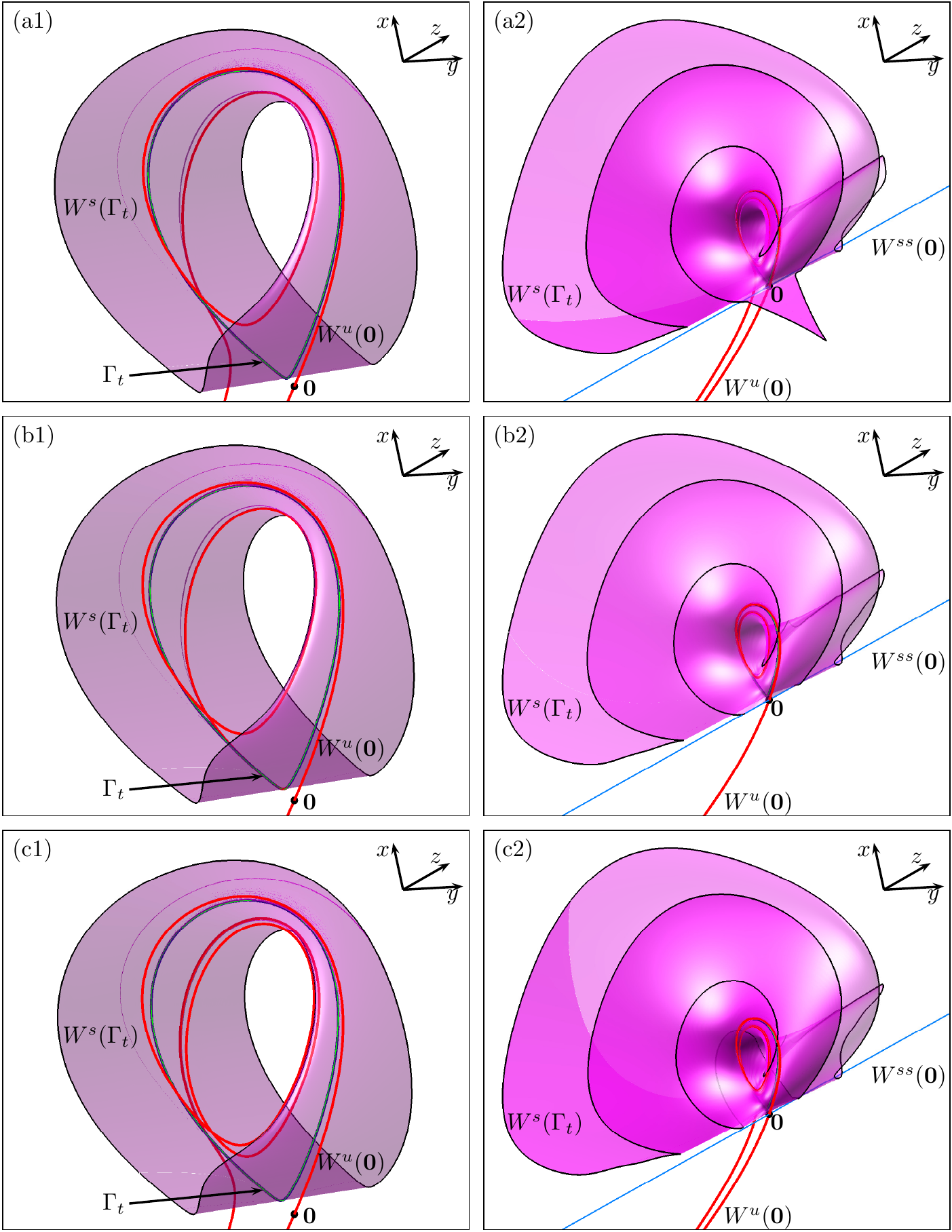}
\caption{The stable manifold $W^s(\Gamma_t)$ (purple surface) of the nonorientable saddle periodic orbit $\Gamma_t$ (green curve), together with $W^u(\mathbf{0})$ (red curve) and $W^{ss}(\mathbf{0})$ (blue curve) before, at, and after the moment of the codimension-one heteroclinic bifurcation $\myEtoP{\mathbf{0}}{\Gamma_t}{\Gamma_o}$ of system~\cref{eq:san} in rows~(a), (b) and (c), respectively.  The left column shows $W^s(\Gamma_t)$ in a tubular neighborhood of $\Gamma_t$, and the right column shows a larger part of $W^s(\Gamma_t)$.  Rows~(a), (b) and (c) are for $(\alpha, \mu)=(0.5, -0.002)$, $(\alpha, \mu)=(0.5, -0.002880324)$ and $(\alpha, \mu)=(0.5, -0.0035)$, respectively. See also the accompanying animation ({\color{red} GKO\_Cflip\_animatedFig11.gif}).} 
\label{fig:tranEtoPNon} 
\end{figure}  

We illustrate in \cref{fig:tranEtoPNon} the transition through the codimension-one heteroclinic bifurcation $\myEtoP{\mathbf{0}}{\Gamma_t}{\Gamma_o}$ and its effect on the reorganization of the stable manifold $W^s(\Gamma_t)$ (purple surface) in phase space. The left column shows $W^s(\Gamma_t)$ in a tubular neighborhood of $\Gamma_t$, showing that $W^s(\Gamma_t)$ is topologically a M\"obius band. The right column~(2) shows a larger part of $W^s(\Gamma_t)$. Panels~(a1) and (a2) show $W^s(\Gamma_t)$, at $\mu=-0.002$, before the bifurcation $\myEtoP{\mathbf{0}}{\Gamma_t}{\Gamma_o}$ and the associated first homoclinic cascade; panels~(b1) and (b2) are at the moment of the bifurcation, when $\mu \approx -0.002880324$; and panels~(a3) and (b3) are for $\mu=-0.0035$, in region~\mBlue{4} past the first cascade; see \cref{fig:Second}.

We first focus on the local behavior in a tubular neighborhood. In panel~(a1), the unstable manifold $W^u(\mathbf{0})$ (red curve) makes one rotation close to $\Gamma_o$ and one close to $\Gamma_t$ and then escapes to infinity. As discussed before, a sequence of homoclinic bifurcations must occur before the bifurcation $\myEtoP{\mathbf{0}}{\Gamma_t}{\Gamma_o}$ is exhibited by system~\cref{eq:san}. During this sequence, the unstable manifold $W^u(\mathbf{0})$ gains turns around $W^s(\Gamma_t)$, thus increasing $\zeta$, until its limiting case at the bifurcation $\myEtoP{\mathbf{0}}{\Gamma_t}{\Gamma_o}$. However, note how the first loop of $W^u(\mathbf{0})$ is still close to $\Gamma_o$ before accumulating onto $\Gamma_t$; this situation is the same as in panel~(a1) of \fref{fig:ConfigHet}.  After $\myEtoP{\mathbf{0}}{\Gamma_t}{\Gamma_o}$, a sequence of homoclinic bifurcations occur which decreases the winding number of $W^u(\mathbf{0})$ around $W^s(\Gamma_t)$, thus decreasing $\zeta$. Finally in region~\mBlue{4}, the unstable manifold $W^u(\mathbf{0})$ makes two rotation close to $\Gamma_o$ and one again close to $\Gamma_t$ before escaping to infinity; see panel~(c1). In particular, the orientation of the homoclinic bifurcation of $\mathbf{0}$ that follows, depends on how $W^u(\mathbf{0})$ passes near $\Gamma_o$ and $\Gamma_t$. 

We gain a deeper understanding by studying a larger portion of $W^s(\Gamma_t)$ as shown in the right column of \fref{fig:tranEtoPNon}. Panel~(a2) illustrates how $W^s(\Gamma_t)$ spirals around $W^s(\Gamma_o)$ due to the presence of heteroclinic orbits from $\Gamma_o$ to $\Gamma_t$; the outer rim of the computed part of $W^s(\Gamma_t)$ is highlighted as the black curve. Note how part of $W^s(\Gamma_t)$ gets close to $W^{ss}(\mathbf{0})$ before starting to spiral again. In panel~(b2), this part of $W^s(\Gamma_t)$ has disappeared, and the existence of the heteroclinic orbit $\myEtoP{\mathbf{0}}{\Gamma_t}{\Gamma_o}$ forces $W^s(\Gamma_t)$ to accumulate onto $W^{ss}(\mathbf{0})$ backward in time.  After the bifurcation $\myEtoP{\mathbf{0}}{\Gamma_t}{\Gamma_o}$, the manifold $W^s(\Gamma_t)$ no longer accumulates onto $W^{ss}(\mathbf{0})$; see panel~(c2). Observe that $W^s(\Gamma_t)$ in panel~(c2) seems identical to $W^s(\Gamma_t)$ in panel~(b2), but it exhibits one additional loop.  This extra loop would spread out if a much larger portion of $W^s(\Gamma_t)$ were computed, and it would get close to $W^{ss}(\mathbf{0})$ in the same way as in panel~(a2).

\begin{figure}
\centering
\includegraphics{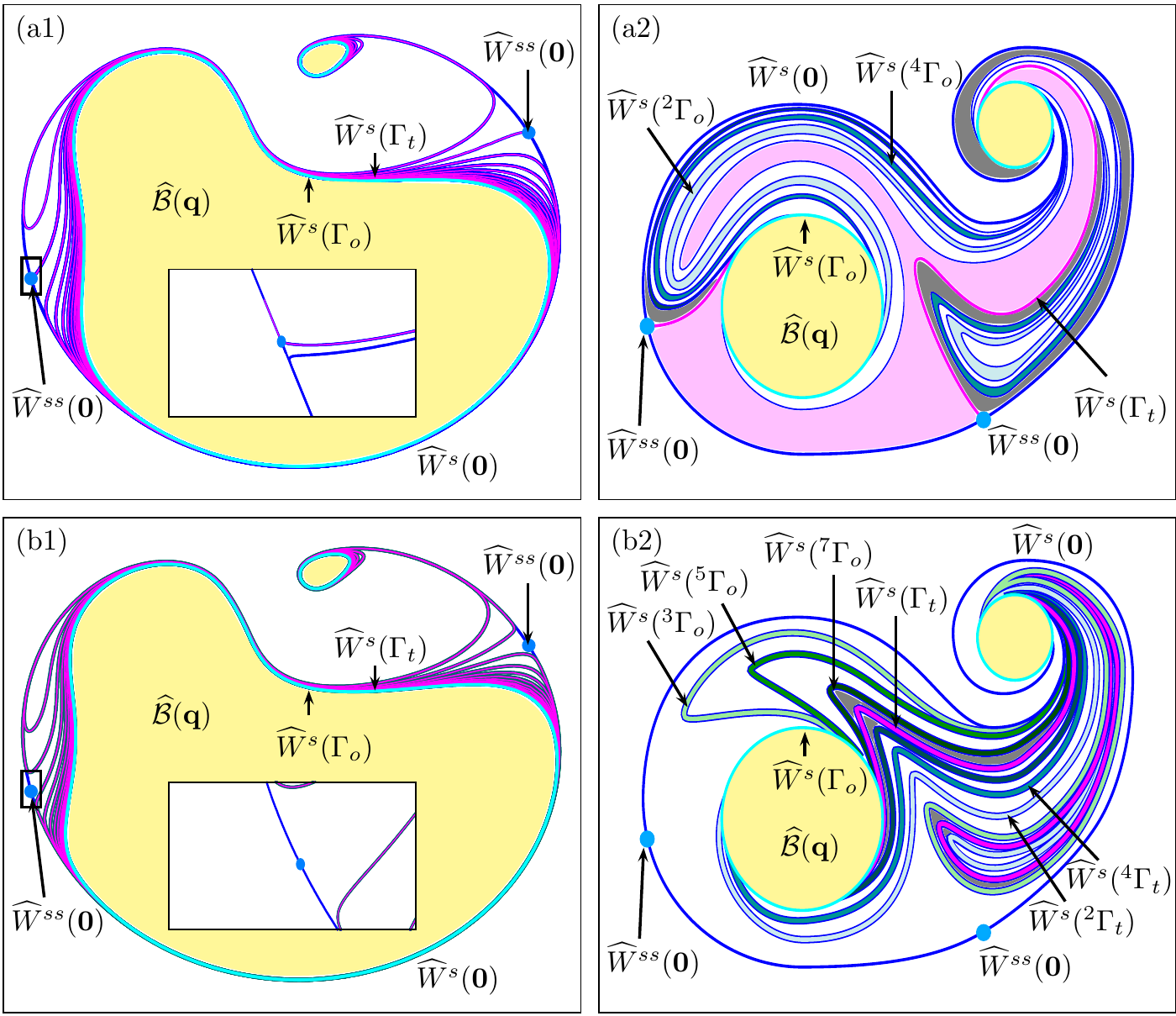}
\caption{Stereographic projection of the intersection sets of the invariant manifolds with $\mS^*$ at the codimension-one heteroclinic bifurcation $\mathbf{Q}^{\Gamma_t}_{\mathbf{0}}$, in region~\mBlue{4}, and at the codimension-one heteroclinic bifurcation $\mathbf{Q}^{\Gamma_o}_{\mathbf{0}}$ in rows~(a) and (b), respectively.  The left column shows the computed projections and the right column corresponding sketches.  The color code and labeling of the regions is the same as in \cref{fig:FirstStereo}.  Panels~(a1), and (b1) are for $(\alpha, \mu)=(0.5, -0.002880324)$ and $(\alpha, \mu)=(0.5, -0.0035)$, respectively.} 
\label{fig:StereoEtoP} 
\end{figure}  

\Cref{fig:StereoEtoP} illustrates the transition to region~\mBlue{4} on the level of the intersection sets of the invariant objects with $\mS^*$. Panels~(a1) and (a2) are stereographic projections of the intersection sets with $\mS^*$ at the moment of the first codimension-one heteroclinic bifurcation $\myEtoP{\mathbf{0}}{\Gamma_t}{\Gamma_o}$ of $\Gamma_t$. At the moment of the bifurcation the curve $\widehat{W}^{s}(\Gamma_t)$ (purple) splits into two curves that meet at the two points in $\widehat{W}^{ss}(\mathbf{0})$; see panel~(a1). This phenomenon is illustrated more clearly in the sketch shown in panel~(a2), where the accumulation of curves in $\widehat{W}^s(\Gamma_t)$ is represented by the pink shading. As mentioned before, each homoclinic bifurcation creates a saddle periodic orbit that gives rise to two new accumulation regions in the intersection set with $\mS^*$; these are shaded in panel~(a2) with different tones of cyan and the label indicates the intersection curve of the principal stable manifold. Only two pairs of accumulation regions are indicated in panel~(a2).  The gray region represents the infinitely many regions created in further homoclinic bifurcations and specifies their relative location in the stereographic projection. It is clear from panels~(a1) and (a2) of \fref{fig:StereoEtoP} that, after crossing of $\myEtoP{\mathbf{0}}{\Gamma_t}{\Gamma_o}$, there is another cascade of homoclinic bifurcations that terminates on the homoclinic bifurcation associated with the blue curve in $\widehat{W}^s(\mathbf{0})$ that bounds the pink region in panel~(a2); this last homoclinic bifurcation corresponds to the last bifurcation that system~\cref{eq:san} exhibits before transitioning into region~\mBlue{4} in \cref{fig:Second}. Panels~(b1) and (b2) in \cref{fig:StereoEtoP} show the situation in region~\mBlue{4} on the level of the stereographic projections.  In particular, each accumulation region is represented by the main intersection curve of the stable manifold of the saddle periodic orbit that exists on it. Observe how one of the curves of $\widehat{W}^s(\Gamma_t)$ and, thus, the corresponding accumulation region accumulate onto a single topological circle of $\widehat{W}^s(\Gamma_o)$, while the other curve accumulates on both topological circles of $\widehat{W}^s(\Gamma_o)$. This is true for each of the accumulation regions of the different saddle periodic orbits, while they accumulate at the same time onto $\widehat{W}^s(\Gamma_t)$ (gray region). 

\subsection{Transition through $\mathbf{Q}_{\mathbf{0}}^{\Gamma_o}$}

\begin{figure}
\centering
\includegraphics{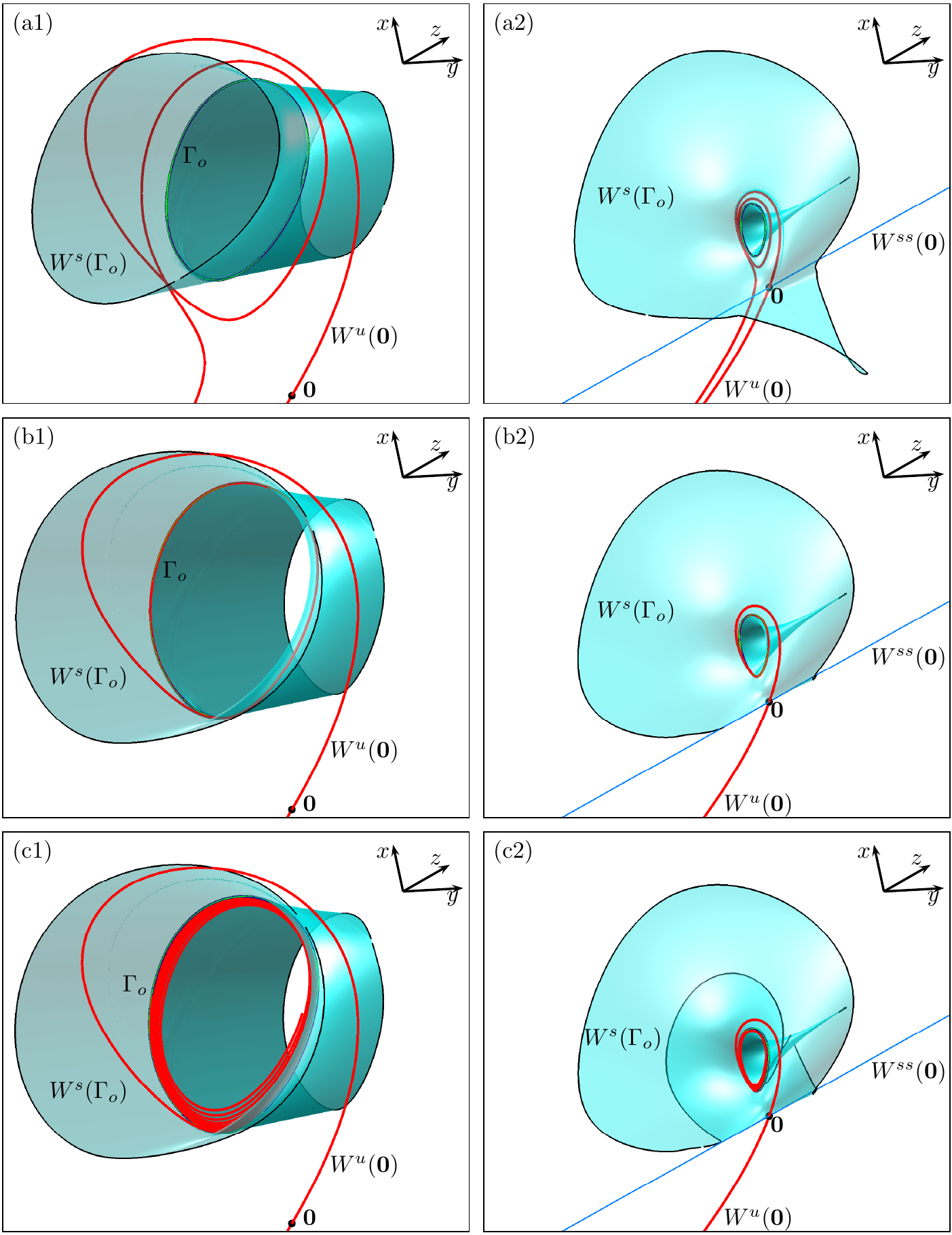}
\caption{The stable manifold $W^s(\Gamma_o)$ (cyan surface) of the orientable saddle periodic orbit $\Gamma_o$ (purple curve), together with $W^u(\mathbf{0})$ (red curve) and $W^{ss}(\mathbf{0})$ before, at, and after the moment of the codimension-one heteroclinic bifurcation $\mathbf{Q}^{\Gamma_o}_{\mathbf{0}}$ of system~\cref{eq:san} in rows~(a), (b) and (c), respectively.  The left column shows $W^s(\Gamma_t)$ in a tubular neighborhood of $\Gamma_t$, and the right column shows a larger part of $W^s(\Gamma_t)$.  Rows~(a), (b) and (c) are for $(\alpha, \mu)=(0.5, -0.002)$, $(\alpha, \mu)=(0.5, -0.004861805)$ and $(\alpha, \mu)=(0.5, -0.005)$, respectively. See also the accompanying animation ({\color{red} GKO\_Cflip\_animatedFig13.gif}).} 
\label{fig:tranEtoPOri}
\end{figure}  

The final step at the end of all cascades is the heteroclinic bifurcation $\mathbf{Q}_{\mathbf{0}}^{\Gamma_o}$. The two-dimensional stable manifold $W^s(\Gamma_o)$ does not interact with any saddle invariant object, and it retains its topological properties during the cascades of homoclinic and heteroclinic bifurcations. However, this changes as soon as $W^s(\Gamma_o)$ intersects $W^u(\mathbf{0})$ at the bifurcation $\mathbf{Q}_{\mathbf{0}}^{\Gamma_o}$, which is illustrated in \fref{fig:tranEtoPOri}.  Specifically, the left column shows $W^s(\Gamma_o)$ (cyan surface), together with $W^u(\mathbf{0})$ (red curve), in a tubular neighborhood of $\Gamma_o$ and the right column shows larger parts of these manifolds. Panels~(a1) and (a2) show, at $\mu=-0.002$, the situation before the bifurcation $\mathbf{Q}_{\mathbf{0}}^{\Gamma_o}$, that is, during the homoclinic and heteroclinic cascades; panels~(b1) and (b2) are for $\mu \approx -0.004861805$ at the moment of the bifurcation $\mathbf{Q}^{\Gamma_o}_{\mathbf{0}}$; and panels~(a3) and (b3) are for $\mu=-0.005$, after the bifurcation. It seems that, in the vicinity of $\Gamma_o$, the transition through $\mathbf{Q}_{\mathbf{0}}^{\Gamma_o}$ does not manifest itself as any topological change for $W^s(\Gamma_o)$, apart from its relative position with respect to $W^u(\mathbf{0})$. Hence, in panels~(a1), (b1) and (c1), one sees $W^u(\mathbf{0})$ outside, at, and inside the topological cylinder $W^s(\Gamma_o)$, respectively.  The difference becomes apparent only when we consider a larger portion of phase space. In panel~(a2), the stable manifold $W^s(\Gamma_o)$ is still a topological cylinder, but note how part of $W^s(\Gamma_o)$ lies close to $W^{ss}(\mathbf{0})$. The moment $\mathbf{Q}_{\mathbf{0}}^{\Gamma_o}$ occurs, in panel~(b2), one trajectory of the unstable manifold manifold $W^u(\mathbf{0})$ lies inside $W^s(\Gamma_o)$ which forces $W^s(\Gamma_o)$ to accumulate onto $W^{ss}(\mathbf{0})$ backward in time. After the bifurcation $\mathbf{Q}_{\mathbf{0}}^{\Gamma_o}$, as shown in panel~(c2) of \cref{fig:tranEtoPOri}, the lower half of $W^s(\Gamma_o)$ starts to spiral around its upper half. The spiralling accumulation is due to the existence of structurally stable homoclinic orbits of $\Gamma_o$. Note in panel~(c2) that one branch of $W^u(\mathbf{0})$ lies in the basin $\mathcal{B}(\mathbf{q})$ and the other one lies in the positive invariant set $V$, as defined in \cref{sec:topInv}. Hence, the winding number $\zeta$ is infinite and there cannot be any homoclinic or heteroclinic bifurcations occurring after $\mathbf{Q}_{\mathbf{0}}^{\Gamma_o}$. This implies that $\mathbf{Q}_{\mathbf{0}}^{\Gamma_o}$ marks the end of the homoclinic and heteroclinic cascade in the $(\alpha,\mu)$-plane. 

\subsection{Computation of bifurcation sequences}
\label{sec:CompCascades}

To determine different homoclinic and heteroclinic cascades, and to understand their organization in the parameter plane, we computed representative bifurcation points and their corresponding $\mu$-values along the line $\alpha=0.5$.  These are listed in \cref{tab:Sequence}, which also lists the sequence of further bifurcations past these cascades. We first discuss the global bifurcations up to $\mathbf{Q}_{\mathbf{0}}^{\Gamma_o}$ and then the remaining bifurcations.

As $\mu$ decreases from the primary homoclinic bifurcation $\mathbf{H_t}$ at $\mu=0$, we encounter homoclinic and heteroclinic cascades in a particular order, followed by saddle-node and period-doubling bifurcations of periodic orbits. As mentioned before, a saddle periodic orbit is created for each homoclinic bifurcation that occurs as $\mu$ decreases; e.g., the saddle periodic orbits $\Gamma_t$ and $^2\Gamma_t$ are created from the homoclinic bifurcations $\mathbf{H_t}$ and $^2\mathbf{H_t}$, respectively.  We now employ a similar notation as used for the heteroclinic bifurcations to denote the subsequent homoclinic bifurcations where the corresponding orbit makes rotations around certain saddle periodic orbits before accumulating backward and forward in time to $\mathbf{0}$. For ease of exposition, we append this information in their corresponding labels. For example, $\myHom{o}{n\Gamma_o,m\Gamma_t}$ represents the codimension-one orientable homoclinic bifurcation where the corresponding orbit makes $n$ turns around $\Gamma_o$ and then $m$ turns around $\Gamma_t$ before accumulating forward in time to $\mathbf{0}$. In particular, the homoclinic bifurcations $\myHom{o}{n\Gamma_o,m\Gamma_t}$ or $\myHom{t}{n\Gamma_o,m\Gamma_t}$, for $n, m \in \N$, in \cref{tab:Sequence} create saddle periodic orbits that make $n+m+1$ loops (the first rotation cannot be assigned to any saddle periodic orbit) and have the same orientation as the homoclinic bifurcation. Heuristically, one can get information about the orientation of the homoclinic bifurcation by looking at the parity of $m$, that is, the number of times the corresponding orbit of the homoclinic bifurcation rotates around $\Gamma_t$ in phase space before converging to $\mathbf{0}$. In general the first turn changes the orientation of the homoclinic bifurcation to nonorientable, and every time a rotation occurs around $\Gamma_t$, the orientation swaps; hence, for odd $m$ the corresponding bifurcation is nonorientable and for even $m$ it is orientable. For example, observe in \cref{tab:Sequence} that the homoclinic cascade of the form $\myHom{t}{\Gamma_o,m\Gamma_t}$, with $m$ even, increases in the number of rotations around $\Gamma_t$ until it accumulates onto $\myEtoP{\mathbf{0}}{\Gamma_t}{\Gamma_o}$; the number of rotations then starts decreasing, with $m$ odd, up to $\myHom{o}{\Gamma_o,\Gamma_t}$.  In particular, $\myHom{o}{\Gamma_o,\Gamma_t}$ is the last bifurcation that occurs before transitioning into region~\mBlue{4}. Note that between $\myHom{o}{\Gamma_o,\Gamma_t}$ and $\myHom{t}{2\Gamma_o}$ there is a larger gap in $\mu$ compared to the previous homoclinic bifurcations.  This is due to the termination and the start of new homoclinic and heteroclinic cascades, respectively. In particular, these two homoclinic bifurcations are related as they create the saddle periodic orbits $^3\Gamma_o$ and $^3\Gamma_t$, which disappear in the saddle-node bifurcation $\mathbf{SNP}_{^3\Gamma_o}$.  Because $\zeta$ remains constant at $\zeta=3$, we also have numerical evidence that in this gap there cannot exist additional homoclinic or heteroclinic bifurcations.  

\begin{table}
\centering
\setlength\extrarowheight{0.75mm}
\begin{tabular}{|l|c|c|l|c|}
\cline{1-2} \cline{4-5}
\textbf{Bifurcation}                                & $\mathbf{\mu}\times [10^{-3}]$                       & \quad & \textbf{Bifurcation}                     & $\mathbf{\mu}\times [10^{-3}]$                       \\ \cline{1-2} \cline{4-5} 
{\color[HTML]{A0522D} $\mathbf{H_t}$}             &   {\color[HTML]{A0522D} 0.0} &           & {\color[HTML]{00A0A0} $\myHom{t}{5\Gamma_o}$}                        & {\color[HTML]{00A0A0} -4.619987} \\ \cline{1-2} \cline{4-5} 
{\color[HTML]{00A0A0} $\mathbf{^2H_t}$}             & {\color[HTML]{00A0A0} -2.880268} &           & {\color[HTML]{9B9B9B} $\myEtoP{\mathbf{0}}{^2\Gamma_t}{5\Gamma_o}$} &{\color[HTML]{9B9B9B}  -4.619993}\\ \cline{1-2} \cline{4-5} 
{\color[HTML]{00A0A0} $\myHom{t}{\Gamma_o,2\Gamma_t}$}     & {\color[HTML]{00A0A0} -2.880324} &           & {\color[HTML]{FF00FF} $\myEtoP{\mathbf{0}}{\Gamma_t}{5\Gamma_o}$} & {\color[HTML]{FF00FF} -4.620120}\\ \cline{1-2} \cline{4-5} 
{\color[HTML]{00A0A0} $\myHom{t}{\Gamma_o,4\Gamma_t}$}     & {\color[HTML]{00A0A0} -2.880324} &           & {\color[HTML]{00A0A0} $\myHom{o}{5\Gamma_o,\Gamma_t}$}  & {\color[HTML]{00A0A0} -4.620128}\\ \cline{1-2} \cline{4-5} 
{\color[HTML]{00A0A0} $\myHom{t}{\Gamma_o,6\Gamma_t}$}     & {\color[HTML]{00A0A0} -2.880324} &           & {\color[HTML]{00A0A0} $\myHom{t}{6\Gamma_o}$}    & {\color[HTML]{00A0A0} -4.704132}\\ \cline{1-2} \cline{4-5} 
{\color[HTML]{00A0A0} $\myHom{t}{\Gamma_o,8\Gamma_t}$}     & {\color[HTML]{00A0A0} -2.880324} &           & {\color[HTML]{00A0A0} $\myHom{o}{6\Gamma_o,\Gamma_t}$} & {\color[HTML]{00A0A0} -4.704235} \\ \cline{1-2} \cline{4-5} 
{\color[HTML]{FF00FF} $\myEtoP{\mathbf{0}}{\Gamma_t}{\Gamma_o}$} & {\color[HTML]{FF00FF} -2.880324} &           & {\color[HTML]{00A0A0} $\myHom{t}{7\Gamma_o}$}  & {\color[HTML]{00A0A0} -4.757563}\\ \cline{1-2} \cline{4-5} 
{\color[HTML]{00A0A0} $\myHom{o}{\Gamma_o,7\Gamma_t}$}     & {\color[HTML]{00A0A0} -2.880324} &           & {\color[HTML]{00A0A0} $\myHom{o}{7\Gamma_o,\Gamma_t}$}     & {\color[HTML]{00A0A0} -4.757636} \\ \cline{1-2} \cline{4-5} 
{\color[HTML]{00A0A0} $\myHom{o}{\Gamma_o,5\Gamma_t}$}     & {\color[HTML]{00A0A0} -2.880324}&           & {\color[HTML]{00A0A0} $\myHom{t}{8\Gamma_o}$}   & {\color[HTML]{00A0A0} -4.792227}\\ \cline{1-2} \cline{4-5} 
{\color[HTML]{00A0A0} $\myHom{o}{\Gamma_o,3\Gamma_t}$}     & {\color[HTML]{00A0A0} -2.880324} &           & {\color[HTML]{00A0A0} $\myHom{o}{8\Gamma_o,\Gamma_t}$}  & {\color[HTML]{00A0A0} -4.792278}\\ \cline{1-2} \cline{4-5} 
{\color[HTML]{00A0A0} $\myHom{o}{\Gamma_o,\Gamma_t}$}     & {\color[HTML]{00A0A0} -2.880325} &           & {\color[HTML]{00A0A0} $\myHom{t}{9\Gamma_o}$}    & {\color[HTML]{00A0A0} -4.815051}\\ \cline{1-2} \cline{4-5} 
{\color[HTML]{00A0A0} $\myHom{t}{2\Gamma_o}$}             & {\color[HTML]{00A0A0} -3.816057} &           & {\color[HTML]{800080} $\mathbf{Q}_{\mathbf{0}}^{\Gamma_o} $} & {\color[HTML]{800080} -4.861805}\\ \cline{1-2} \cline{4-5} 
{\color[HTML]{9B9B9B} $\myEtoP{\mathbf{0}}{^2\Gamma_t}{2\Gamma_o}$} & {\color[HTML]{9B9B9B} -3.816058} &           &{\color[HTML]{3531FF}  $\mathbf{F}$}                        & {\color[HTML]{3531FF} -7.054355} \\ \cline{1-2} \cline{4-5} 
{\color[HTML]{FF00FF} $\myEtoP{\mathbf{0}}{\Gamma_t}{2\Gamma_o}$}         & {\color[HTML]{FF00FF} -3.816233} &           & {\color[HTML]{6665CD} $\mathbf{Tan}_{\Gamma_o}$}   & {\color[HTML]{6665CD} -7.076705} \\ \cline{1-2} \cline{4-5} 
{\color[HTML]{9B9B9B} $\myEtoP{\mathbf{0}}{^2\Gamma_t}{2\Gamma_o,\Gamma_t}$} & {\color[HTML]{9B9B9B} -3.816233} &           & {\color[HTML]{009901} $\mathbf{SNP}_{^{10}\Gamma_o}$} & {\color[HTML]{009901} -7.077572} \\ \cline{1-2} \cline{4-5} 
{\color[HTML]{00A0A0} $\myHom{o}{2\Gamma_o,\Gamma_t}$}     & {\color[HTML]{00A0A0} -3.816234} &           & {\color[HTML]{009901} $\mathbf{SNP}_{^9\Gamma_o}$} & {\color[HTML]{009901} -7.078122} \\ \cline{1-2} \cline{4-5} 
{\color[HTML]{00A0A0} $\myHom{t}{3\Gamma_o}$}             & {\color[HTML]{00A0A0} -4.249463} &           & {\color[HTML]{009901} $\mathbf{SNP}_{^7\Gamma_o}$} & {\color[HTML]{009901} -7.080554} \\ \cline{1-2} \cline{4-5} 
{\color[HTML]{9B9B9B} $\myEtoP{\mathbf{0}}{^2\Gamma_t}{3\Gamma_o}$}         & {\color[HTML]{9B9B9B} -4.249465} &           &{\color[HTML]{009901} $\mathbf{SNP}_{^6\Gamma_o}$} & {\color[HTML]{009901} -7.083202}  \\ \cline{1-2} \cline{4-5} 
{\color[HTML]{FF00FF} $\myEtoP{\mathbf{0}}{\Gamma_t}{3\Gamma_o}$}         & {\color[HTML]{FF00FF} -4.249668}&           & {\color[HTML]{009901} $\mathbf{SNP}_{^5\Gamma_o}$} & {\color[HTML]{009901} -7.088049}\\ \cline{1-2} \cline{4-5} 
{\color[HTML]{9B9B9B} $\myEtoP{\mathbf{0}}{^2\Gamma_t}{3\Gamma_o,\Gamma_t}$} & {\color[HTML]{9B9B9B} -4.249669} &           & {\color[HTML]{009901} $\mathbf{SNP}_{^4\Gamma_o}$} & {\color[HTML]{009901} -7.097747} \\ \cline{1-2} \cline{4-5} 
{\color[HTML]{00A0A0} $\myHom{o}{3\Gamma_o,\Gamma_t}$}     & {\color[HTML]{00A0A0} -4.249669} &           & {\color[HTML]{009901} $\mathbf{SNP}_{^3\Gamma_o}$} & {\color[HTML]{009901} -7.120570} \\ \cline{1-2} \cline{4-5} 
{\color[HTML]{00A0A0} $\myHom{t}{4\Gamma_o}$}             & {\color[HTML]{00A0A0} -4.483178} &           & {\color[HTML]{FE0000} $\mathbf{PD}_{^8\Gamma_t}$} & {\color[HTML]{FE0000} -7.151054}\\ \cline{1-2} \cline{4-5} 
{\color[HTML]{9B9B9B} $\myEtoP{\mathbf{0}}{^2\Gamma_t}{4\Gamma_o}$}         & {\color[HTML]{9B9B9B} -4.483183} &           & {\color[HTML]{FE0000} $\mathbf{PD}_{^4\Gamma_t}$} & {\color[HTML]{FE0000} -7.153300}\\ \cline{1-2} \cline{4-5} 
{\color[HTML]{FF00FF} $\myEtoP{\mathbf{0}}{\Gamma_t}{4\Gamma_o}$}         & {\color[HTML]{FF00FF} -4.483359} &           & {\color[HTML]{FE0000} $\mathbf{PD}_{^2\Gamma_t}$} & {\color[HTML]{FE0000} -7.163762}\\ \cline{1-2} \cline{4-5} 
{\color[HTML]{9B9B9B} $\myEtoP{\mathbf{0}}{^2\Gamma_t}{4\Gamma_o,\Gamma_t}$} & {\color[HTML]{9B9B9B} -4.483360} &           & {\color[HTML]{FE0000} $\mathbf{PD}_{\Gamma_t}$} & {\color[HTML]{FE0000} -7.211185}\\ \cline{1-2} \cline{4-5} 
{\color[HTML]{00A0A0} $\myHom{o}{4\Gamma_o,\Gamma_t}$}     & {\color[HTML]{00A0A0} -4.483359}&           &  {\color[HTML]{009901} $\mathbf{SNP}_{\Gamma_o}$}                    & {\color[HTML]{009901} -7.386406}                                       \\ \cline{1-2} \cline{4-5} 
\end{tabular}
\vspace{2mm}
\caption{Computed $\mu$-values at a selection from the infinitely many bifurcations points along the slice $\alpha=0.5$ of system~\cref{eq:san}. The color of each bifurcation point matches the corresponding curve in \cref{fig:Second}. We also show the heteroclinic bifurcation $\mathbf{Q}_{\mathbf{0}}^{^2\Gamma_t}$ from $\mathbf{0}$ to $^2\Gamma_t$ (gray). The ordering and clustering of these bifurcations is illustrated in \cref{fig:SequenceSlice}.}
\label{tab:Sequence}
\end{table}

We also present in \cref{tab:Sequence} the $\mu$-values of representative codimension-one heteroclinic bifurcations from $\mathbf{0}$ to $^2\Gamma_t$. For example, at the heteroclinic bifurcation $\myEtoP{\mathbf{0}}{^2\Gamma_t}{2\Gamma_o}$, system~\cref{eq:san} exhibits a heteroclinic orbit that connects $\mathbf{0}$ to $^2\Gamma_t$ while rotating twice around $\Gamma_o$. This shows that there exist additional cascades inside the cascade $\myHom{o}{2\Gamma_o,m\Gamma_t}$ with $m$ even, and the same holds for the cascade $\myHom{t}{2\Gamma_o,m\Gamma_t}$ with $m$ odd, as our topological sketch in panel~(b2) of \cref{fig:StereoEtoP} suggests due to the $\lambda$-lemma. In general, there must be cascades of the form $\myHom{o/t}{2\Gamma_o,m\Gamma^*}$, $m \in \N$, where $\Gamma^*$ is a different saddle periodic orbit, accumulating onto heteroclinic bifurcations $\myEtoP{\mathbf{0}}{\Gamma^*}{2\Gamma_o}$. Therefore, there exist entire clusters of cascades. These clusters of cascades are repeated starting from $\myHom{t}{2\Gamma_o}$ and ending at $\myHom{o}{2\Gamma_o,\Gamma_t}$, and again for the pair $\myHom{t}{3\Gamma_o}$ and $\myHom{o}{3\Gamma_o,\Gamma_t}$, and so on.  In general, the center of each cluster corresponds to the heteroclinic bifurcation $\myEtoP{\mathbf{0}}{\Gamma_t}{n\Gamma_o}$, with $n=2,3,4,5,...$, which occurs only once in each cluster. For example, note from \cref{tab:Sequence} that the heteroclinic bifurcation $\myEtoP{\mathbf{0}}{\Gamma_t}{2\Gamma_o}$ lies in between the bifurcations $\myEtoP{\mathbf{0}}{^2\Gamma_t}{2\Gamma_o}$ and $\myEtoP{\mathbf{0}}{^2\Gamma_t}{2\Gamma_o,\Gamma_t}$. Furthermore, the last orientable homoclinic bifurcation from the second cluster, i.e., $\myHom{o}{2\Gamma_o,\Gamma_t}$, is associated with the first nonorientable homoclinic bifurcation in the second cluster, i.e. $\myHom{t}{3\Gamma_o}$, because they create the saddle periodic orbits $^4\Gamma_o$ and $^4\Gamma_t$ that disappear in the saddle-node bifurcation $\mathbf{SNP}_{^4\Gamma_o}$. These two homoclinic bifurcations also form a large region in the bifurcation diagram~\cref{fig:Second} with $\zeta=4$ constant. After $\myHom{o}{2\Gamma_o,\Gamma_t}$ occurs, the rotation of $W^u(\mathbf{0})$ around $\Gamma_t$ moves closer to $\Gamma_o$ as $\mu$ decreases, until it becomes $\myHom{t}{3\Gamma_o}$.  Similarly, the third cluster of cascades ends with the orientable homoclinic bifurcation $\myHom{o}{3\Gamma_o,\Gamma_t}$, which is associated with the first homoclinic bifurcation $\myHom{t}{4\Gamma_o}$, because they create $^5\Gamma_o$ and $^5\Gamma_t$ which disappear in $\mathbf{SNP}_{^5\Gamma_o}$; these two homoclinic bifurcations also form a large region in the bifurcation diagram~\cref{fig:Second}, now with $\zeta=5$. In general, the last orientable homoclinic bifurcation of a cluster and the first nonorientable homoclinic bifurcation in the following cluster of cascades bound a large region in the bifurcation diagram~\cref{fig:Second} with the number of their loops corresponding to the value of $\zeta$.  We are able to compute the homoclinic clusters up to $\myHom{o}{8\Gamma_o,\Gamma_t}$ and $\myHom{t}{9\Gamma_o}$, and detect the heteroclinic bifurcations up to $\myEtoP{\mathbf{0}}{\Gamma_t}{5\Gamma_o}$.  The cascade ends on $\mathbf{Q}_{\mathbf{0}}^{\Gamma_o}$, which marks the start, as $\mu$ decreases, of a regime with transversal homoclinic orbits of $\Gamma_o$, that is, the existence of Smale--horseshoe dynamics \cite{Taken1}. 

\begin{figure}
\centering
\includegraphics[width=15.3cm]{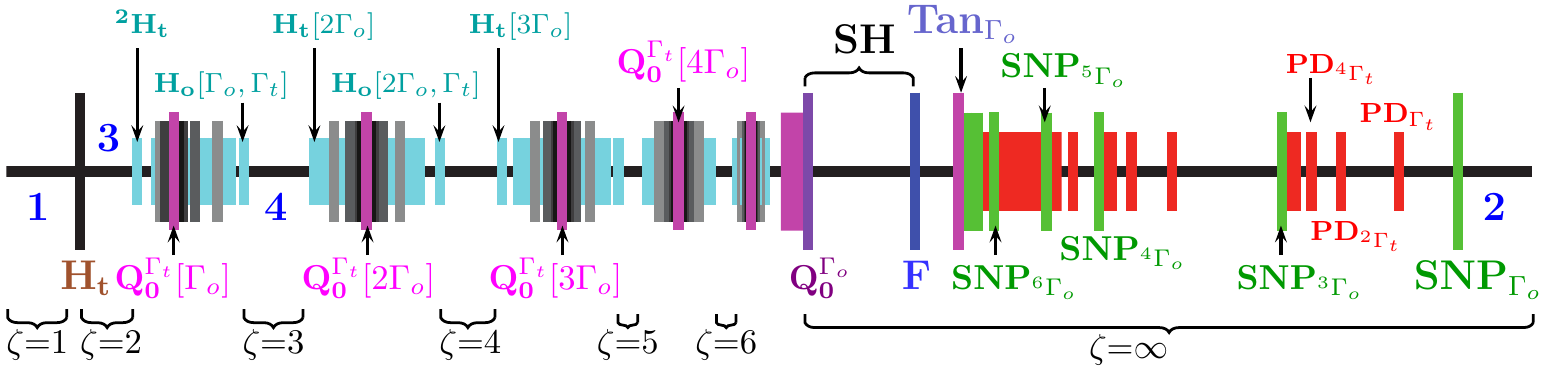} 
\caption{Sketch of the bifurcation sequence along the slice $\alpha=0.5$ for the bifurcation diagram of system~\cref{eq:san} in \cref{fig:Second}. Here, $\mu$ decreases to the right; see also \cref{tab:Sequence}.} 
\label{fig:SequenceSlice}
\end{figure}  

\subsection{Graphical representation of the bifurcation sequence}

\Cref{fig:SequenceSlice} presents a sketch of the bifurcation sequence along the slice $\alpha=0.5$ from \cref{tab:Sequence}. Here, we illustrate how the infinitely many homoclinic and heteroclinic bifurcations organize themselves according to our discussion in the previous paragraph. The intersection sets presented in \cref{fig:StereoEtoP} give us a good explanation for the gap between $\myHom{o}{\Gamma_o,\Gamma_t}$ and $\myHom{t}{2\Gamma_o}$ in \fref{fig:SequenceSlice} and the corresponding gap between the $\mu$-values in \cref{tab:Sequence}: the curves in $\widehat{W}^{s}(\mathbf{0})$ involved in these two bifurcations are the boundaries of their respective clusters of accumulation, which are locally isolated. Notice that this is the case for all pairs of homoclinic bifurcations $\myHom{o}{m\Gamma_o,\Gamma_t}$ and $\myHom{t}{(m+1)\Gamma_o}$, as found for $m=1,2,3,4,5,6,7,8,9$ in \cref{tab:Sequence}. Furthermore, each gap provides a parameter window during which the extra loop of $W^u(\mathbf{0})$, created in the respective codimension-one heteroclinic bifurcation $\mathbf{Q}_{\mathbf{0}}^{\Gamma_t}$, moves closer to $\Gamma_o$; recall that the next homoclinic bifurcation of $\Gamma_o$ will be nonorientable. During these gaps, the value of the winding number $\zeta$ remains constant and is given by the number of loops of the homoclinic orbits that bound the region in $(\alpha,\mu)$-plane. Note that after every cluster of accumulation, the unstable manifold $W^u(\mathbf{0})$ gains a rotation that lies closer to $\Gamma_o$. This means that, as $\mu$ decreases, the $\zeta$ value in each gap increases until it reaches infinity at $\mathbf{Q}_{\mathbf{0}}^{\Gamma_o}$.  It was proven in \cite[Theorems 2 and 5]{Hom1} that these homoclinic bifurcations follow a specific symbolic ordering that is related to their orientation.  In particular, the homoclinic bifurcations presented in \cref{tab:Sequence} follow this ordering, and the gaps in \fref{fig:SequenceSlice} correspond to the isolation regions described in these theorems. We remark that this is the first time that this ordering of homoclinic bifurcation cascades has been computed in a specific vector field. Going beyond known results, we are also able to identify the codimension-one heteroclinic bifurcations $\mathbf{Q}_{\mathbf{0}}^{\Gamma_t}$ as the limiting curves of these cascades. Moreover, we  explain how $\mathbf{Q}_{\mathbf{0}}^{\Gamma_o}$ acts as boundary curve of the region in the parameter plan where these homoclinic cascades can exist.

\Cref{fig:SequenceSlice} also illustrates the order of further bifurcations past $\mathbf{Q}_{\mathbf{0}}^{\Gamma_o}$ that are computed and listed in \cref{tab:Sequence}. There is a Smale--horseshoe region $\mathbf{SH}$ after transitioning through $\mathbf{Q}_{\mathbf{0}}^{\Gamma_o}$; this region is discussed in more detail in \cref{sec:HorseShoe}. Inside $\mathbf{SH}$, the stable manifold $W^s(\mathbf{0})$ becomes tangent to the unstable manifold $W^u(\Gamma_o)$ of $\Gamma_o$ at the bifurcation $\mathbf{F}$, which plays the same role as the fold bifurcation curve of the heteroclinic orbit from $\mathbf{q}$ to $\mathbf{0}$ for cases \textbf{A} and \textbf{B} in \cite{Agu1,And1}. The other boundary of the Smale--horseshoe region $\mathbf{SH}$ is the codimension-one homoclinic bifurcation $\mathbf{Tan}_{\Gamma_o}$ of $\Gamma_o$, where $W^u(\Gamma_o)$ and $W^s(\Gamma_o)$ have a quadratic tangency. Past this bifurcation, there are no more homoclinic orbits of $\Gamma_o$ and one finds small regions of existence of chaotic attractors \cite{Nau1}, which are then destroyed via (reversed) period-doubling and saddle-node cascades of periodic orbits. This is discussed in more detail in \cref{sec:strAttrPD}.  The right-most period-doubling cascade in \fref{fig:SequenceSlice}, the one ending with $\mathbf{PD}_{\Gamma_t}$, changes $\Gamma_t$ into an attracting periodic orbit $\Gamma^a_t$ when $^2\Gamma^a_t$ disappears.  The orbit $\Gamma^a_t$ then disappears in the saddle-node bifurcation $\mathbf{SNP}_{\Gamma_o}$, where it merges with the only other remaining periodic orbit $\Gamma_o$; this marks the entry into region~\mBlue{2}.

\section{Smale--horseshoe region}
\label{sec:HorseShoe} 

\begin{figure}
\centering
\includegraphics{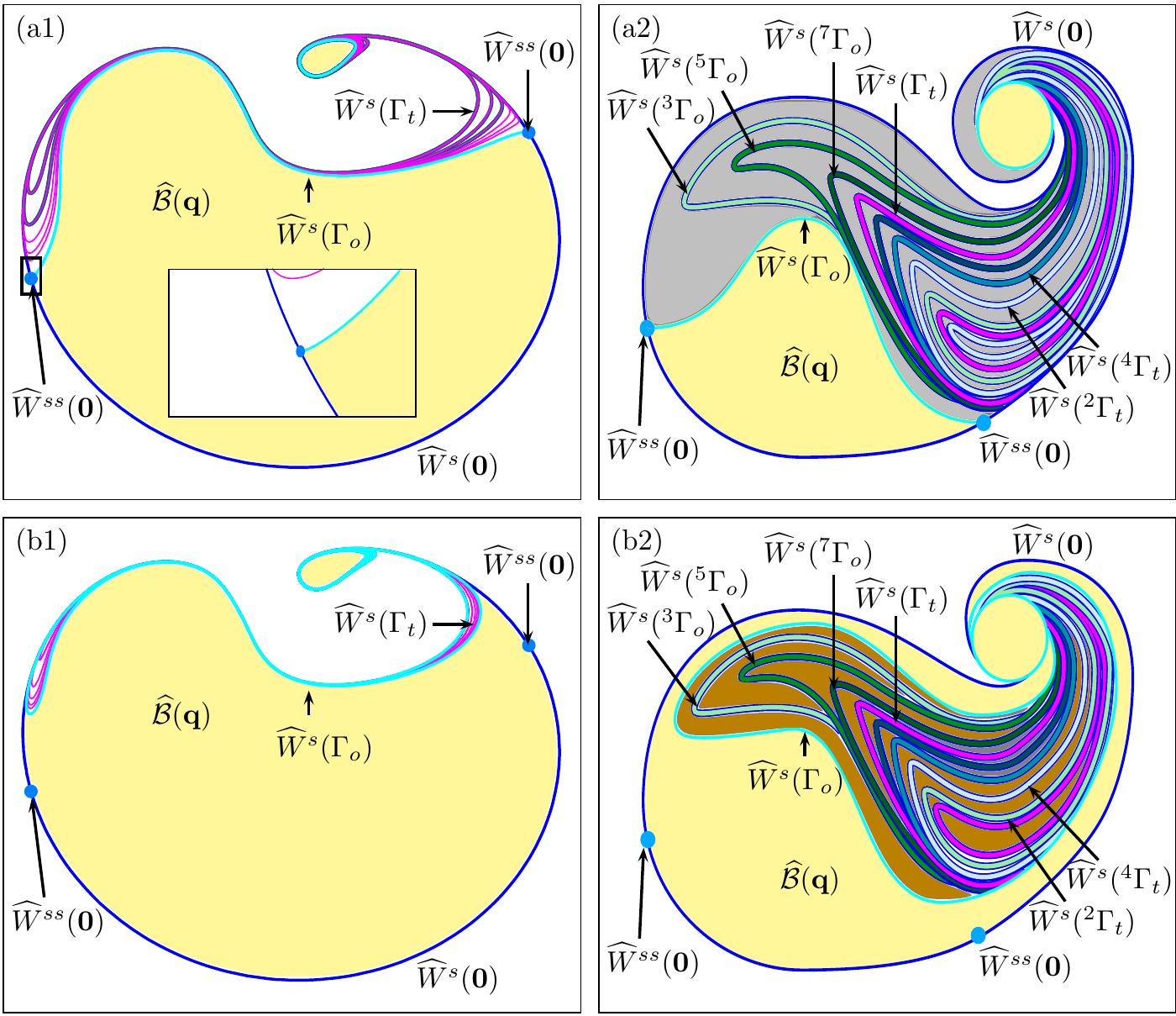}
\caption{Stereographic projection of the intersection sets of the invariant manifolds with $\mS^*$ at, and after the codimension-one heteroclinic bifurcation $\mathbf{Q}^{\Gamma_o}_{\mathbf{0}}$ in rows~(a) and (b), respectively.  Column~(1) shows the computed projections and column~(2) shows corresponding sketches.  The color code and labeling of the regions is the same as in \cref{fig:FirstStereo}.  Panels~(a1) and (b1) are for $(\alpha, \mu)=(0.5, -0.004861805)$ and $(\alpha, \mu)=(0.5, -0.0065)$, respectively.}
\label{fig:stereoEtoPOri}
\end{figure} 

\Fref{fig:stereoEtoPOri} shows the transition from the bifurcation $\mathbf{Q}_{\mathbf{0}}^{\Gamma_o}$ to the Smale--horseshoe region on the level of the intersection sets with the sphere $\mS^*$. As before, the left column shows the numerically computed curves and the right column shows sketches where the gray and brown regions highlight the relative positions of the accumulation regions; each accumulation region is represented by the main intersection curve of the stable manifold of the saddle periodic orbit that exists on it. We present the stereographic projection at the codimension-one bifurcation $\mathbf{Q}_{\mathbf{0}}^{\Gamma_o}$ in panels~(a1) and (a2). Notice in panel~(a1) how one of the topological circles in $\widehat{W}^s(\Gamma_o)$ forms a connecting curve between the two intersection points in $\widehat{W}^{ss}(\mathbf{0})$, as a consequence of $W^s(\Gamma_o)$ accumulating on $W^{ss}(\mathbf{0})$ backward in time; see \cref{fig:tranEtoPOri}. In this limiting case, all the accumulation regions accumulate onto this single curve $\widehat{W}^s(\Gamma_o)$, while spiraling around the other topological circle in $\widehat{W}^s(\Gamma_o)$. The boundary of $\widehat{\mathcal{B}}(\mathbf{q})$ is formed by $\widehat{W}^s(\Gamma_o)$ and the segment of $\widehat{W}^{s}(\mathbf{0})$ in beween the two points $\widehat{W}^{ss}(\mathbf{0})$. Before $\widehat{W}^s(\Gamma_o)$ connects the two points in $\widehat{W}^{ss}(\mathbf{0})$ at $\mathbf{Q}_{\mathbf{0}}^{\Gamma_o}$, infinitely many heteroclinic and homoclinic bifurcations must occur as the corresponding intersection sets accumulate in a spiraling manner onto $\widehat{W}^s(\Gamma_o)$; see panels~(b1) and (b2) of \cref{fig:StereoEtoP}. This explain why the cascades of heteroclinic bifurcations $\mathbf{Q}_{\mathbf{0}}^{\Gamma_t}$ and homoclinic bifurcation accumulate onto the bifurcation curve $\mathbf{Q}_{\mathbf{0}}^{\Gamma_o}$ in \cref{fig:Second}.  

\begin{figure}
\centering
\includegraphics{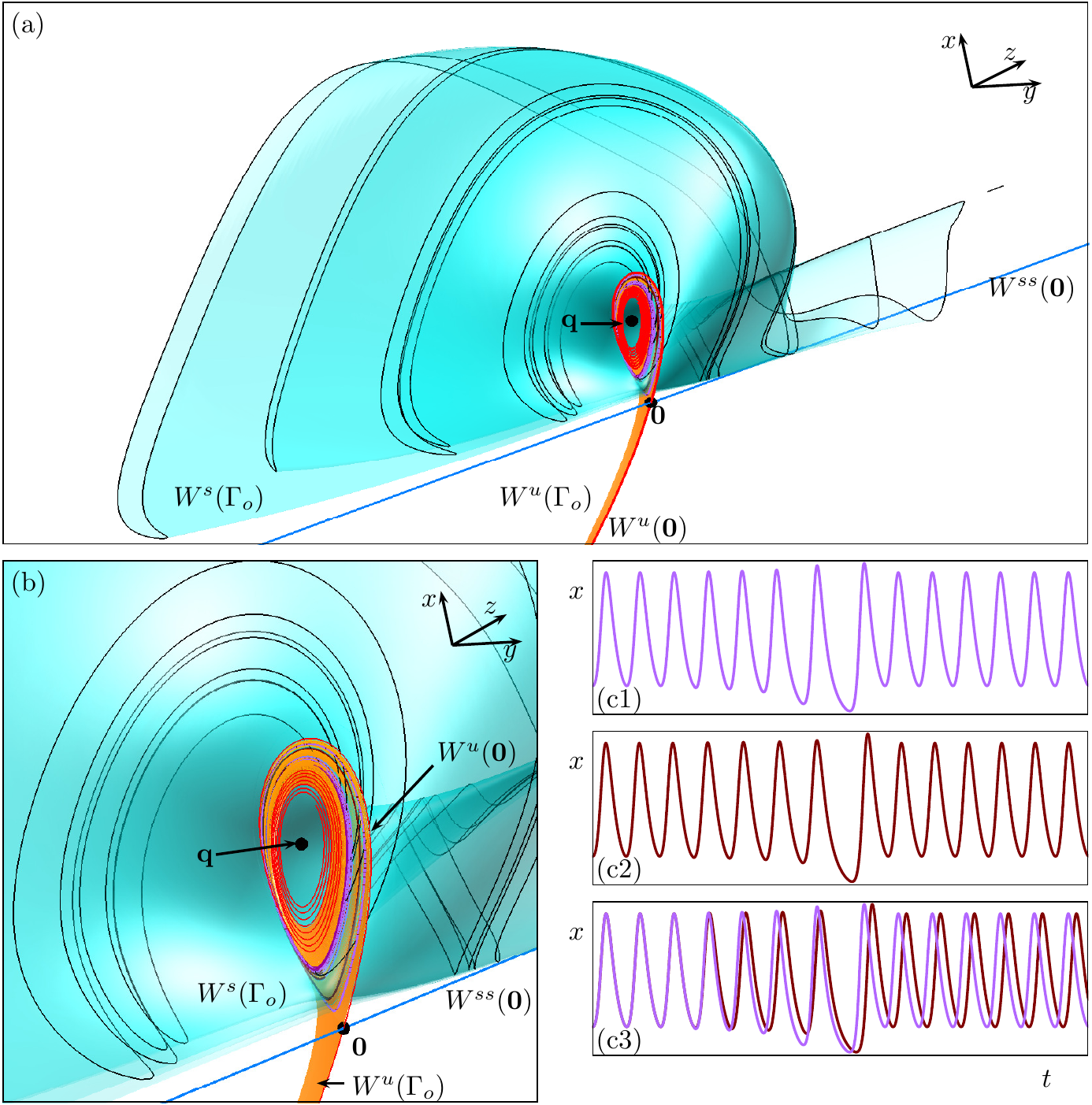}
\caption{Configuration in phase space of $W^s(\Gamma_o)$ (cyan surface) and $W^u(\Gamma_o)$ (orange surface) of system~\cref{eq:san} in the Smale--horseshoe region. Panel (a) shows a big part of $W^s(\Gamma_o)$ and $W^u(\Gamma_o)$ in phase space; while panel (b) shows an enlargement of two transversal homoclinic orbits (maroon and lilac curves) of $\Gamma_o$. Panels (c1) and (c2) show the $x$-component of the transversal homoclinic orbits of $\Gamma_o$ separately; while panel (c3) shows them together.  Also shown are $W^u(\mathbf{0})$ (red curve) and $W^{ss}(\mathbf{0})$ (blue curve).  The panels are for $(\alpha, \mu)=(0.5, -0.0065)$. See also the accompanying animation ({\color{red} GKO\_Cflip\_animatedFig16.gif}).}
\label{fig:HorseShoeDyn} 
\end{figure} 

In panels~(b1) and (b2) of \cref{fig:stereoEtoPOri}, after the heteroclinic bifurcation $\mathbf{Q}_{\mathbf{0}}^{\Gamma_o}$, the curve $\widehat{W}^s(\Gamma_o)$ that connected the two points $\widehat{W}^{ss}(\mathbf{0})$ now accumulates onto the topological circle in $\widehat{W}^s(\Gamma_o)$ that persists through this bifurcation. This accumulation is a consequence of the existence of the homoclinic orbits of $\Gamma_o$ and the $\lambda$-lemma. Since there are infinitely many homoclinic orbits and heteroclinic cycles, there are infinitely many intersection curves in $\widehat{W}^s(\Gamma_o)$, $\widehat{W}^s(\Gamma_t)$, $\widehat{W}^s(^2\Gamma_t)$, etc., that accumulate onto the topological circle in $\widehat{W}^s(\Gamma_o)$. Hence, the accumulation region in the intersection sets is more complicated; we illustrate this change by the brown shading in panel~(b2).  Note also that the boundary of the basin of attraction $\partial \widehat{\mathcal{B}}(\mathbf{q}) $ of $\mathbf{q}$ after the bifurcation $\mathbf{Q}_{\mathbf{0}}^{\Gamma_o}$ in panels~(b1) and (b2) has changed from two topological circles to a topological circle with two handles, that is, $\partial \widehat{\mathcal{B}}(\mathbf{q}) \subset \widehat{W}^{s}(\mathbf{0}) \cup \widehat{W}^s(\Gamma_o)$. 

\subsection{Dynamics inside the region $\mathbf{SH}$}

Past the codimension-one heteroclinic bifurcation $\mathbf{Q}_{\mathbf{0}}^{\Gamma_o} $ of system~\cref{eq:san}, we find structurally stable homoclinic orbits of $\Gamma_o$; hence, there exist Smale--horseshoe dynamics in phase space \cite{Taken1}. \Fref{fig:HorseShoeDyn}~(a) shows such homoclinic orbits (maroon and lilac curves) in more detail. The manifolds $W^s(\Gamma_o)$ (cyan surface) and $W^u(\Gamma_o)$ (orange surface) intersect along the homoclinic orbits $\Gamma_o$.  To illustrate further how these manifolds interact, the accompanying animation shows the manifold $W^s(\Gamma_o)$ from \fref{fig:HorseShoeDyn}~(a) as it is grown in size. Note how the outer boundary of the lower half of the topological cylinder $W^s(\Gamma_o)$ (black curve) accumulates in a spiraling fashion onto the top half of $W^s(\Gamma_o)$. Panel~(b) shows an enlargement of panel~(a) around the homoclinic orbits, where this accumulation is illustrated more clearly. The time series of the $x$-component of the two homoclinic orbits are shown in panels~(c1) and (c2), and together in panel~(c3). The lilac orbit moves faster away from $\Gamma_o$ than the maroon orbit.  The existence of these homoclinic orbits implies infinitely many secondary homoclinic orbits of $\Gamma_o$ \cite{Taken1}. In addition, their existence also implies the appearance of heteroclinic cycles between different saddle periodic orbits, such as $\Gamma_o$, $\Gamma_t$, $^2\Gamma_t$, etc., due to the accumulation of their stable manifolds backward in time onto $W^s(\Gamma_o)$ and the accumulation of $W^u(\Gamma_o)$ onto their respective unstable manifolds.

\subsection{Transition through $\mathbf{F}$}

\begin{figure}
\centering
\includegraphics[scale=0.95]{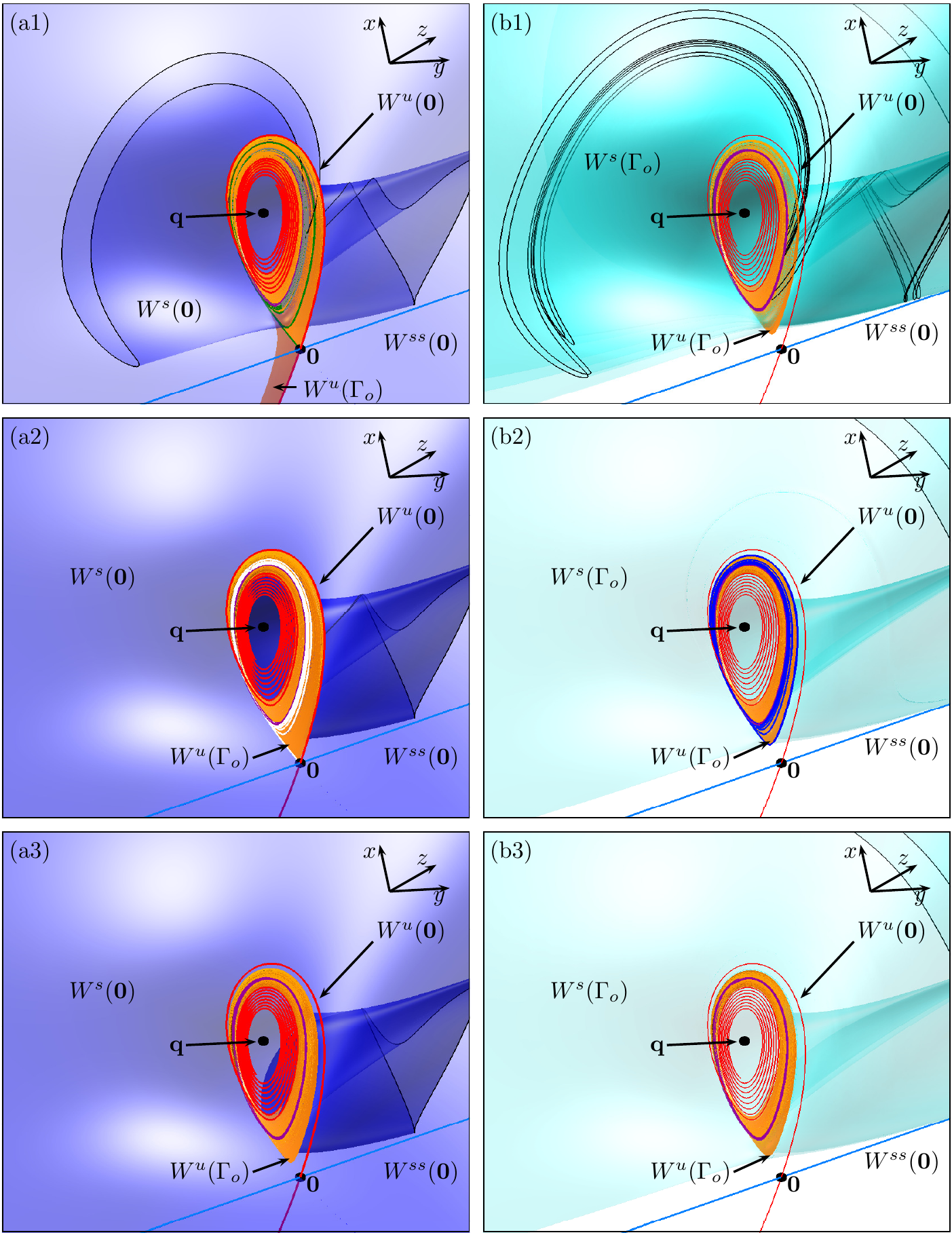}
\caption{Phase portraits of the transition through the bifurcation curves $\mathbf{F}$ (left column) and $\mathbf{Tan}_{\Gamma_o}$ (right column) in system~\cref{eq:san}. Shown are the effect of the transitions with respect to $W^s(\mathbf{0})$ and $W^u(\Gamma_o)$ (left column), and with respect to $W^s(\Gamma_o)$ and $W^u(\Gamma_o)$ (right column). Shown are $W^s(\mathbf{0})$ (blue surface), $W^s(\Gamma_o)$ (cyan surface) and $W^u(\Gamma_o)$ (orange surface); also shown are the corresponding tangency orbits (white and blue curves), $W^u(\mathbf{0})$ (red curve), and $W^{ss}(\mathbf{0})$ (blue curve).  Panels~(a1), (a2) and (a3) are for $(\alpha, \mu)=(0.5, -0.0065)$, $(\alpha, \mu)=(0.5, -0.007054355)$ and $(\alpha, \mu)=(0.5, -0.00706)$, respectively. Panels~(b1), (b2) and (b3) are for $(\alpha, \mu)=(0.5, -0.00706)$, $(\alpha, \mu)=(0.5, -0.007076705)$ and $(\alpha, \mu)=(0.5, -0.0071)$, respectively. See also the accompanying animation ({\color{red} GKO\_Cflip\_animatedFig17.gif}).}
\label{fig:CFHT}
\end{figure} 

The numerical results shown in \cref{tab:Sequence} and \fref{fig:Second} indicate that the fold bifurcation $\mathbf{F}$, in the vicinity of $\mathbf{C_I}$, occurs before the intersection between $W^u(\Gamma_o)$ and $W^s(\Gamma_o)$ becomes tangential, that is, before the bifurcation $\mathbf{Tan}_{\Gamma_o}$. The left column of \fref{fig:CFHT} illustrates the crossing of $\mathbf{F}$, that is, the moment $W^u(\Gamma_o)$ (orange surface) is tangent to $W^s(\mathbf{0})$ (blue surface).  Note the transversal intersection between $W^s(\mathbf{0})$ and $W^u(\Gamma_o)$ in panel~(a1) that persists since its creation in region~\mBlue{1}. Since $\Gamma_o \rightarrow \Gamma_t \rightarrow \mathbf{0}$, there exist infinitely many heteroclinic orbits from $\Gamma_o$ to $\mathbf{0}$, of which we only show two in panel~(a1) (green curves). As $\mathbf{F}$ occurs, which is illustrated in panel~(a2), the stable manifold $W^s(\mathbf{0})$ goes through a quadratic tangency with $W^u(\Gamma_o)$; this is similar to cases \textbf{A} and \textbf{B} \cite{Agu1,And1}, where $W^s(\mathbf{0})$ exhibited a quadratic tangency with $W^u(\mathbf{q})$. In particular for case~\textbf{C}, the infinitely many heteroclinic orbits merge in a tangent heteroclinic orbit, shown as the white curve in panel~(a2).  After the transition through $\mathbf{F}$ shown in panel~(a3), $W^s(\mathbf{0})$ no longer intersects $W^u(\Gamma_o)$; hence, $W^s(\mathbf{0})$ no longer accumulates onto $W^s(\Gamma_o)$ and the unstable manifold $W^u(\mathbf{0})$ no longer bounds $W^u(\Gamma_o)$. In fact, we find that $W^s(\mathbf{0})$ no longer intersects any unstable manifold of any other saddle periodic orbit.

\section{Transition to a strange attractor and period-doubling cascade}
\label{sec:strAttrPD} 

Past $\mathbf{Tan}_{\Gamma_o}$ one finds chaotic attractors, and cascades of period-doubling and saddle-node bifurcations. This allows system~\cref{eq:san} to transition back to region~\mBlue{2}; see \cref{sec:EasyRegions}. 

\subsection{Transition through  $\mathbf{Tan}_{\Gamma_o}$}

\begin{figure}
\centering
\includegraphics{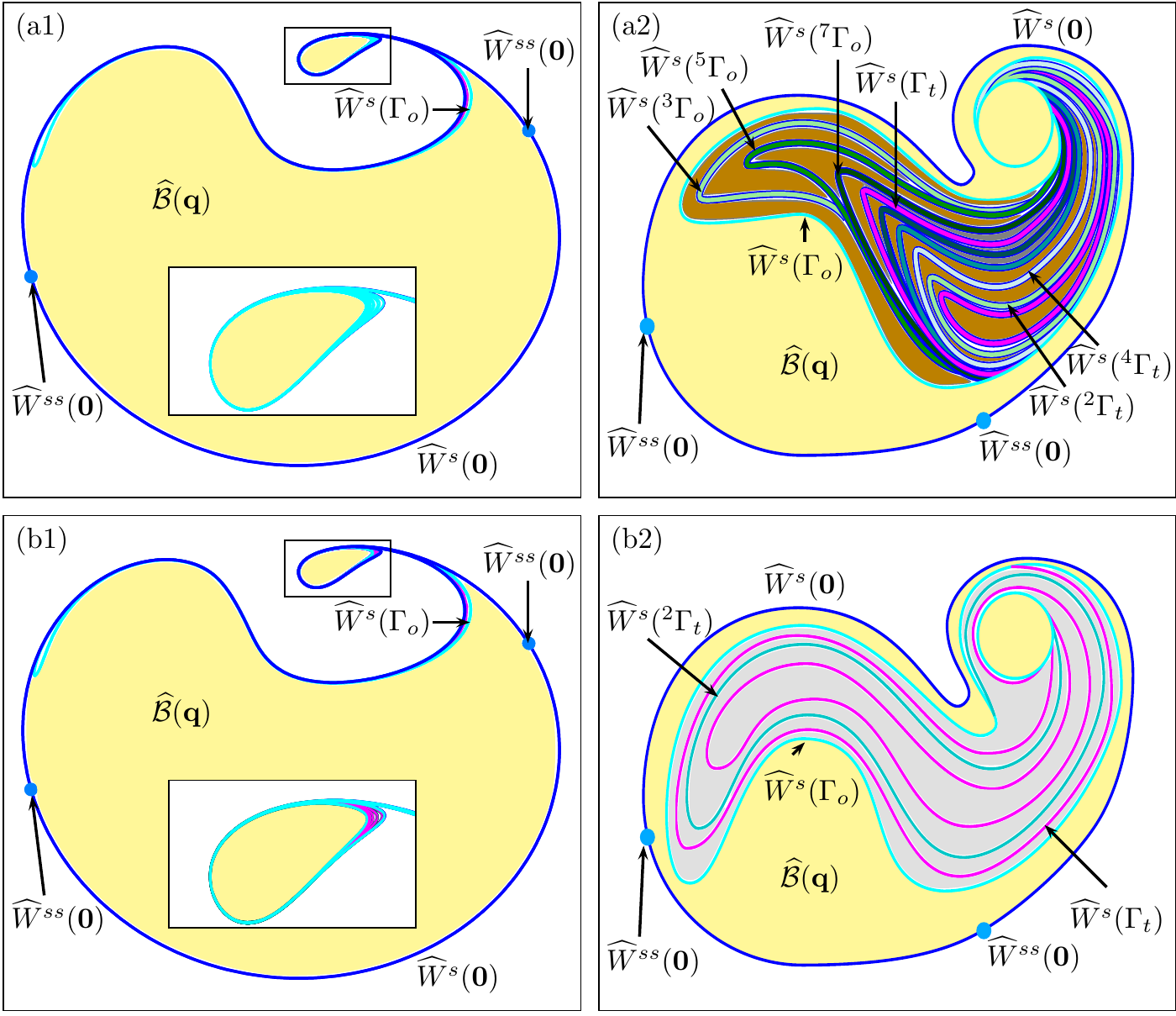}
\caption{Stereographic projection of the intersection sets of the invariant manifolds with $\mS^*$ after crossing the bifurcation curves $\mathbf{F}$ (top row) and $\mathbf{Tan}_{\Gamma_o}$ (bottom row).  The left column shows the computed projections and the right column shows the corresponding sketches.  The color code and labeling is the same as in \cref{fig:FirstStereo}.  Panels~(a1) and (b1) are for $(\alpha, \mu)=(0.5, -0.00706)$ and $(\alpha, \mu)=(0.5, -0.0071)$, respectively.}
\label{fig:StereoFT}
\end{figure} 

The right column of \cref{fig:CFHT} illustrates the crossing of $\mathbf{Tan}_{\Gamma_o}$, that is, the last tangency between $W^u(\Gamma_o)$ (orange surface) and $W^s(\Gamma_o)$ (cyan surface). After the transition through $\mathbf{F}$ and before $\mathbf{Tan}_{\Gamma_o}$ occurs, the intersection between $W^s(\Gamma_o)$ and $W^u(\Gamma_o)$ still persists, as shown in panel~(b1).  As $\mu$ is decreased, system~\cref{eq:san} must exhibit a sequence of codimension-one homoclinic bifurcations of secondary tangencies between $W^s(\Gamma_o)$ and $W^u(\Gamma_o)$ \cite{Taken1}, until it reaches the last homoclinic tangency at $\mathbf{Tan}_{\Gamma_o}$; shown in panel~(b2). The bifurcation $\mathbf{Tan}_{\Gamma_o}$ represents the last intersection (blue curve) between $W^u(\Gamma_o)$ and $W^s(\Gamma_o)$. After the transition through $\mathbf{Tan}_{\Gamma_o}$, shown in panel~(b3), the manifold $W^s(\Gamma_o)$ stops accumulating on itself, because it no longer intersects $W^u(\Gamma_o)$.

\Cref{fig:StereoFT} shows additional details of the transition through $\mathbf{Tan}_{\Gamma_o}$ by considering corresponding intersection sets with the sphere $\mS^*$. Row~(a) shows the situation in between the bifurcations $\mathbf{F}$ and $\mathbf{Tan}_{\Gamma_o}$ were the intersection set $\widehat{W}^s(\mathbf{0})$ is formed by a single closed curve; the left column shows the computed stereographic projections and the right column the corresponding sketches. Hence, the accumulation region (brown region) that is sketched in panel~(a2) does not contain any intersection curves of $\widehat{W}^s(\mathbf{0})$. We find that the other intersection curves do not go through any significant changes for parameter values in between the bifurcations $\mathbf{F}$ and $\mathbf{Tan}_{\Gamma_o}$. Panels~(b1) and (b2) illustrate the consequence of crossing $\mathbf{Tan}_{\Gamma_o}$. In particular, the intersection set $\widehat{W}^s(\Gamma_o)$ now consists of two single closed curves that do not accumulate onto any other intersection curve. Nevertheless, the existence of nontransversal intersections of the unstable manifolds of other saddle periodic orbits with $W^s(\Gamma_o)$ forces the other intersection curves to accumulate onto $\widehat{W}^s(\Gamma_o)$. Since there are no homoclinic orbits of $\Gamma_o$ anymore, we color the accumulation region light gray in panel~(b2). As we will see in the next section, this region may correspond to the basin of attraction of a strange attractor or of an attracting periodic orbit. Note also that $\widehat{W}^s(\Gamma_o)$ is a topological annulus. The insets in panels~(a1) and (b1) show enlargements around one of the topological circles of $W^s(\Gamma_o)$, illustrating the accumulation of $W^s(\Gamma_o)$ on itself and the lack of it before and after the transition through $\mathbf{Tan}_{\Gamma_o}$, respectively.

\subsection{Evidence of the chaotic attractor}

\begin{figure}
\centering
\includegraphics{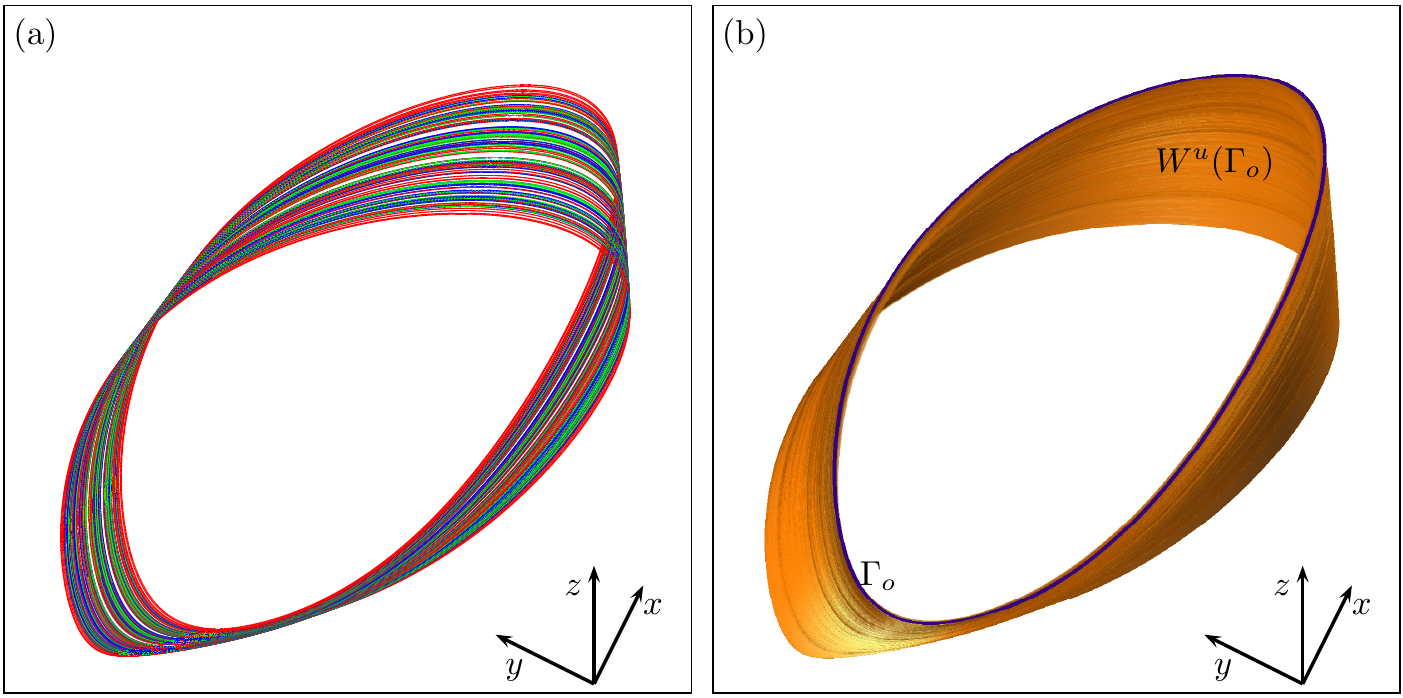}
\caption{Approximation of the chaotic attractor of system~\cref{eq:san} close to the codimension-one homoclinic bifurcation $\mathbf{Tan}_{\Gamma_o}$ of $\Gamma_o$. Panel (a) shows three trajectories (red, blue and green), with initial conditions that differ in the sixth decimal place. Panel (b) shows a portion of $W^u(\Gamma_o)$.  Here $(\alpha,\mu)\approx(0.5,-0.007076768)$.} 
\label{fig:StranAttractor} 
\end{figure}

\begin{figure}
\centering
\includegraphics{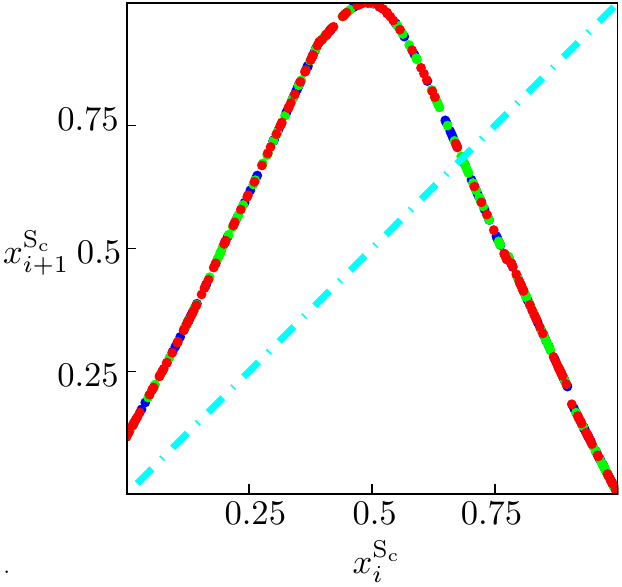} 
\caption{First return map of the scaled $x$-coordinates of the three trajectories shown in \fref{fig:StranAttractor} at the Poincar\'e section $y=0$.  The line $x_{i+1}^{\rm S_c}=x_i^{\rm S_c}$ is indicated (cyan dashed line) and the points are colored according to their corresponding trajectories in \fref{fig:StranAttractor}.} 
\label{fig:ReturnMap}   
\end{figure} 

\Cref{fig:StranAttractor} shows a chaotic attractor in the phase space of system~\cref{eq:san} for $\mu=-0.007076768$. Panel~(a) shows trajectories of three different initial conditions that agree up to five decimal places. The trajectories are shown after their transients have been discarded, to illustrate that they are attracted to a lower-dimensional object in phase space. In particular, the trajectories do not remain arbitrarily close to each other as time progresses, but they do remain arbitrarily close to $W^u(\Gamma_t)$, which is shown in panel~(b). This evidence suggests the existence of a chaotic attractor, which would be contained in the closure of $W^u(\Gamma_o)$.

As further evidence, \fref{fig:ReturnMap} shows the successive returns $x_i^{\rm S_c}$ to the plane $y=0$ of the $x$-coordinates of each of these three trajectories, plotted as $x_{i+1}^{\rm S_c}$ versus $x_i^{\rm S_c}$, where $x_i^{\rm S_c}$ is scaled to the interval $[0,1]$. Clearly, the return map in \cref{fig:ReturnMap} is a unimodal map, which is effectively another clear indication that system~\cref{eq:san} has a chaotic attractor. In fact, it was proven that the Poincar\'e return map close to an inclination flip bifurcation of case~\textbf{C}, up to some $C^r$-rescaling, is $C^r$-close to the family of unimodal maps $\psi_{\overline{a}}(u,v)=(0,1-\overline{a}v^2)$ \cite{Nau1}. We infer from \cref{fig:StranAttractor,,fig:ReturnMap} that the strange attractor contained in the closure of $W^u(\Gamma_o)$ can, therefore, be identified with a template that appears to be equivalent to the template of the R\"{o}ssler attractor \cite{Ali1}. We conjecture that this strange attractor is destroyed in the tangency $\mathbf{Tan}_{\Gamma_o}$ between $W^u(\Gamma_o)$ and $W^s(\Gamma_o)$; however, the situation might be more complicated and involve a loss of hyperbolicity at a first period-doubling bifurcation, as considered in \cite{asaoka1997}.

\subsection{Transition through $\mathbf{PD}$ back to region~\mBlue{2}}

\begin{figure}
\centering
\includegraphics{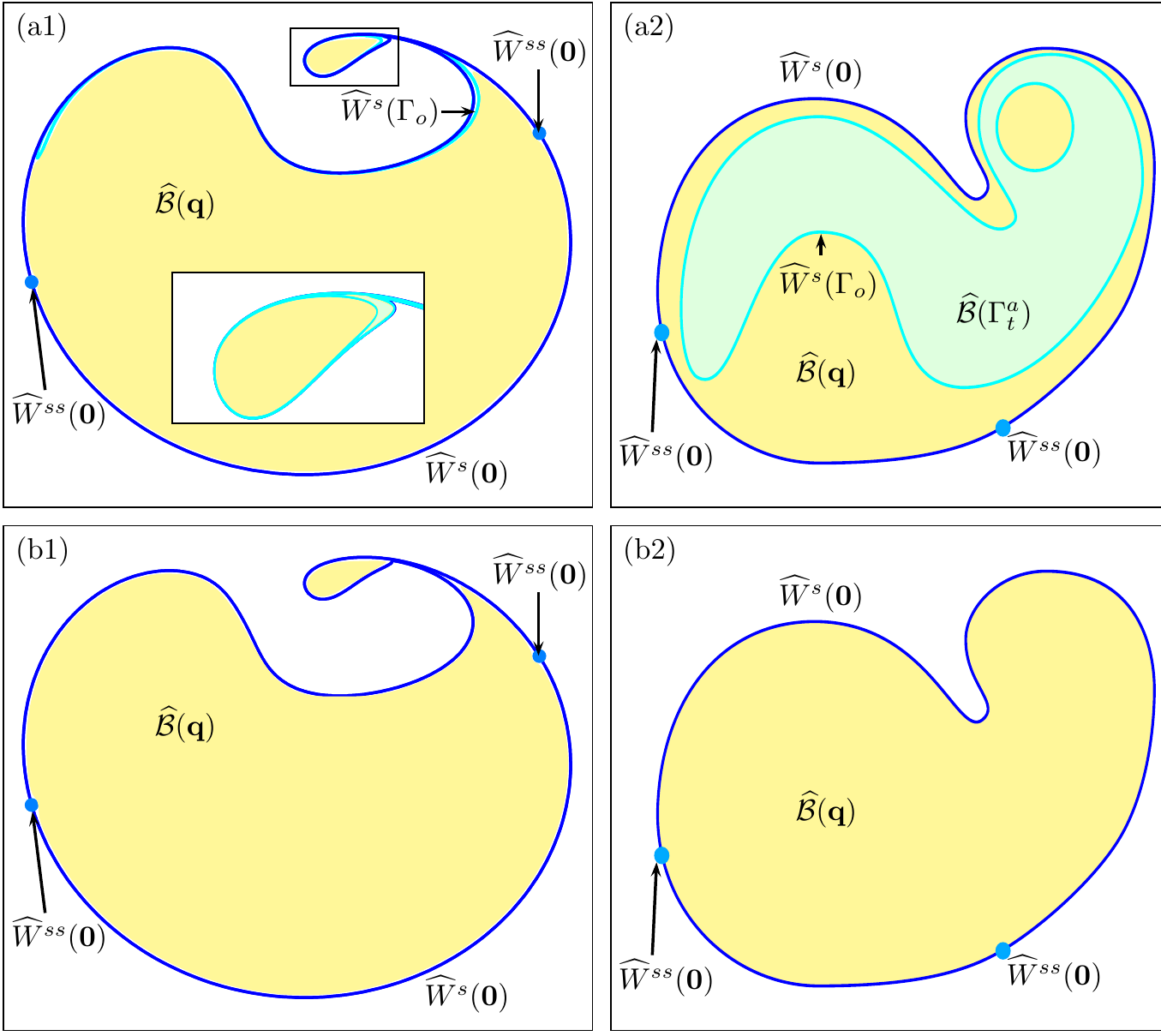}
\caption{Stereographic projection of the intersection sets of the invariant manifolds with $\mS^*$ before (top row) and after (bottom row) the last saddle-node bifurcation $\mathbf{SNP}_{\Gamma_o}$ of periodic orbits.  The left column shows the computed projections and the right column shows the corresponding sketches.  The color code and nomenclature of the regions is the same as given in \cref{fig:FirstStereo}.  Panels~(a1) and (b1) are for $(\alpha, \mu)=(0.5, -0.0073)$ and $(\alpha, \mu)=(0.5, -0.0074)$, respectively.} 
\label{fig:FinalStereo}
\end{figure} 

The saddle periodic orbits that are created during the homoclinic and heteroclinic cascades disappear in cascades of period-doubling and saddle-node bifurcations as $\mu$ is decreased. During this process the intersection sets of different stable manifolds disappear with the corresponding periodic orbits. In fact, the accumulation region alternates between being a basin of attraction of an attracting periodic orbit and that of a strange attractor.  This is caused by the saddle-node bifurcations that create periodic windows and start different period-doubling cascades \cite{Ali1}. These period-doubling and saddle-node cascades terminate at the final saddle-node bifurcation $\mathbf{SNP}_{\Gamma_o}$, which marks the transition back to region~\mBlue{2}.

\Fref{fig:FinalStereo} illustrates the moment before and after the bifurcation $\mathbf{SNP}_{\Gamma_o}$ on the level of the intersection sets of invariant manifolds with $\mS^*$; as before, the left column shows the computed stereographic projections and the right column the corresponding sketches. In panels~(a1) and (a2), after the period-doubling bifurcation $\mathbf{PD}_{\Gamma_t}$, all periodic orbits have disappeared with the exception of $\Gamma_o$ and $\Gamma_t$, which is now attracting periodic orbit $\Gamma^a_t$. Immediately after the period-doubling bifurcation, $\Gamma^a_t$ has a nonorientable strong stable manifold $W^{ss}(\Gamma^a_t)$; however, $\Gamma^a_t$ becomes an attracting periodic orbit $\Gamma^a_o$ with an orientable strong stable manifold before reaching $\mathbf{SNP}_{\Gamma_o}$. This transition from having a nonorientable and an orientable strong stable manifold occurs via a crossing of the curves $\mathbf{CC^-_{\Gamma^a}}$ and $\mathbf{CC^+_{\Gamma^a}}$ where the Floquet multipliers of $\Gamma^a_t$ ($\Gamma^a_o$) change from being real positive (negative) to complex conjugate. In between these curves, the attracting periodic orbit $\Gamma^a$ does not have a well-defined strong stable manifold \cite{And1}. We discuss the existence of the curves $\mathbf{CC}^{\pm}_{\Gamma^a}$ in more detail in \cref{sec:Bubbles}, where we study the bifurcation diagram of system~\cref{eq:san} over an even larger parameter range.  On the level of the intersection sets shown in \cref{fig:FinalStereo}(a), the former accumulation region has now become the basin of attraction $\widehat{\mathcal{B}}(\Gamma^a_t)$ of $\Gamma^a_t$ (green). Furthermore, the basin $\widehat{\mathcal{B}}(\mathbf{q})$ is a disconnected region with boundaries $\widehat{W}^s(\Gamma_o)$ and $\widehat{W}^s(\mathbf{0})$. After, the saddle-node bifurcation $\mathbf{SNP}_{\Gamma_o}$, which marks the disappearance of $\Gamma^a$ and $\Gamma_o$, we find ourselves again in region~\mBlue{2}, as is illustrated in panels~(b1) and (b2). Hence, these panels are topological equivalent to panel~\mBlue{2} of \fref{fig:FirstStereo}.

\section{Global picture in the $(\alpha,\mu)$-plane} 
\label{sec:Bubbles}

\begin{figure}
\centering
\includegraphics{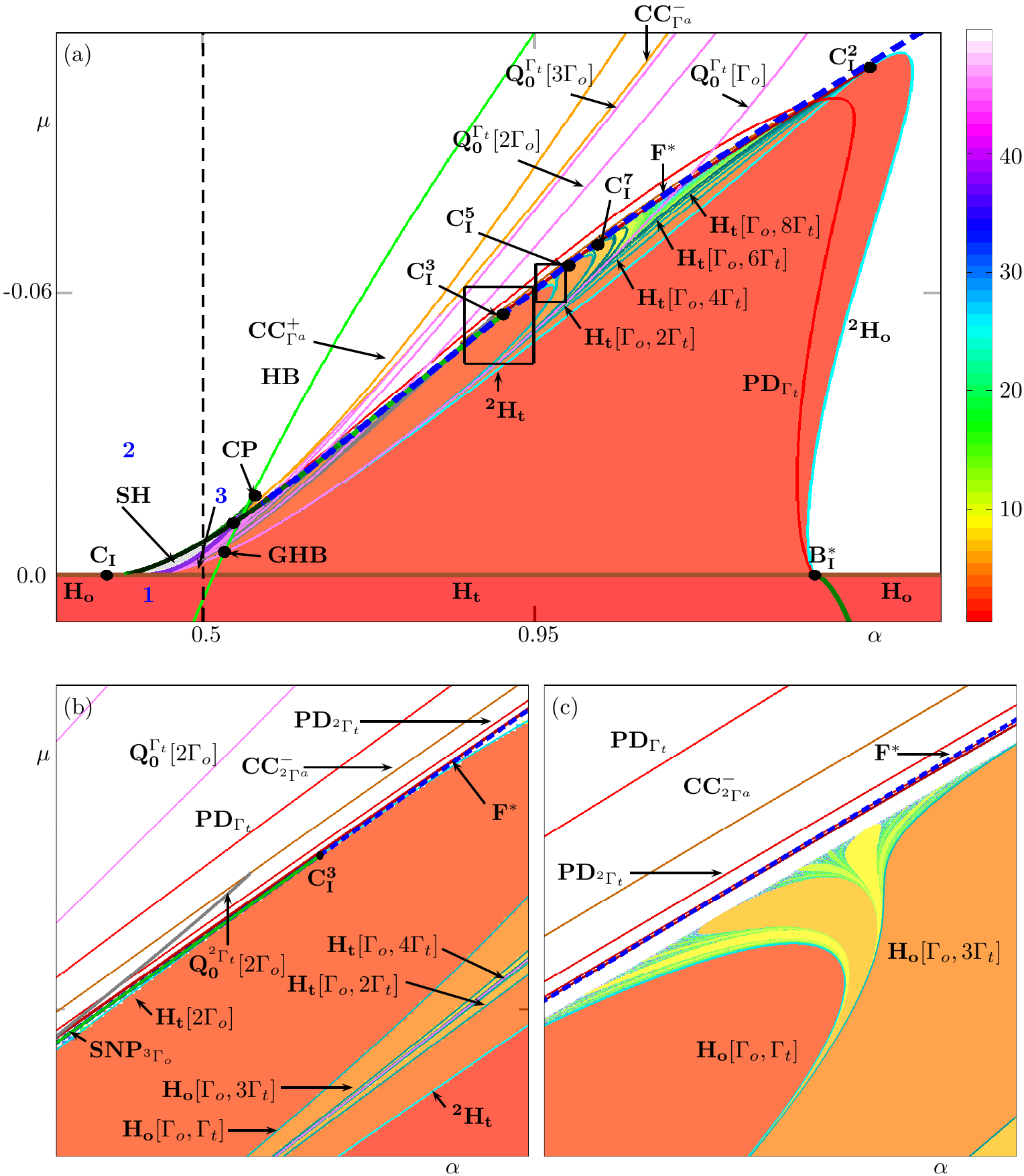}
\caption{Bifurcation diagram of system~\cref{eq:san} for a larger region in the $(\alpha,\mu)$-plane than shown in \cref{fig:Second}.  Panels~(b) and (c) are enlargements as indicated in panel (a). The color code of the regions and labeling of the curves are the same as in \cref{fig:Second}. We also show the curve $\mathbf{Q}_{\mathbf{0}}^{^2\Gamma_t}$ (gray), and the curves $\mathbf{CC}^{\pm}_{\Gamma^a}$ and $\mathbf{CC}^{-}_{^2\Gamma^a}$ (orange) of $\Gamma^a$ and $^2\Gamma^a$, respectively.}
\label{fig:Third}
\end{figure}   

Homoclinic flip bifurcations arise in different mathematical models that describe physical phenomena \cite{Lina1,Yan1}. Certainly from the applications point of view, it is of interest to study how the bifurcation curves that emanate from the homoclinic flip bifurcation point of case~\textbf{C} are organized beyond a small neighborhood of this central codimension-two point. Hence, we now focus on the role of the homoclinic flip bifurcation point $\mathbf{C_I}$ as an organizing center for dynamics and bifurcations in a larger region of the $(\alpha, \mu)$-plane. The overall structure of the bifurcation set in such a larger parameter range features further codimension-two bifurcation phenomena that occur at some distance from $\mathbf{C_I}$. As a  particular type of bifurcation structure we find homoclinic bubbles, which are known to arise as a mechanism for  transitioning between cases~\textbf{B}~and~\textbf{C} in parameter slices near a codimension-three homoclinic flip bifurcation \cite{Hom2}. We show that homoclinic bubbles organize the different bifurcation curves that emerge locally from $\mathbf{C_I}$ more globally in the $(\alpha, \mu)$-plane.

\Cref{fig:Third} shows the numerically computed bifurcation diagram of $\mathbf{C_I}$ for a much larger parameter range in panel~(a), as well as two enlargements in panels~(b) and (c). In particular, the parameter sweep of $\zeta$ now plays an even more prominent role in discerning small and subtle interactions of the different bifurcation curves. In \cref{fig:Third}, we recognize from \cref{fig:Second} and \cref{tab:Sequence} the principal homoclinic branch (brown curve) with $\mathbf{C_I}$, as well as the curves of saddle-node bifurcation $\mathbf{SNP}$ (dark-green), of period-doubling bifurcation $\mathbf{PD}$ (red), of homoclinic bifurcation $\mathbf{Tan}_{\Gamma_o}$ (violet), of heteroclinic bifurcation $\mathbf{Q}_{\mathbf{0}}^{\Gamma_t}$ (magenta), $\mathbf{Q}_{\mathbf{0}}^{\Gamma_o}$ (purple) and $\myEtoP{\mathbf{0}}{^2\Gamma_t}{2\Gamma_o}$ (gray), and of homoclinic bifurcation $\mathbf{H_o}[m\Gamma_o,n\Gamma_t]$ $\mathbf{H_t}[m\Gamma_o,n\Gamma_t]$ (cyan curves).  Note also the region $\mathbf{SH}$ where Smale--horseshoe dynamics occurs, and the curve $\mathbf{HB}$ of Hopf bifurcation (green curve) that transforms $\mathbf{q}$ into an unstable saddle focus.  Panel~(a) shows that the curve $\mathbf{HB}$ does not interact with any other bifurcation curve for $\mu$-values larger than that at the generalized Hopf bifurcation point $\mathbf{GHB}$. We also observe more clearly the curve $\mathbf{F^*}$ (dashed blue); recall that this curve represents the moment when $W^s(\mathbf{0})$ becomes tangent to $W^u(\mathbf{q})$, and it is the extension of the curve $\mathbf{F}$ past its intersection with $\mathbf{SNP^*}$, as shown in \cref{fig:Second}.  At first glance, $\mathbf{F^*}$ seems to bound the region where the homoclinic curves exists, but this is not the case, as we are going to discuss in more detail in the following section. We also show the curves $\mathbf{CC}^{-}_{\Gamma^a}$ and $\mathbf{CC}^{+}_{\Gamma^a}$ (orange) which represent the moment the Floquet multipliers of the attracting periodic orbit $\Gamma^a$ become complex conjugate with negative and positive real parts, respectively.  In addition, we also present the curve $\mathbf{CC}^{-}_{^2\Gamma^a}$ (dark orange) of the attracting periodic orbit $^2\Gamma^a$; these are visible in panels~(b) and (c). 

The curves $\mathbf{Q}_{\mathbf{0}}^{\Gamma_t}$ shown in \cref{fig:Third} extend to the top part of the figure in panel~(a). They represent the extensions of the heteroclinic bifurcation curves in \cref{fig:Second} past the transition of $\Gamma_t$ to the attracting periodic orbit $\Gamma^a_t$, where $\mathbf{Q}_{\mathbf{0}}^{\Gamma_t}$ corresponds to the intersection of the two-dimensional strong stable manifold $W^{ss}(\Gamma^a_t)$ with $W^u(\mathbf{0})$, which is also a codimension-one phenomenon. As such, these curves terminate at the respective curve $\mathbf{CC}^{-}_{\Gamma^a}$, that is, at the moment when $W^{ss}(\Gamma^a_t)$ is no longer well defined.  This also holds for the heteroclinic bifurcation $\myEtoP{\mathbf{0}}{^2\Gamma_t}{2\Gamma_o}$ (gray curve), which can be seen to terminate at $\mathbf{CC}^{-}_{^2\Gamma^a}$ in \cref{fig:Third}(b). 

\subsection{Cascades of inclination flip bifurcations}

The larger parameter range in \cref{fig:Third}(a) reveals how the homoclinic bifurcations, which were hard to discern in \cref{fig:Second}, fan out and are more distinguishable. Furthermore, we clearly see how each homoclinic bifurcation curve encloses a region in the $(\alpha,\mu)$-plane with a constant $\zeta$ value; indeed, the $\zeta$-value in each region increases as a heteroclinic bifurcation $\mathbf{Q}_{\mathbf{0}}^{\Gamma_t}$ or $\mathbf{Q}_{\mathbf{0}}^{^2\Gamma_t}$ is approached; see also \cref{sec:Cascades}. 

It is an important feature of \cref{fig:Third} that many of the homoclinic, saddle-node and period-doubling bifurcation curves emanating from $\mathbf{C_I}$ end at other codimension-two inclination flip bifurcation points. Notably, the principal homoclinic branch goes through a secondary inclination flip bifurcation $\mathbf{B^*_I}$ of case~\textbf{B}, which is the end point of the curves $\mathbf{^2H_o}$ and $\mathbf{PD}_{\Gamma_t}$. Moreover, other homoclinic curves exhibit inclination flip bifurcations $\mathbf{C^n_I}$ of case~\textbf{C}, where $n$ represents the number of loops that the corresponding homoclinic orbit makes in phase space.

The bifurcation diagram near one of these inclination flip points, $\mathbf{C^3_I}$, is shown in \cref{fig:Third}(b).  Notice how $\mathbf{C^3_I}$ is responsible for the transformation of the orientable homoclinic bifurcation $\myHom{o}{\Gamma_o,\Gamma_t}$ to the nonorientable homoclinic bifurcation $\myHom{t}{2\Gamma_o}$, both of which emanate from the initial inclination flip bifurcation $\mathbf{C_I}$.  Recall that these homoclinic bifurcations are related as they represent the last and the first bifurcation of their respective accumulation cluster, as discussed in \cref{sec:CompCascades}. Moreover, they create the saddle periodic orbits $^3\Gamma_o$ and $^3\Gamma_t$, respectively, and they bound a region in the parameter plane where $\zeta=3$; see also \cref{fig:SequenceSlice}. So it does not come as a surprise that $\mathbf{SNP}_{^3\Gamma_o}$ (green curve) ends at the bifurcation point $\mathbf{C^3_I}$, as this bifurcation curve is responsible for the disappearance of $^3\Gamma_o$ and $^3\Gamma_t$; indeed $\mathbf{SNP}_{^3\Gamma_o}$ plays the same role as $\mathbf{SNP}_{\Gamma_o}$ in the unfolding of $\mathbf{C_I}$.  In \Cref{fig:Third}(c) we observe that $\myHom{o}{\Gamma_o,\Gamma_t}$ turns around before reaching $\mathbf{C^3_I}$, and then the curve continues smoothly past $\mathbf{C^3_I}$ as $\myHom{t}{2\Gamma_o}$ before returning to $\mathbf{C_I}$. This same scenario can also be observed for the inclination flip points $\mathbf{C^5_I}$ and $\mathbf{C^7_I}$.  Moreover, \cref{fig:Third}(c) shows that there exist infinitely many regions with constant $\zeta$ in between the two inclination flip points $\mathbf{C^3_I}$ and $\mathbf{C^5_I}$. The boundaries of each of these regions are homoclinic bifurcation curves emanating from $\mathbf{C_I}$, which also exhibit an inclination flip bifurcation of case~\textbf{C} along them and then come back to $\mathbf{C_I}$ as an homoclinic bifurcation curve of the opposite orientation. As a consequence, the respective saddle-node and period-doubling bifurcation curves extend from $\mathbf{C_I}$ to this additional inclination flip bifurcation point. 

\begin{figure}
\centering
\includegraphics[height=15.2cm]{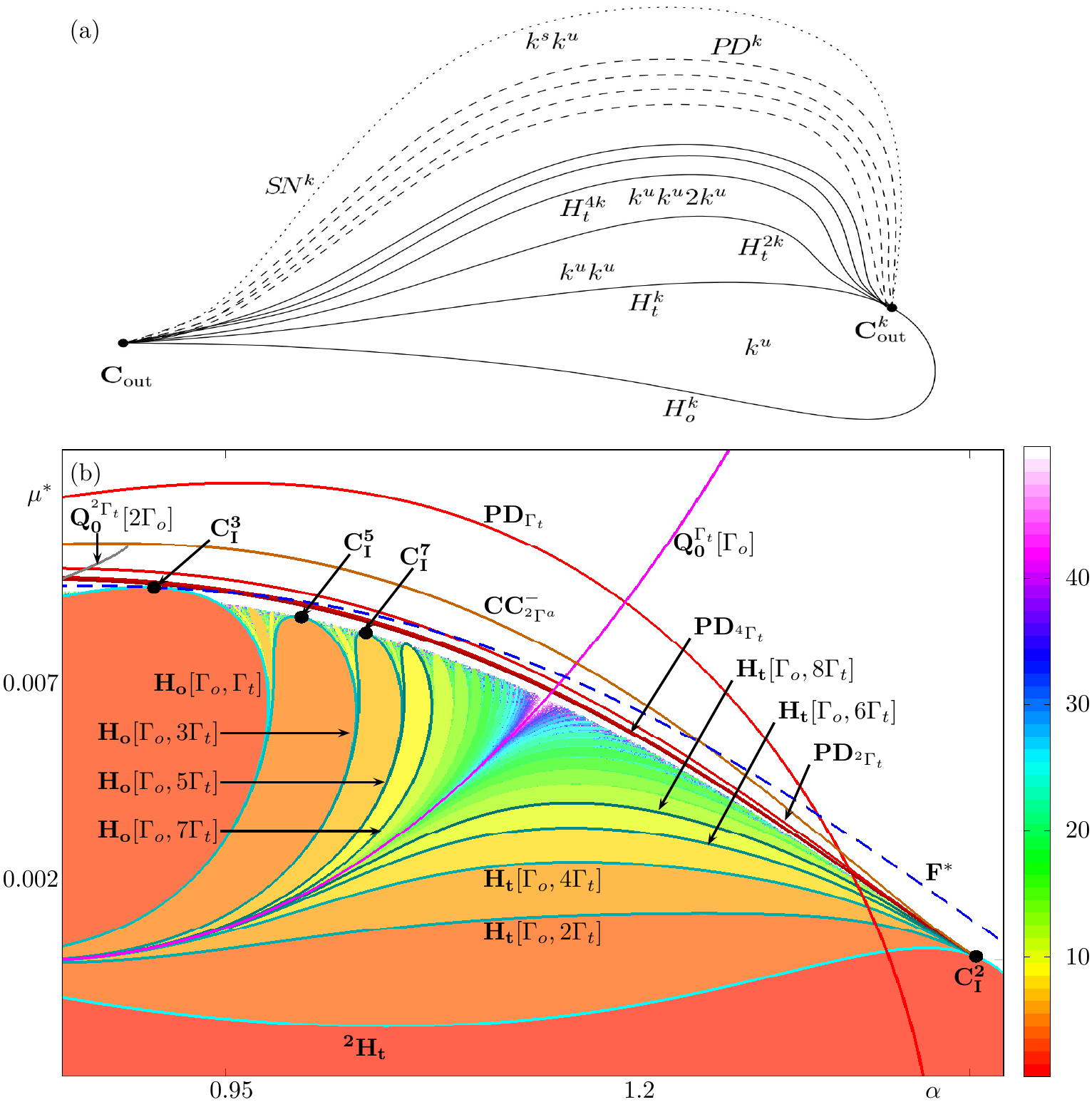}
\caption{Comparison of a theoretical and numerical inclination flip bubble. Panel (a) shows the sketch of an inclination flip bubble from \cite[Fig. 12]{Hom2}; panel (b) shows a region of \cref{fig:Third} in the $(\alpha, \mu^*)$-plane, where $\mu^*:=\mu+0.1246\alpha-0.06644$. Color code of the regions and nomenclature of the curves are as in \cref{fig:Third}. [The inset is reproduced from Journal of Dynamics and Differential Equations, Resonant Homoclinic Flip Bifurcations, 12(4), 2000, pages 807-850, A. J. Homburg and B. Krauskopf, \textcopyright Plenum Publishing Corporation 2000 with permission of Springer.] } 
\label{fig:Bubble}
\end{figure}  

\subsection{Homoclinic bubbles}

This type of overall bifurcation structure generated by a particular homoclinic curve is called a homoclinic bubble \cite{Hom2}. Recall our discussion of the gaps between accumulation clusters along the slice $\mu=0.5$, which we found in \cref{tab:Sequence} and sketched in \cref{fig:SequenceSlice}. \Cref{fig:Third}~(a) shows that the pairs of homoclinic bifurcations $\myHom{o}{m\Gamma_o,\Gamma_t}$ and $\myHom{t}{(m+1)\Gamma_o}$, as found for $m=1,2,3,4,5,6,7,8,9$ in \cref{tab:Sequence} form bubbles in the parameter plane. In particular, identified cascades of these bubbles explain the disappearance of the infinitely many codimension-one curves that emanate from the unfolding of case~\textbf{C} during a transition to case~\textbf{B}. In fact, two types of bubbles were conjectured \cite{Hom2} and subsequently confirmed numerically \cite{OldKra1}: one is characterized by a homoclinic-doubling cascade and the other by an additional homoclinic flip bifurcation of case~\textbf{C}. \Cref{fig:Bubble} (a) shows the theoretical sketch of a single homoclinic $k$-bubble, for $k \in \N$, of the latter type, reproduced from \cite{Hom2}.  Notice how an orientable $k$-homoclinic bifurcation curve $H_o^k$ emanates from the principal homoclinic flip bifurcation point $\mathbf{C_{out}}$ and then exhibits an inclination flip bifurcation $\mathbf{C}^k_{\mathbf{out}}$, after which it is the nonorientable $k$-homoclinic bifurcation curve $H_t^k$. In this way, the homoclinic and period-doubling bifurcation curves of the form $2^nk$, for $n \in \N$, emanating from $\mathbf{C_{out}}$ disappear in the unfolding of $\mathbf{C}^k_{\mathbf{out}}$. For any $k \in \N$, there exists $k$-bubbles that annihilate all the corresponding $2^nk$ curves of $\mathbf{C_{out}}$. Those are exactly the bifurcation structures we observe in the bifurcation diagram of \cref{fig:Third} in the transition between the inclination flip points $\mathbf{C_I}$ and $\mathbf{B^*_I}$.

\Cref{fig:Bubble}(b) shows a different enlargement of the bifurcation diagram of $\mathbf{C_I}$. Here, we use the coordinate transformation $\mu^*:=\mu+0.1246\alpha-0.06644$ and plot the bifurcation diagram in the $(\alpha, \mu^*)$-plane to improve visualization and comparison with the sketch in panel~(a). The homoclinic curve $\mathbf{^2H_t}$ exhibits the inclination flip bifurcation $\mathbf{C_I^2}$, which allows $\mathbf{^2H_t}$ to become $\mathbf{^2H_o}$ so it can disappear at $\mathbf{B^*_I}$. This is one of the most important curves in the transition between the two cases \cite{Hom2}. However, it does not form a bubble as presented in \cite{Hom2} because both homoclinic curves do not emanate from the same point $\mathbf{C_I}$. Nevertheless, the inclination flip point $\mathbf{C_I^2}$ plays the same role as the homoclinic flip point $\mathbf{C}^k_{\mathbf{out}}$ of a $k$-bubble for $k=2$ in terms of absorbing/generating the respective bifurcation curves.  Notice in \cref{fig:Third}(b) how the cascade of nonorientable homoclinic bifurcation curves (cyan) of the form $\myHom{t}{\Gamma_o,n\Gamma_t}$, for $n=2,4,6,8$, and the period-doubling bifurcation curves $\mathbf{PD}_{^2\Gamma_t}$ and $\mathbf{PD}_{^4\Gamma_t}$ (red) connect the points $\mathbf{C_I}$ and $\mathbf{C^2_I}$ as sketched in panel~(a).  

The same phenomenon occurs at the points $\mathbf{C^3_I}$, $\mathbf{C^5_I}$, and $\mathbf{C^7_I}$, but the curves are much closer together. \Cref{fig:Bubble}(b) clearly shows the complicated structure of how some homoclinic bifurcation curves create bubbles and others connect to inclination flip points on bubbles. It also shows how infinitely many homoclinic bubbles accumulate on $\myEtoP{\mathbf{0}}{\Gamma_t}{\Gamma_o}$, which is illustrated by the increase of $\zeta$. Furthermore, notice how the $\zeta$-value also increases in the different regions in between bubbles, such as those of $\mathbf{C^3_I}$, $\mathbf{C^5_I}$ and $\mathbf{C^7_I}$, each of which is associated with a codimension-one heteroclinic bifurcation of the type $\mathbf{Q}_{\mathbf{0}}^{\Gamma^*}$ for a suitable saddle periodic orbit $\Gamma^*$; see \cref{sec:Cascades}.  The accumulation of these homoclinic bifurcation curves and bubbles onto a heteroclinic bifurcation of the form $\mathbf{Q}_{\mathbf{0}}^{\Gamma_t}$ was not identified in either in \cite{Hom2} or \cite{OldKra1}.  Another interesting detail about the bubbles in \Cref{fig:Third}(b) is how they seem to be bounded by a single smooth curve, where the inclination flip points $\mathbf{C^3_I}$, $\mathbf{C^5_I}$ and $\mathbf{C^7_I}$ occur along it. At first glance, it appears that this curve is $\mathbf{F^*}$, that is, the fold curve of a tangency between $W^u(\mathbf{q})$ and $W^s(\mathbf{0})$, as shown in panel~(a) and (b) in \cref{fig:Third}. However, \cref{fig:Bubble}~(b) indicates that this is not the case. This boundary curve represents a boundary crisis \cite{Hin2}, and infinitely many inclination flip bifurcations of case~\textbf{C} occur along it.  We believe that this curve of boundary crisis involves the tangency of an, as yet, unknown invariant object.

Finally, we remark that the eigenvalues of $\mathbf{0}$ do not change as $\alpha$ and $\mu$ vary in the bifurcation diagrams shown in \cref{fig:Third} and \cref{fig:Bubble}; see \cref{sec:san}. This raises the question of the existence of $\mathbf{B^*_I}$ in \cref{fig:Third}, because it fullfils the eigenvalue conditions of case~\textbf{C}. Preliminary work suggest that it also fullfils the necessary geometry condition for case~\textbf{C}. On the other hand, our bifurcation analysis clearly identifies the point $\mathbf{B^*_I}$ as an inclination flip bifurcation of case~\textbf{B}.  Further analysis of this codimension-two point is beyond the scope of this paper, and left for future work.

\section{Discussion} 
\label{sec:Dis}

We conducted a detailed case study of a codimension-two inclination flip bifurcation point $\mathbf{C_I}$ of case \textbf{C} in Sandstede's model~\cref{eq:san}. More specifically, we presented the bifurcation set in the $(\alpha,\mu)$-parameter plane and illustrated the associated dynamics at representative parameter points in regions of distinct qualitative behavior. This required the computation of relevant global invariant manifolds of saddle objects in phase space, as well as many curves of homoclinic, heteroclinic, saddle-node and period-doubling bifurcations. Moreover, we calculated a winding number and identified the regions where it is constant.

Near the central point $\mathbf{C_I}$ in the $(\alpha,\mu)$-plane we found a very complicated structure of bifurcations. As predicted by what is known about the theoretical unfolding \cite{Hom1}, we identified cascades of local and global bifurcations, as well as regions with Smale--horseshoe dynamics and strange attractors. We clarified the precise arrangement of these different ingredients and, moreover, identified a number of new bifurcations near $\mathbf{C_I}$. Specifically, we found that cascades of homoclinic bifurcations accumulate on codimension-one heteroclinic bifurcations between a nonorientable saddle periodic orbit $\Gamma_t$, and the central equilibrium $\mathbf{0}$; different types of cascades can be distinguished by the number of rotations that the global orbit makes near the saddle periodic orbit. Moreover, we identified the boundaries of the Smale--horseshoe region: the codimension-one heteroclinic bifurcation $\mathbf{Q}_{\mathbf{0}}^{\Gamma_o}$, where there is a connecting orbit from $\mathbf{0}$ to the orientable saddle periodic orbit $\Gamma_o$; and the codimension-one homoclinic bifurcation $\mathbf{Tan}_{\Gamma_o}$ of $\Gamma_o$ (which was conjectured to exist in \cite{Nau1}).

We proceeded by considering the bifurcation set in the $(\mu,\alpha)$-plane more globally, further away from the central point $\mathbf{C_I}$, to determine the overall organization of the emanating curves of codimension-one bifurcations. First of all, we found that the region with Smale--horseshoe dynamics is, in fact, bounded, because the curves $\mathbf{Q}_{\mathbf{0}}^{\Gamma_o}$ and $\mathbf{Tan}_{\Gamma_o}$ intersect to form a codimension-two heteroclinic cycle between $\mathbf{0}$ and the orientable saddle periodic orbit $\Gamma_o$. Zooming out even more,  revealed a complex overall picture involving the transition to an inclination flip bifurcation of case~\textbf{B}. A prominent feature of it are bubbles formed by certain homoclinic bifurcation curves in parameter plane, which exhibit an additional homoclinic flip bifurcation of case~\textbf{C} that changes the orientation of the respective homoclinic bifurcation. Such bubbles were proposed as a crucial ingredient in the transition between cases~\textbf{B} and \textbf{C} in a codimension-three resonant homoclinic flip bifurcation~\cite{Hom2}, and then found numerically in \cite{OldKra1}.  We identified infinitely many homoclinic bubbles in the $(\mu,\alpha)$-plane, clarified the role of heteroclinic bifurcations for their organization and described their accumulation on a specific boundary curve.

The detailed numerical investigation presented here identified several generic bifurcation phenomena that were not considered before, as part of the unfolding of a homoclinic flip bifurcation of case~\textbf{C} or otherwise. Now that they are known and described, they can be studied theoretically by considering them in a more abstract and general setting. Hence, an advanced numerical study as presented here can contribute to a better understanding of theoretical constructs, especially in situations when the theory is very intricate. In this context it will be interesting to study, in the same spirit, the configuration of manifolds responsible for the inward twist case $\mathbf{C_{in}}$, which features a related but different unfolding~\cite{Hom1}. However, at present, we are not aware of a concrete example of this codimension-two phenomenon in a three-dimensional vector field.

More generally, the results presented here can be seen as a showcase of the capabilities of advanced numerical methods based on two-point boundary value problems for the bifurcation analysis of a given system with complicated global bifurcations. In particular, homoclinic flip bifurcations have been found, for example, in the Hindmarsh--Rose model of a spiking neuron, where they explain the creation of large spiking excursions of periodic orbits in the presence of slow-fast dynamics~\cite{Lina1}. Further study of the role of global bifurcations in systems with multiple time scales is a promising direction for future research from both a theory and an applications point of view; for a recent example see \cite{Mujica1}.


\bibliographystyle{siam}

\begin{thebibliography}{10}

\bibitem{Agu1}
{\sc P.~Aguirre, B.~Krauskopf, and H.~M. Osinga}, {\em Global invariant
  manifolds near homoclinic orbits to a real saddle: ({Non})orientability and
  flip bifurcation}, SIAM J. Appl. Dyn. Syst., 12 (2013), pp.~1803--1846.

\bibitem{asaoka1997}
{\sc M.~Asaoka}, {\em A natural horseshoe-breaking family which has a period
  doubling bifurcation as the first bifurcation}, J. Math. Kyoto Univ., 37
  (1997), pp.~493--511.

\bibitem{Barrio1}
{\sc R.~Barrio, M.~Lefranc, M.~A. Mart\'inez, and S.~Serrano}, {\em Symbolic
  dynamical unfolding of spike-adding bifurcations in chaotic neuron models},
  {EPL (Europhysics Letters)}, 109 (2015).

\bibitem{Shil4}
{\sc R.~Barrio and A.~Shilnikov}, {\em Parameter-sweeping techniques for
  temporal dynamics of neuronal systems: case study of hindmarsh-rose model},
  J. Math. Neurosci., 1 (2011).

\bibitem{Shil3}
{\sc R.~Barrio, A.~Shilnikov, and L.~Shilnikov}, {\em Kneadings, symbolic
  dynamics and painting {L}orenz chaos}, Internat. J. Bifur. Chaos Appl. Sci.
  Engrg., 22 (2012).

\bibitem{Call1}
{\sc R.~C. Calleja, E.~J. Doedel, A.~R. Humphries, A.~Lemus-Rodriguez, and
  E.~B. Oldeman}, {\em Boundary-value problem formulations for computing
  invariant manifolds and connecting orbits in the circular restricted three
  body problem}, Celest. Mech. Dyn. Astron., 114 (2012), pp.~77--106.

\bibitem{Kirk1}
{\sc A.~R. Champneys, V.~Kirk, E.~Knobloch, B.~E. Oldeman, and J.~D.~M.
  Rademacher}, {\em Unfolding a tangent equilibrium-to-periodic heteroclinic
  cycle}, SIAM J. Appl. Dyn. Syst., 8 (2009), pp.~1261--1304.

\bibitem{san2}
{\sc A.~R. Champneys, Y.~Kuznetsov, and B.~Sandstede}, {\em A numerical toolbox
  for homoclinic bifurcation analysis}, Internat. J. Bifur. Chaos Appl. Sci.
  Engrg., 6 (1996), pp.~867--887.

\bibitem{Deng1}
{\sc B.~Deng}, {\em Homoclinic twisting bifurcations and cusp horseshoe maps},
  J. Dynam. Differential Equations, 5 (1993), pp.~417--467.

\bibitem{Doe1}
{\sc E.~J. Doedel}, {\em {AUTO}: A program for the automatic bifurcation
  analysis of autonomous systems}, Congr. Numer., 30 (1981), pp.~265--284.

\bibitem{Doe2}
{\sc E.~J. Doedel and B.~E. Oldeman}, {\em {AUTO}-\textup{07}p: Continuation
  and Bifurcation Software for Ordinary Differential Equations}, Department of
  Computer Science, Concordia University, Montreal, Canada, 2010.
\newblock With major contributions from A. R. Champneys, F. Dercole, T. F.
  Fairgrieve, Y. Kuznetsov, R. C. Paffenroth, B. Sandstede, X. J. Wang and C.
  H. Zhang; available at \url{http://www.cmvl.cs.concordia.ca/}.

\bibitem{Ali1}
{\sc R.~Gilmore and M.~Lefranc}, {\em The Topology of Chaos: Alice in Stretch
  and Squeezeland}, Wiley-Interscience, 2002.

\bibitem{And1}
{\sc A.~Giraldo, B.~Krauskopf, and H.~M. Osinga}, {\em Saddle invariant objects
  and their global manifolds in a neighborhood of a homoclinic flip bifurcation
  of case {B}}, SIAM J. Appl. Dyn. Syst., 16 (2017), pp.~640--686.

\bibitem{Guck1}
{\sc J.~Guckenheimer and P.~Holmes}, {\em Nonlinear Oscillations, Dynamical
  Systems, and Bifurcations of Vector Fields}, {Springer-Verlag, New York},
  1983.

\bibitem{Hom1}
{\sc A.~J. Homburg, H.~Kokubu, and M.~Krupa}, {\em The cusp horseshoe and its
  bifurcations in the unfolding of an inclination-flip homoclinic orbit},
  Ergodic Theory Dynam. Systems, 14 (1994), pp.~667--693.

\bibitem{Hom2}
{\sc A.~J. Homburg and B.~Krauskopf}, {\em Resonant homoclinic flip
  bifurcations}, J. Dynam. Differential Equations, 12 (2000), pp.~807--850.

\bibitem{san3}
{\sc A.~J. Homburg and B.~Sandstede}, {\em Homoclinic and heteroclinic
  bifurcations in vector fields}, in Handbook of Dynamical Systems, H.~W.
  Broer, B.~Hasselblatt, and F.~Takens, eds., vol.~3, {Elsevier, New York},
  2010, pp.~381--509.

\bibitem{kis1}
{\sc M.~Kisaka, H.~Kokubu, and H.~Oka}, {\em Bifurcations to n-homoclinic
  orbits and n-periodic orbits in vector fields}, J. Dynam. Differential
  Equations, 5 (1993), pp.~305--357.

\bibitem{Kopper1}
{\sc M.~Koper}, {\em Bifurcations of mixed-mode oscillations in a
  three-variable autonomous {V}an der {P}ol-{D}uffing model with a cross-shaped
  phase diagram}, Phys. D, 80 (1995), pp.~72--94.

\bibitem{Kra2}
{\sc B.~Krauskopf and H.~M. Osinga}, {\em Computing invariant manifolds via the
  continuation of orbit segments}, in Numerical Continuation Methods for
  Dynamical Systems: {Path Following and Boundary Value Problems},
  B.~Krauskopf, H.~M. Osinga, and J.~Gal\'an-Vioque, eds., {Springer, The
  Netherlands}, 2007, pp.~117--154.

\bibitem{KraRie1}
{\sc B.~Krauskopf and T.~{Rie{\ss}}}, {\em A {L}in's method approach to finding
  and continuing heteroclinic connections involving periodic orbits},
  Nonlinearity, 21 (2008), pp.~1655--1690.

\bibitem{Kuz2}
{\sc Y.~A. Kuznetsov, O.~Feo, and S.~Rinaldi}, {\em Belyakov homoclinic
  bifurcations in a tritrophic food chain model}, SIAM J. Appl. Dyn. Syst., 62
  (2001), p.~462–487.

\bibitem{Lina1}
{\sc D.~Linaro, A.~Champneys, M.~Desroches, and M.~Storace}, {\em
  Codimension-two homoclinic bifurcations underlying spike adding in the
  {Hindmarsh-Rose} burster}, SIAM J. Appl. Dyn. Syst., 11 (2012), pp.~939--962.

\bibitem{Lohr1}
{\sc A.~Lohse and A.~Rodrigues}, {\em Boundary crisis for degenerate singular
  cycles}, Nonlinearity, 30 (2017), pp.~2211--2245.

\bibitem{Mujica1}
{\sc J.~Mujica, B.~Krauskopf, and H.~M. Osinga}, {\em Tangencies between global
  invariant manifolds and slow manifolds near a singular hopf bifurcation},
  SIAM J. Appl. Dyn. Syst., 17 (2018), pp.~1395--1431.

\bibitem{Nau1}
{\sc V.~Naudot}, {\em Strange attractor in the unfolding of an inclination-flip
  homoclinic orbit}, Ergodic Theory Dynam. Systems, 16 (1996), pp.~1071--1086.

\bibitem{Nau2}
\leavevmode\vrule height 2pt depth -1.6pt width 23pt, {\em A strange attractor
  in the unfolding of an orbit-flip homoclinic orbit}, Dyn. Syst., 17 (2002),
  pp.~45--63.

\bibitem{OldKra1}
{\sc B.~E. Oldeman, B.~Krauskopf, and A.~R. Champneys}, {\em Numerical
  unfoldings of codimension-three resonant homoclinic flip bifurcations},
  Nonlinearity, 14 (2001), pp.~597--621.

\bibitem{Hin1}
{\sc H.~M. Osinga}, {\em Nonorientable manifolds in three-dimensional vector
  fields}, Internat. J. Bifur. Chaos Appl. Sci. Engrg., 13 (2003),
  pp.~553--570.

\bibitem{Hin2}
\leavevmode\vrule height 2pt depth -1.6pt width 23pt, {\em Locus of boundary
  crisis: Expect infinitely many gaps}, Phys. Rev. E, 74 (2006).

\bibitem{Palis1}
{\sc J.~Palis and W.~de~Melo}, {\em Geometric Theory of Dynamical Systems},
  {Springer, New York}, 1982.

\bibitem{Taken1}
{\sc J.~Palis and F.~Takens}, {\em Hyperbolicity and Sensitive Chaotic Dynamics
  at Homoclinic Bifurcations: Fractal Dimensions and Infinitely Many Attractors
  in Dynamics}, Cambridge University Press, 1995.

\bibitem{san4}
{\sc B.~Sandstede}, {\em Verzweigungstheorie Homokliner Verdopplungen}, PhD
  thesis, University of Stuttgart, {Stuttgart, Germany}, 1993.

\bibitem{san1}
\leavevmode\vrule height 2pt depth -1.6pt width 23pt, {\em Constructing
  dynamical systems having homoclinic bifurcation points of codimension two},
  J. Dynam. Differential Equations, 9 (1997), pp.~269--288.

\bibitem{Shil5}
{\sc L.~P. Shilnikov}, {\em A case of the existence of a denumerable set of
  periodic motions}, Sov. Math. Dokl., 6 (1965), pp.~163--166.

\bibitem{Shil2}
\leavevmode\vrule height 2pt depth -1.6pt width 23pt, {\em On the generation of
  a periodic motion from trajectories doubly asymptotic to an equilibrium state
  of saddle type}, Mat. Sb. (N.S.), 77(119) (1968), pp.~461--472.

\bibitem{Leonid1}
{\sc L.~P. Shilnikov, A.~L. Shilnikov, D.~V. Turaev, and L.~O. Chua}, {\em
  Methods of Qualitative Theory in Nonlinear Dynamics (Part \textup{I})},
  vol.~4, World Scientific, Singapore, 1998.

\bibitem{Doe5}
{\sc B.~W., A.~Champneys, E.~Doedel, W.~Govaerts, Y.~Kuznetsov, and
  B.~Sandstede}, {\em Numerical continuation, and computation of normal forms},
  in Handbook of Dynamical Systems, 1999.

\bibitem{Wie1}
{\sc S.~M. Wieczorek}, {\em Global bifurcation analysis in laser systems}, in
  Numerical Continuation Methods for Dynamical Systems: {Path Following and
  Boundary Value Problems}, B.~Krauskopf, H.~M. Osinga, and J.~Gal\'an-Vioque,
  eds., {Springer, The Netherlands}, 2007, pp.~177--220.

\bibitem{Wie2}
{\sc S.~M. Wieczorek and B.~Krauskopf}, {\em Bifurcations of n-homoclinic
  orbits in optically injected lasers}, Nonlinearity, 18 (2005),
  pp.~{1095--1120}.

\bibitem{Wigg1}
{\sc S.~Wiggins}, {\em Introduction to Applied Nonlinear Dynamical Systems and
  Chaos}, {Springer-Verlag, New York}, 2nd~ed., 2003.

\bibitem{Shil1}
{\sc T.~Xing, R.~Barrio, and A.~Shilnikov}, {\em Symbolic quest into homoclinic
  chaos}, Internat. J. Bifur. Chaos Appl. Sci. Engrg., 24 (2014).

\bibitem{Yan1}
{\sc E.~Yanagida}, {\em Branching of double pulse solutions from single pulse
  solutions in nerve axon equations}, J. Differential Equations, 66 (1987),
  pp.~243--262.

\bibitem{Kirk2}
{\sc W.~Zhang, B.~Krauskopf, and V.~Kirk}, {\em How to find a codimension-one
  heteroclinic cycle between two periodic orbits}, Discrete Contin. Dyn. Syst.,
  32 (2012), pp.~2825--2851.

\end{thebibliography}


\end{document}